\documentclass[a4paper,12pt,leqno]{amsart}
\usepackage[a4paper, top=2.4cm, bottom=2cm, left=3cm, right=3cm]{geometry}
\usepackage{amsmath,amsfonts,amsthm,amssymb,dsfont}
\usepackage[alphabetic]{amsrefs}
\usepackage[OT4]{fontenc}
\usepackage{enumerate}
\usepackage{mathrsfs}
\usepackage{enumerate, xspace}
\usepackage[pdftex]{graphicx}
\usepackage{color}
\usepackage{hyperref}
\usepackage{xcolor}

\long\def\symbolfootnote[#1]#2{\begingroup
	\def\thefootnote{\fnsymbol{footnote}}\footnote[#1]{#2}\endgroup}

\newtheorem{theorem}{Theorem}[section]
\newtheorem{lem}[theorem]{Lemma}
\newtheorem{thm}[theorem]{Theorem}

\newtheorem{prop}[theorem]{Proposition}
\newtheorem{cor}[theorem]{Corollary}

\theoremstyle{definition}
\newtheorem{rem}[theorem]{Remark}
\newtheorem{defin}[theorem]{Definition}

\newtheorem{constr}[theorem]{Construction}

\newcommand{\type}{\operatorname{Type}}
\newcommand{\cq}{\mathcal{Q}}
\newcommand{\ch}{\mathcal{H}}
\newcommand{\cj}{\mathcal{J}}
\newcommand{\od}{\widehat{\Sigma}}
\newcommand{\Si}{\Sigma}

\newcommand{\bU}{\mathbb U} 
\newcommand{\bV}{\mathbb V} 
\newcommand{\bC}{\mathfrak C}

\newcommand{\wX}{\widehat X}

\newcommand{\ca}{\mathcal {A}}

\newcommand{\fan}{\operatorname{Fan}}
 
\newcommand{\lk}{\operatorname{lk}} 
 
\newcommand{\cH}{\mathcal {H}} 
\newcommand{\cK}{\mathcal {K}} 
\newcommand{\cL}{\mathcal L}

\newcommand{\whC}{\widehat C}

\newcommand{\whX}{\widehat X}
\newcommand{\sX}{\mathsf X}
\newcommand{\sY}{\mathsf Y}
\newcommand{\wsX}{\widehat {\mathsf X}}
\newcommand{\wsY}{\widehat {\mathsf Y}}
\newcommand{\act}{\curvearrowright}
\newcommand{\si}{\sigma}
\newcommand{\ka}{\kappa}
\newcommand{\hk}{\widehat\kappa}

\newcommand{\st}{\operatorname{St}}

\title{353-combinatorial curvature and the $3$-dimensional $K(\pi,1)$ conjecture} 
\author[J.~Huang]{Jingyin Huang$^{\ast}$}

	\address{
		Department of Mathematics,
		The Ohio State University, 
		231 W. 18th Ave,
		Columbus, OH 43210, U.S.              
	}
	
	\email{huang.929@osu.edu}
\author[P.~Przytycki]{Piotr Przytycki$^{\dag}$}
	
	\address{
		Department of Mathematics and Statistics,
		McGill University,
		Burnside Hall,
		805 Sherbrooke Street West,
		Montreal, QC,
		H3A 0B9, Canada}
	
	\email{piotr.przytycki@mcgill.ca}

	\thanks{$\ast$ Partially supported by a Sloan fellowship and NSF grant DMS-2305411}
		
		\thanks{$\dag$ Partially supported by NSERC}

\begin{document}
	\maketitle
	
\begin{abstract}
		\noindent
		We prove the $K(\pi,1)$ conjecture for Artin groups of dimension $3$. As an ingredient, we introduce a new form of combinatorial non-positive curvature.
	\end{abstract}

\section{Introduction}

The $K(\pi,1)$ conjecture for Artin groups, due to Arnold, Brieskorn, Pham, and Thom, predicts that each Artin group has a $K(\pi,1)$ space that is a complex manifold described in the terms of the canonical linear representation of the associated Coxeter group. 
See \cite{paris2014k,godelle2012basic,deligne,charney1995k}, for background and a summary of progress on this conjecture before the 2010s, and \cite{mccammond2017artin,juhasz2018relatively,paolini2021proof,paolini2021dual,delucchi2022dual,juhasz2023class,goldman2022k,haettel2021lattices,haettel2022conjecture,haettel2023new,huang2023labeled,huang2024,goldman2025deligne} for more recent developments. In this article we prove the following.

\begin{thm}
	\label{thm:main intro}
Let $A$ be an Artin group of dimension $\le 3$. Then $A$ satisfies the $K(\pi,1)$ conjecture. In particular, $A$ is torsion free.
\end{thm}

The \emph{dimension} of an Artin group $A$ is the maximal cardinality of a subset $S'$ of the standard generating set of $A$ such that the subgroup of $A$ generated by $S'$ is spherical. It is conjectured that this quantity is equal to the cohomological dimension of $A$, and by Theorem~\ref{thm:main intro} this is true if one of these two quantities is $\le 3$. The dimension $\le 2$ case of Theorem~\ref{thm:main intro} was established in 1995 \cite{CharneyDavis}.

Using \cite{jankiewicz2023k}, we also deduce the centre conjecture.

\begin{thm}
Let $A$ be an Artin group of dimension $\le 3$. If $A$ has no nontrivial spherical factor, then it has trivial centre.
\end{thm}

Theorem~\ref{thm:main intro} is a special case of Theorem~\ref{thm:general}, where we establish 
the $K(\pi,1)$ conjecture for many new Artin groups in each dimension, since the class of Artin groups that we treat contains arbitrarily large  
irreducible spherical parabolic subgroups.

\subsection{Combinatorial non-positive curvature}
\label{subsec:353}
A key ingredient of the proof is a contractibility criterion for a class of complexes satisfying a new form of combinatorial non-positive curvature, called $353$-square complexes.

\begin{thm}
	\label{thm:contractible intro}
	The thickening of a wide stable $353$-square complex is contractible.
\end{thm}

Let us define 353-square complexes and their thickenings (for the notions of \emph{wide} and \emph{stable}, see Section~\ref{sec:353square}).
A \emph{square complex} is a $2$-dimensional combinatorial complex $X$, where $X^1$ is a bipartite simplicial graph, with vertex set partitioned into sets $\mathcal A$ and $\mathcal D$, and with attaching maps of the $2$-cells distinct embedded cycles of length $4$ (a \emph{cycle} in a graph is a closed edge-path, or, shortly, an edge-loop). We often identify a $2$-cell with its attaching map, and we call it a \emph{square}. Not all embedded cycles of length $4$ are assumed to be squares. We refer to Section~\ref{subsec:discdiagram} for the background on (minimal) disc diagrams in $X$.
	
Vertices $a,a'\in \mathcal A$ (or $d,d'\in \mathcal D$) are \emph{close} if they belong to a common square. The \emph{thickening} $X^\boxtimes$ of a square complex $X$ is the flag simplicial complex whose $1$-skeleton is obtained from $X^1$ by adding edges between close vertices. We adopt the convention that if we label a vertex of $X$ by $a,a',a_1$ etc, then it belongs to $\mathcal A$.

\begin{defin}
	\label{def:cubecorner}
	 A \emph{cube corner} $C$ is a square complex isomorphic to the subcomplex of the boundary of a 3-dimensional cube formed of three squares containing a common vertex, called the \emph{centre} of $C$. A \emph{cube corner} in $X$ is a disc diagram $C\to X$. A cube corner in $X$ is \emph{minimal} if its boundary $6$-cycle does not bound a disc diagram in $X$ with $<3$ squares.
     \end{defin}

\begin{defin}
\label{def:353square}
A \emph{$353$-square complex} is a simply connected square complex satisfying the following properties.
	\label{def:axioms}
	\begin{enumerate}[(1)]
		\item \label{def:axioms1}
		If $dad_1a'$ and $dad_2a'$ are squares, then $d_1ad_2a'$ is a square (see Figure~\ref{fig:353def}(1)).
		\item \label{def:axioms5}
		Let $d$ be a vertex of a minimal cube corner $C$ lying in exactly two squares $ada_1d_1,ada_2d_2$ of $C$. Then there is no square $d'a_1da_2$ (see Figure~\ref{fig:353def}(2)).
		\item \label{def:axioms3}
		Let $d$ be a vertex of a minimal cube corner $C$ lying in exactly two squares $ada_1d_1,ada_2d_2$ of $C$. Suppose that there is $a'\neq a$ with $a_1da',a_2da'$ also lying in squares (see Figure~\ref{fig:353def}(3)). Then $a'$ is a neighbour of $d_1$ and $d_2$ and $ada'd_1,ada'd_2$ are squares.
		\item \label{def:axioms4} Let $E,E'$ be as in Figure~\ref{fig:353def}(4a), (4b).  For any disc diagram $f\colon E\to X$ whose restriction to each cube corner of $E$ is minimal, there is a diagram $f'\colon E'\to X$ with the same boundary as $f$, such that $f'(a')$ is a neighbour of~$f(d)$.
	 \item \label{def:axioms0} Previous properties hold if we interchange $\mathcal A$ and $\mathcal D$.
	\end{enumerate}
\end{defin}

\begin{figure}[h]
	\centering
	\includegraphics[scale=0.9]{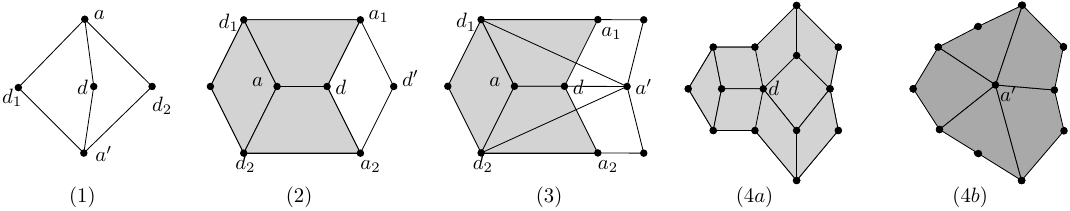}
	\caption{}
	\label{fig:353def}
\end{figure}

Definition~\ref{def:axioms} is motivated by the structure of the \emph{icosahedral honeycomb} of the hyperbolic $3$-space $\mathbb H^3$, with Schläfli symbol $\{3,5,3\}$. It is one of the four compact, regular, space-filling  honeycombs in $\mathbb H^3$, and it was the least understood one from the perspective of combinatorial non-positive curvature (while the other three honeycombs, viewed as cell complexes, are cell-Helly \cite[Def 3.5]{huang2021helly}).

Given the icosahedral honeycomb of $\mathbb H^3$, viewed as a combinatorial complex $Z$, we define the associated square complex $X$. Vertices in $\mathcal A$ correspond to the icosahedra of $Z$, and vertices in $\mathcal D$ correspond to the vertices of $Z$. Vertices $x\in\mathcal A$ and $y\in \mathcal D$ are neighbours if the icosahedron corresponding to $x$ contains the vertex corresponding to $y$. We span squares on all embedded 4-cycles of $X^1$. Definition~\ref{def:axioms} is conceived by listing local combinatorial features of $X$ of non-positive curvature flavour. The list of local properties in Definition~\ref{def:353square} leads to a collection of global properties in Lemma~\ref{lem:diagram}, which is an analogue of the Cartan--Hadamard theorem. These complexes have quadratic Dehn function (Lemma~\ref{lem:diagram}(ii)), and their balls satisfy a weak form of convexity \cite{cannon1987almost} (Lemma~\ref{lem:diagram}(iv))). Finally, we have contractibility as in Theorem~\ref{thm:contractible intro}.

Let $\Lambda$ be the Coxeter diagram that is the linear graph formed of three edges with consecutive labels $3,5,3$, and let $\Delta$ be the  Artin complex (see Definition~\ref{def:Artin0}) of $A_\Lambda$, which is a 3-dimensional simplicial complex. A major step towards Theorem~\ref{thm:main intro} is showing that $\Delta$ is contractible. This is done via Theorem~\ref{thm:contractible intro}, see Definition~\ref{def:353simplicial1} and Corollary~\ref{cor:contr}.

In \cite{CharneyDavis}, Charney and Davis proposed to equip $\Delta$ with a piecewise Euclidean metric (the Moussong metric) and to show that $\Delta$ is $\mathrm{CAT}(0)$, hence contractible. Proving $\mathrm{CAT}(0)$ amounts to studying loops of length $<2\pi$ in the links of the vertices of~$\Delta$. Each such loop gives an equation of the form $w_1w_2\cdots w_n=1$ in the Artin group~$A_{H_3}$ of type $H_3$, subject to the constraint that each $w_i$ lies in an appropriate parabolic subgroup of $A_{H_3}$. Thus proving that $\Delta$ is $\mathrm{CAT}(0)$ relies on understanding the `varieties' of solutions to a finite (but large) set of such equations over $A_{H_3}$. There are no established theories in algebraic geometry to understand the solution varieties of such equations, and the ambient group $A_{H_3}$ being exceptional further obscures the picture. This is the main difficulty of the $\mathrm{CAT}(0)$ approach.

The $\mathrm{CAT}(0)$ approach inspired us to look for a `softer' notion of non-positive curvature, leading to Definition~\ref{def:axioms} and its simplicial companion, Definition~\ref{def:353simplicial1}. Like in the $\mathrm{CAT}(0)$ approach, proving the contractility of $\Delta$ reduces to studying a collection of short loops in the link of each vertex. However, it is a much smaller collection of loops, hence the number of the associated equations over $G$ that we need to solve is significantly reduced. Miraculously, we avoid solving some of the most sophisticated equations needed in the $\mathrm{CAT}(0)$ approach. However, we do not completely avoid the task of analysing the solution varieties of some of these equations, which takes a substantial portion of the article.

\subsection{Reading guide} 
Section~\ref{sec:prelim} consists of preliminaries.
The article is divided into Part I ranging from Section~\ref{sec:prelim1}
to Section~\ref{sec:353simplicial}, and Part II ranging from Section~\ref{sec:relative} to Section~\ref{sec:3dim}.  Part I takes up most of the article, and it concerns the Artin complex of a \textbf{single} Artin group.

\begin{defin}
\label{def:Artin0}
Let $A_\Lambda$ be an Artin group with Coxeter graph $\Lambda$ and generating set $S$. Its \emph{Artin complex} $\Delta_\Lambda$ \cite{CharneyDavis,godelle2012k,cumplido2020parabolic} is a simplicial complex defined as follows. For each $s\in S$, let $A_{\hat s}$ be the standard parabolic subgroup generated by $\hat s=S\setminus\{s\}$. The vertices of $\Delta_\Lambda$ correspond to the left cosets of $\{A_{\hat s}\}_{s\in S}$. Moreover, vertices span a simplex if the corresponding cosets have non-empty common intersection. A vertex of $\Delta_\Lambda$ corresponding to a left coset of $A_{\hat s}$ has \emph{type $\hat s$}.
\end{defin}

Let $\Lambda$ be the 353 Coxeter diagram from the previous subsection. The main goal of Part I is to establish two properties of $\Delta_\Lambda$. First, $\Delta_\Lambda$ is contractible, which implies the $K(\pi,1)$ conjecture for $A_\Lambda$. Second, each induced embedded 4-cycle in $\Delta_\Lambda$ of type $\hat s\hat t\hat s\hat t$  
has a common neighbour in $\Delta_\Lambda$ of type $\hat r$, where $r$ is a vertex of $\Lambda$ separating $s$ and $t$. This second property is useful for proving the $K(\pi,1)$ conjecture for other Artin groups.

Part I is performed in two steps. In Step 1, we show that the link of each vertex of~$\Delta_{\Lambda}$ satisfies a list of properties (Sections~\ref{sec:prelim1}-\ref{sec:typeII}).  In Step 2, we introduce a more general family of complexes, called \emph{$353$-simplicial complexes},  that are simply connected and whose vertex links satisfy the same list of properties (Definition~\ref{def:353simplicial1}). We prove, under two minor assumptions, that such a complex 
is contractible and has the desired 4-cycle property mentioned in the previous paragraph (Sections~\ref{sec:353square} and~\ref{sec:353simplicial}).

Let us discuss Step 2 in more detail.
In Section~\ref{sec:353square}, we study 353-square complexes, establishing properties of minimal disc diagrams bounded by
certain cycles in these complexes and proving the contractility  of their thickenings (Theorem~\ref{thm:contractible intro}). In Section~\ref{sec:353simplicial}, we introduce the notion of a 353-simplicial complex. For each 353-simplicial complex $\Delta$, we can construct an associated 353-square complex $X$ whose thickening is homotopy equivalent to $\Delta$ (under mild assumptions), implying the contractibility of $\Delta$ and the desired 4-cycle property.

Coming back to Step 1, it remains to show that $\Delta_\Lambda$ is a 353-simplicial complex. 
By 
definition, this reduces to  proving that two kinds of \emph{critical cycles} in $\Delta_{H_3}$ are \emph{admissible}. Here $\Delta_{H_3}$ denotes the Artin complex of the spherical Artin group $A_{H_3}$, which is isomorphic to the vertex link of $\Delta_\Lambda$. Critical cycles in $\Delta_{H_3}$ and the notion of their \emph{admissibility} are introduced at the beginning of Sections~\ref{sec:typeI} and~\ref{sec:typeII}, and most of Sections~\ref{sec:prelim1}-\ref{sec:typeII} is the proof of the admissibility of critical cycles.

In Section~\ref{sec:prelim1}, we give the background on hyperplane arrangements and associated complexes needed later.
For each 
collection $\ca$ of affine hyperplanes in $\mathbb R^n$ passing through the origin, we consider the complex manifold $M(\ca\otimes \mathbb C^n)=\mathbb C^n - \bigcup_{H\in\ca} (H\otimes \mathbb C)$.  Let $\ca$ be the collection of reflection hyperplanes for the canonical linear representation of the Coxeter group of type $H_3$ acting on $\mathbb R^3$. Then $\pi_1M(\ca\otimes \mathbb C^n)$ is the pure Artin group $PA_{H_3}$ of $A_{H_3}$. 
It is difficult to analyse a cycle $\omega$ in $\Delta_{H_3}$ directly. Instead, we consider a subset $\ca'\subset \ca$, which gives an inclusion $M(\ca\otimes \mathbb C^n)\to M(\ca'\otimes \mathbb C^n)$. This induces a quotient map between groups $PA_{H_3}\to \pi_1 M(\ca'\otimes \mathbb C^n)$, and a surjective simplicial map from $\Delta_{H_3}$ to another complex $\Delta_{\ca'}$ (we use this notation only in the Introduction). The cycle $\omega\subset \Delta_{H_3}$ is sent to 
a cycle $\omega'\subset \Delta_{\ca'}$. It turns out that for suitable choices of $\ca'$, the complex $\Delta_{\ca'}$ contains large subcomplexes that are `non-positively curved'. If the subcomplex is large enough to contain~$\omega'$, then we can use the non-positive curvature to analyse $\omega'$, and then lift the information back to $\omega$. This last step is nontrivial, since we are losing information in the quotient map $PA_{H_3}\to \pi_1 M(\ca'\otimes \mathbb C^n)$.

In Section~\ref{sec:subarrangement}, we discuss 
the possible subset $\ca'$. Actually, we choose two subsets $\ca_1$ and $\ca_2$, so certain information that is lost 
as a consequence of
one choice survives for the other choice. For each $\ca_i$, we indicate what is the non-positively curved subcomplex of $\Delta_{\ca_i}$ that we have found. This section is mostly a review of \cite{huang2024}.

The material in Section~\ref{sec:splitting system} is new. The non-positively curved subcomplex of $\Delta_{\ca_2}$ found in \cite{huang2024} is not large enough for our purpose. We show in Section~\ref{sec:splitting system} that there is a larger subcomplex of $\Delta_{\ca_2}$ that satisfies a new form of non-positive curvature, governed by what we have called a \emph{splitting system} (Definition~\ref{def:splitting}). We use it to understand  
minimal disc diagrams in the subcomplex.
Given these ingredients, we treat  
critical $8$-cycles in Section~\ref{sec:typeI} and 
critical $10$-cycles
in Section~\ref{sec:typeII}.

We prove Theorem~\ref{thm:main intro} in Part II of the paper (Sections~\ref{sec:relative}-\ref{sec:3dim}). Our point of departure is the following criterion by Godelle and Paris.

\begin{thm}[{\cite[Thm 2.2]{huang2023labeled}, which is a reformulation of \cite[Thm 3.1]{godelle2012k}}]
 \label{thm:Kpi1} 
 Let $\Lambda$ be non-spherical with $\Delta_\Lambda$ contractible. If $A_{\Lambda'}$ satisfies the $K(\pi,1)$ conjecture for all subdiagrams $\Lambda'$ induced on all but one vertex of $\Lambda$, then $A_\Lambda$ satisfies the $K(\pi,1)$ conjecture.
\end{thm}

To show the contractibility of $\Delta_\Lambda$, we adopt the strategy from \cite{huang2023labeled}.  Roughly speaking, we first show that $\Delta_\Lambda$ deformation retracts to a suitable subcomplex, which is the \emph{relative Artin complex} (Definition~\ref{def:rel}). We then show that this subcomplex is non-positively curved in an appropriate sense, and so it is contractible. 
Section~\ref{sec:relative} summarises the properties of relative Artin complexes from \cite{huang2023labeled} and some additional contracitibility criteria from \cite{bessis2006garside,haettel2021lattices}.

In Section~\ref{sec:convex}, we introduce a geometric tool needed for executing our strategy, the notion of convexity for a class of simplicial complexes that are closely related to Garside categories \cite{Bestvina1999,charney2004bestvina,bessis2006garside,haettel2022link}. A convex subcomplex, as defined here, can be detected locally—specifically, by examining the links of vertices—using a purely combinatorial criterion (see Definition~\ref{def:Bconvex}). This notion is inspired by the Bestvina normal form \cite{Bestvina1999}. Intriguingly, even for the tessellation of $\mathbb{E}^2$ by equilateral triangles, our notion differs from the more classical ones.

In Section~\ref{sec:3dim}, we prove Theorem~\ref{thm:main intro} by induction on the number of generators of the Artin group. In the process, we obtain the following byproducts or enhancements of Theorem~\ref{thm:main intro}, 
some of which (items 3 and 4) might have applications outside the $K(\pi,1)$ conjecture.
\begin{enumerate}
    \item We show that the $K(\pi,1)$ conjecture holds not only for all the $3$-dimensional Artin groups,  but also for many higher dimensional ones (Theorem~\ref{thm:general}).
    \item We derive a general result that reduces the $K(\pi,1)$ conjecture for arbitrary Artin groups to properties of Artin groups whose Coxeter diagrams do not contain triangles (Corollary~\ref{cor:reduction}).
    \item We show that the `girth condition' holds for each $3$-dimensional Artin group (Theorem~\ref{thm:3d}) 
    \item For each 3-dimensional Artin group that is not spherical, we construct a `non-positively curved' relative Artin complex on which the group acts.
\end{enumerate}

\section{Preliminaries}
\label{sec:prelim}
\subsection{Artin complex}
\label{subsec:Artin complex}
 A \emph{Coxeter diagram} $\Lambda$ is a finite simplicial graph with vertex set $S=\{s_i\}_i$ and labels $m_{ij}=3,4,\ldots ,\infty$ for each edge $s_is_j$. If $s_is_j$ is not an edge, we define $m_{ij}=2$. The Artin group $A_\Lambda$ is the group with generator set $S$ and relations $s_is_js_i\cdots=s_js_is_j\cdots$ with both sides alternating words of length $m_{ij}$, whenever $m_{ij}<\infty$. The Coxeter group $W_\Lambda$ is obtained from $A_\Lambda$ by adding relations $s_i^2=1$. 

The \emph{pure Artin group} $PA_\Lambda$ is the kernel of the obvious homomorphism $A_\Lambda\to W_\Lambda$.
We say that $A_\Lambda$ (or $\Lambda$) is \emph{spherical}, if $W_\Lambda$ is finite. Recall that any $S'\subset S$ generates a subgroup of $A_\Lambda$ isomorphic to $A_{\Lambda'}$, where $\Lambda'$ is the subdiagram of $\Lambda$ induced on~$S'$. Such a subgroup is called a \emph{standard parabolic subgroup}.

We refer to Definition~\ref{def:Artin0} for the notion of the Artin complex $\Delta_\Lambda$ of the Artin group $A_\Lambda$.
It follows from \cite[Prop~4.5]{godelle2012k} that $\Delta_\Lambda$ is a flag complex. Note that given $g\in A_\Lambda$, the vertices corresponding to the collection of the left cosets $\{gA_{\hat s}\}_{s\in S}$ span a top-dimensional simplex of $\Delta_\Lambda$. This gives a bijective correspondence between the elements of $A_\Lambda$ and the top-dimensional simplices of $\Delta_\Lambda$.
The \emph{Coxeter complex}~$\bC_\Lambda$ is defined analogously, where we replace~$A_{\hat s}$ by $W_{\hat s}<W_\Lambda$
generated by~$\hat s$. A vertex of $\bC_\Lambda$ corresponding to a left coset of $W_{\hat s}$ has \emph{type} $\hat s$. 
We have that $\bC_\Lambda$ is the quotient of $\Delta_\Lambda$ under the action of~$PA_\Lambda$. 

\begin{rem}[{\cite[Cor 6.5]{huang2023labeled}}]
	\label{rem:adj}
	For $i=1,2,3$, let $x_i\in \Delta^0_\Lambda$ be of type $\hat s_i$. 
	Suppose that $s_1$ and $s_3$ belong to distinct components of $\Lambda\setminus\{s_2\}$.
	If $x_2$ is a neighbour of both $x_1$ and $x_3$, then $x_1$ is a neighbour of $x_3$.
\end{rem}

We need the following generalisation of Remark~\ref{rem:adj}. The \emph{type} of a face of $\Delta_\Lambda$ is 
the intersection of the types of its vertices. Again, faces of type $\widehat T=S\setminus T$ are in bijective correspondence with the left cosets of $A_{\Lambda \setminus T}$, where $\Lambda \setminus T\subset \Lambda$ is the subdiagram induced on $\widehat T$.
The \emph{type} of a vertex $v$ of the barycentric subdivision $\Delta'_\Lambda$ of $\Delta_\Lambda$ is 
the type of the face of $\Delta_\Lambda$  with barycentre $v$.
Given two vertices $x,y$ of~$\Delta'_\Lambda,$
we write $x\sim y$ if they are contained a common simplex of $\Delta_\Lambda$. Then $x\sim y$ if and only if the corresponding two left cosets intersect.

\begin{lem}
	\label{lem:transitive}
 Let  $x_1,x_2,x_3$ be vertices of $\Delta'_{\Lambda}$ of type $\widehat S_1,\widehat S_2,\widehat S_3$, respectively. Suppose that any  $s_1\in S_1\setminus S_2$ and $s_3\in S_3\setminus S_2$ belong to distinct components of $\Lambda\setminus S_2$. If $x_1\sim x_2$ and $x_2\sim x_3$, then $x_1\sim x_3$. 
\end{lem}

\begin{proof}
The proof is identical to that of \cite[Lem 10.4]{huang2023labeled}. We include it for the convenience of the reader. We can assume that $S_2$ does not contain $S_1$ (or $S_3$), since otherwise the left coset corresponding to $x_2$ would be contained in the left coset corresponding to $x_1$, and so $x_2\sim x_3$ would imply $x_1\sim x_3$.

Up to the left translation, we can assume that $x_2$ corresponds to the identity coset $A_{\Lambda\setminus S_2}$. For $i=1,3$, let $\Lambda_i$ be the union of the components of $\Lambda\setminus S_2$ that are disjoint from $S_i$. By our hypotheses, we have $\Lambda_i\neq\emptyset$ and $\Lambda_1\cup\Lambda_3$ contains all the vertices of $\Lambda\setminus S_2$. Since $A_{\Lambda\setminus S_2}$ is the direct sum of the Artin groups with Coxeter diagrams the components of $\Lambda\setminus S_2$, any left coset of $A_{\Lambda_1}$ in $A_{\Lambda\setminus S_2}$ and any left coset of $A_{\Lambda_3}$ in~$A_{\Lambda\setminus S_2}$ intersect. For $i=1,3$, let $H_i$ be the left coset of $A_{\Lambda\setminus S_i}$ in~$A_\Lambda$ corresponding to $x_i$. For $i=1,3$, since $\Lambda_i\subset \Lambda\setminus S_i$, we have that $A_{\Lambda\setminus S_2}\cap H_i$ contains a left coset of $A_{\Lambda_i}$ in $A_{\Lambda\setminus S_2}$. Thus we have $A_{\Lambda\setminus S_2}\cap H_1\cap H_3\neq\emptyset$ and so $x_1\sim x_3$.
\end{proof}

\subsection{Posets}
\label{subsec:chamber complex}
Let $S$ be a set of size $n$.
A simplicial complex $X$ is of \emph{type $S$} if all the maximal simplices of $X$ have dimension $n-1$ and there is a function $\type \colon X^0\to S$ such that $\type(x)\neq\type(y)$ if $x$ and $y$ are neighbours. 
Note that the restriction of $\type$ to the vertex set of each maximal simplex is a bijection.

As an example, if $A_\Lambda$ is an Artin group, and $S$ is the vertex set of $\Lambda$, then the Artin complex $\Delta_{\Lambda}$ is a simplicial complex of type $S$ (or, rather, $\widehat S$).

\begin{defin}
	\label{def:order}
	Let $X$ be a simplicial complex of type $S$. Any total order $<$ on $S$ induces the following relation $<$ on $X^0$. We declare $x<y$ if $x$ and $y$ are neighbours, and $\type(x)<\type(y)$.
\end{defin}

Let $P$ be a poset, i.e.\ a partially ordered set.
Let $Q\subset P$. An \emph{upper bound} (resp.\ \emph{lower bound}) for $Q$ is an element $x\in P$ such that $q\le x$ (resp.\ $x\le q$) for any $q\in Q$. An upper bound $x$ of $Q$ is the  \emph{join} of $Q$ if $x\le y$ for any upper bound $y$ of $Q$. A lower bound $x$ of $Q$ is the \emph{meet} of $Q$ if $y\le x$ for any lower bound $y$ of $Q$. We write $x\vee y$ and $x\wedge y$ for the join and meet of $\{x,y\}$ (if they exist). We say that $P$ is a \emph{lattice} if $P$ is a poset and any two elements of $P$ have the join and the meet. 
For $a,b\in P$ with $a\le b$, the \emph{interval $[a,b]$ between $a$ and $b$} is the collection of all the elements $x$ of $P$ satisfying $a\le x$ and $x\le b$. A poset $P$ is \emph{weakly graded} if there is a \emph{poset map} $r\colon P\to \mathbb Z$, i.e.\ for every $x<y$ in $P$, we have $r(x)<r(y)$. Such $r$ is called a \emph{rank} function.
A \emph{bowtie} $x_1y_1x_2y_2$ consists of distinct elements of $P$ satisfying $x_i<y_j$ for all $i,j=1,2$. The name comes from the fact that if we draw $y_1,y_2$ above $x_1,x_2$ in the Hasse diagram, then we obtain a bowtie shaped configuration.
\begin{defin}
	A poset $P$ is \emph{bowtie free} if for any bowtie $x_1y_1x_2y_2$ there exists $z\in P$ satisfying $x_i\le z\le y_j$ for all $i,j=1,2$.
\end{defin}

\begin{lem}[{\cite[Prop 1.5]{brady2010braids} and \cite[Prop 2.4]{haettel2023new}}]
	\label{lem:posets}
	If $P$ is a bowtie free weakly graded poset, then any subset $Q\subset P$ with an upper bound has the join, and any $Q\subset P$ with a lower bound has the meet.
\end{lem}

\begin{proof}
	The case of $|Q|=2$ is \cite[Prop 2.4]{haettel2023new}. This easily implies the case of finite~$Q$. Thus for infinite~$Q$ with an upper bound $u$, we have that each finite subset of $Q$ has the join, which is $\le u$. Let $T$ be a finite subset of $Q$ such that the join $u_T$ of~$T$ has largest rank among all the joins of the finite subsets of $Q$. We claim that $u_T$ is the join of $Q$. To justify the claim, it is enough to show that for any $q\in Q$, we have $q\le u_T$. Let $u$ be the join of $T\cup \{q\}$. Then we have $u_T\le u$. On the other hand, we have $r(u)\le r(u_T)$ by our choice of $T$. Thus $u_T=u$, and so $q\le u_T$, as desired. The assertion on the meet is proved analogously.
\end{proof}

\begin{defin}
	\label{def:flag}
	A poset is \emph{upward flag} if any three pairwise upper bounded elements have an upper bound. A poset is \emph{downward flag} if any three pairwise lower bounded elements have a lower bound. A poset is \emph{flag} if it is upward flag and downward flag.
\end{defin}

\begin{defin}
	\label{def:weak flag}
A poset is \emph{weakly upward flag} if any three elements pairwise upper bounded by non-maximal elements have an upper bound.
	Analogously, we define \emph{weakly downward flag} and \emph{weakly flag} posets.
\end{defin}

We will be often discussing Coxeter diagrams $\Lambda$ that are linear graphs with consecutive vertices $s_1,\dots, s_n$. In that case, we write shortly $\Lambda=s_1\cdots s_n$. 

\begin{thm}[{\cite[Prop 6.6]{haettel2021lattices}}]
	\label{thm:tripleBn}
	Let $\Lambda=s_1\cdots s_n$ be the Coxeter diagram of type~$B_n$ with $m_{s_{n-1}s_n}=4$, and total order $\hat s_1<\cdots< \hat s_n$.  
	Then the induced relation~$<$ on $\Delta_\Lambda^0$ from Definition~\ref{def:order} is a partial order that is
	weakly graded, bowtie free and upward flag
\end{thm}

\begin{theorem}[{\cite[Thm 7.1]{huang2024}}]
	\label{thm:flag}
	Let $\Lambda=s_1s_2s_3$ be the Coxeter diagram of type~$H_3$ with $m_{s_2s_3}=5$, and $\hat s_1<\hat s_2< \hat s_3$. Then the induced relation~$<$ on $\Delta_\Lambda^0$ from Definition~\ref{def:order} is a partial order that is
	weakly graded, bowtie free and upward flag.
\end{theorem}
\subsection{Disc diagrams}
\label{subsec:discdiagram}
A map from a CW complex $Y$ to a CW complex $X$ is \emph{combinatorial} if its restriction to each open cell of $Y$ is a homeomorphism onto an open cell of $X$. A CW complex $X$ is \emph{combinatorial}, if the attaching map of each open cell of $X$ is combinatorial for some subdivision of the sphere.

A \emph{disc diagram} $D$ is a finite contractible combinatorial complex with a fixed embedding in the plane $\mathbb R^2$. We can then view $\mathbb R^2\cup 
\{\infty\}$ as the combinatorial complex that is a union of $D$ and a $2$-cell \emph{at infinity}. The \emph{boundary cycle} of $D$ is the edge-loop in $D$ that is the attaching map of the cell at infinity. A \emph{disc diagram in a combinatorial complex $X$} is a combinatorial map $f\colon D\to X$, where $D$ is a disc diagram. The \emph{boundary cycle} of $f$ is the composition of the boundary cycle of $D$ with $f$. A disc diagram $f\colon D\to X$ is \emph{minimal} if it has minimal area (i.e.\ number of $2$-cells in $D$) among all diagrams in $X$ with the same boundary cycle. We say that $f$ is \emph{reduced} if it is locally injective at $D\setminus D^{0}$. The following is a well-known variation of a result by Van Kampen. 

\begin{lem}[{\cite[Lem~2.16 and 2.17]{mccammond2002fans}}] 
\label{lem:VK}
Any homotopically trivial cycle $\omega$ in~$X$ is the boundary cycle of a disc diagram $f\colon D\to X$. Any minimal disc diagram is reduced.
\end{lem}

Note that if $\omega$ is not embedded, then $D$ might not be homeomorphic to a disc. 

Suppose that the corners of each $p$-gon of a disc diagram $D$ are assigned real numbers, called \emph{angles}, with sum $(p-2)\pi$. Let $v$ be a vertex of $D$ whose link in $D$ has $n_v$ components. We define the \emph{curvature at $v$ of $D$} to be $(2-n_v)\pi$ minus the sum of all the angles at $v$. We will use the following `Gauss--Bonnet theorem'.

\begin{thm}[{\cite[Thm~4.6]{mccammond2002fans}}]
\label{thm:GB}
The sum of the curvatures at all the vertices~$v$ of $D$ equals $2\pi$.
\end{thm}

Here is an example of an application of Theorem~\ref{thm:GB}. However, we will be using without reference various similar results on $2$-dimensional $\mathrm{CAT}(0)$ simplicial complexes, especially in Sections~\ref{sec:typeI}--\ref{sec:typeII}.

\begin{lem} 
\label{lem:CAT(0)}
Let $Y$ be a $2$-dimensional $\mathrm{CAT}(0)$ simplicial complex of type $\{\hat s,\hat t, \hat p\}$, all of whose triangles have type $\hat s\hat t\hat p$ with angles $\frac{\pi}{4},\frac{\pi}{2},\frac{\pi}{4}$ or $\frac{\pi}{6},\frac{\pi}{2},\frac{\pi}{3}$. 
\begin{enumerate}[(i)]
\item
Then any induced $4$-cycle $\omega$ in $Y^1$ has type $\hat s\hat p\hat s\hat p$ and has a common neighbour of type $\hat t$.
\item
In the $\frac{\pi}{6},\frac{\pi}{2},\frac{\pi}{3}$ case, for $n\leq 6$, any embedded $2n$-cycle $\omega$ in $Y^1$ of type $(\hat t\hat p)^n$ has a common neighbour of type $\hat s$ and satisfies $n=6$. 
\item
In the $\frac{\pi}{4},\frac{\pi}{2},\frac{\pi}{4}$ case, for $n\leq 4$, any embedded $2n$-cycle $\omega$ in $Y^1$ of type $(\hat t\hat p)^n$ has a common neighbour of type $\hat s$ and satisfies $n=4$. 
\end{enumerate}
\end{lem}
\begin{proof} For part (i), let $D\to Y$ be a minimal disc diagram bounded by the $4$-cycle~$\omega$. Let $T_1,\ldots, T_4\subset D$ be the triangles containing the boundary edges. Note that since $\omega$ is induced, the~$T_i$ are distinct. Furthermore, the sum of the angles of~$T_i$ incident to $\partial D$ is at least~$4\cdot \frac{\pi}{2}$. Since $D$ is minimal, it is locally $\mathrm{CAT}(0)$, and by Theorem~\ref{thm:GB} 
the sum of the angles at $\partial D$ is at most~$2\pi$. Consequently, we have equality, and there are no other triangles in $D$ incident to~$\partial D$. As a result, there are no other triangles in $D$, as desired.

For part (ii), we consider $2n$ triangles $T_i\subset D$ containing the boundary edges. The sum of the angles of $T_i$ incident to $\partial D$ is at least $2n\big(\frac{\pi}{2}+\frac{\pi}{3}\big)=\frac{5n\pi}{3}$. By Theorem~\ref{thm:GB}, 
the sum of the angles at $\partial D$ is at most $2n\pi-2\pi$. Consequently $\frac{5n\pi}{3}\leq 2n\pi-2\pi$, and so $n\geq 6$  and we conclude as before. This part also follows from \cite[Lem~9.8]{huang2023labeled}.
The proof of part (iii) is analogous to that of part (ii).
\end{proof}

\section{Complexes for hyperplane arrangements}
\label{sec:prelim1}
\subsection{Hyperplane arrangements and their dual polyhedra}\label{subsec:zonotope}
A \emph{hyperplane arrangement} in the vector space $\mathbb R^n$ is a locally finite family $\mathcal A$ of affine hyperplanes.
 Let $\cq(\ca)$ be the set of nonempty affine subspaces that are intersections of subfamilies of $\ca$ (here $\mathbb R^n\in\cq(\ca)$ as the intersection of an empty family). Each point $x\in\mathbb R^n$ belongs to a unique element of $\cq(\ca)$ that is minimal with respect to inclusion, called the \emph{support} of~$x$. A \emph{fan} of $\ca$ is a maximal connected subset of $\mathbb R^n$ consisting of points with the same support. 
Denote the collection of all fans of $\ca$ by $\fan(\ca)$.
Note that $\mathbb R^n$ is the (disjoint) union of $\fan(\ca)$. 
We define a partial order on $\fan(\ca)$ so that $U_1<U_2$ if $U_1$ is contained in the closure of $U_2$.
Let $b\Si_{\ca}$ be the simplicial complex that is the geometric realisation of this poset.
For each $U\in \fan(\ca)$, we choose a point $x_U\in U$. 
This gives a piecewise linear embedding $b\Si_{\ca}\subset\mathbb R^n$ sending the vertex of $b\Si_\ca$ corresponding to $U$ to $x_U$.

By  
\cite[pp. 606-607]{s87}, the simplicial complex $b\Si_{\ca}$ is the barycentric subdivision of a combinatorial complex $\Si_{\ca}$ whose vertices correspond to the top-dimensional fans.
Namely, for each vertex of $b\Si_{\ca}$ corresponding to $U\in \fan(\ca)$, the union of all the simplices of $b\Si_{\ca}$ corresponding to chains with smallest element $U$ is homeomorphic to a closed disc \cite[Lem 6]{s87}, which becomes the face of $\Si_{\ca}$ corresponding to~$U$. We will  sometimes view $b\Si_\ca$ and $\Si_\ca$ as subspaces of $\mathbb R^n$.
For $B\in \cq(\ca)$, a face $F$ of $\Si_{\ca}$ is \emph{dual} to $B$, if $B$ contains the fan $U$ corresponding to $F$ and $\dim(B)=\dim(U)$.
We equip the 1-skeleton of $\Si_{\ca}$ with the path metric $d$ such that each edge has length $1$. Given vertices $x,y\in \Si^0_{\ca}$, it turns out that $d(x,y)$  is the number of hyperplanes separating $x$ and~$y$ \cite[Lem~1.3]{deligne}.  

\begin{lem}[{\cite[Lem 3]{s87}}]
	\label{lem:gate}
	Let $x\in \Si^0_{\ca}$ and let $F$ be a face of $\Si_{\ca}$. Then there exists unique $\Pi_F(x)\in F^0$ such that $d(x,\Pi_F(x))\le d(x,y)$ for any $y\in F^0$. \end{lem}

The vertex $\Pi_F(x)$ is called the \emph{projection} of $x$ to $F$.  A hyperplane $H\in \ca$ \emph{crosses} a face $F$ of $\Si_{\ca}$ if $H$ is dual to an edge of $F$. For an edge $xy$ of $\Si_\ca$, if the hyperplane dual to $xy$ crosses $F$, then $\Pi_F(x)\Pi_F(y)$ is an edge dual to the same hyperplane, otherwise we have $\Pi_F(x)=\Pi_F(y)$. Thus $\Pi_F$ extends naturally to a map $\Si^1_\ca\to F^1$.

\begin{lem}
\label{lem:unlabelled}
	Let $E$ and $F$ be faces of $\Si_\ca$. Then $\Pi_F(E^0)=F'^0$ for some face $F'\subset F$. 
\end{lem}

\begin{proof}
 Let $\ca'$ be the collection of all the hyperplanes that cross both $E$ and $F$. Note that for any edge-path $P$ in $E$, any edge of $\Pi_F(P)$ is dual to an element of $\ca'$.

Let $B\in\cq(\ca)$ be the intersection of all the elements of $\ca'$. If $\ca'=\emptyset$, then we set $B=\mathbb R^n$. Let $E'$ be any face of $E$ dual to $B$, which is a vertex for $B=\mathbb R^n$, and let $w$ be any vertex of $E'$. By the above discussion, $\Pi_F(w)$ is contained in a face $F'$ of $\Si$ dual to $B$. Moreover, we have $F'\subset F$. Furthermore, $\Pi_F({E'}^0)={F'}^0$ and $\Pi_F(E^0)\subset {F'}^0$, as desired. 
\end{proof}

In the situation of Lemma~\ref{lem:unlabelled}, we write $F'=\Pi_F(E)$. The assignment $E\to \Pi_F(E)$ gives rise to a piecewise linear map $\Pi_F\colon\Si_\ca \cong b\Si_\ca\to bF\cong F$.

\subsection{Salvetti complex}
\label{sec:Salvetti}
Let $V=\Si_\ca^0$. 
Consider the set of pairs $(F,v)$, where $F$ is a face of $\Si_\ca$ and $v\in V$. We define an equivalence relation $\sim$ on this set by $(F,v)\sim (F',v')$ whenever  $F=F'$ and $\Pi_F(v') = \Pi_F(v).$
Note that each equivalence class $[F,v']$ contains a unique representative of form $(F,v)$ with $v\in F^0$.  The \emph{Salvetti complex} $\widehat\Si_\ca$ is obtained from  $\Si_\ca\times V$ (a disjoint union of copies of $\Si_\ca$) by identifying faces $F\times v$ and $F\times v'$ whenever $[F,v]=[F,v']$ \cite[pp.\ 608]{s87}.
For example, for each edge $F=v_0v_1$ of $\Si_\ca$, we obtain two edges $F\times v_0$ and $F\times v_1$ of $\widehat\Si_\ca$, glued along their endpoints $v_0\times v_0$ and $v_1\times v_1$. We orient the edge $F\times v_0$ from $v_0\times v_0$ to $v_1\times v_0=v_1\times v_1$. Then $\widehat\Si_\ca^0=V$, while $\widehat\Si_\ca^1$ is obtained from $\Si_\ca^1$ by doubling each edge.  
Thus each edge of the form $F\times v$ is oriented so that its endpoint is farther from $v$ in~$F^1$ than its starting point.

There is a natural map $p\colon\widehat\Si_\ca\to\Si_\ca$ forgetting the second coordinate. 
For each subcomplex $Y$ of $\Si_\ca$, we write $\widehat Y=p^{-1}(Y)$. 
If $F$ is a face of $\Si_\ca$, then $\widehat F$ is a \emph{standard subcomplex} of $\widehat\Si_\ca$. 

\begin{lem}[{\cite[Lem 4.5]{huang2024}}]
	\label{lem:compactible}
	Let $E$ and $F$ be faces of $\Si_\ca$.
	If $[E,v_1]=[E,v_2]$, then $[\Pi_F(E),v_1]=[\Pi_F(E),v_2]$.
\end{lem}

\begin{defin}
	\label{def:retraction}
	Let $F$ be a face of $\Si_\ca$. Consider the disjoint union of $V$ copies of the map $\Pi_F$, where $\Pi_F\times v\colon\Si_\ca\times v\to  F\times v$. It follows from Lemma~\ref{lem:compactible} that this map factors to a map $\Pi_{\widehat F}\colon \widehat\Si_\ca\to \widehat F$, which is a retraction (see \cite[Thm 2.2]{godelle2012k}).
\end{defin}

The following key property of $\Pi_{\widehat F}$ follows directly from Definition \ref{def:retraction}.

\begin{lem}	\label{lem:retraction property}
	Let $E$ and $F$ be faces of $\Si_\ca$. Then $\Pi_{\widehat F}(\widehat E)=\widehat{\Pi_F(E)}$.
\end{lem}

Let $\mathcal A\otimes \mathbb C$ be the complexification of $\mathcal A$, which is a collection of affine complex hyperplanes in $\mathbb C^n$. Define
$$M(\mathcal A\otimes \mathbb C)=\mathbb C^n - \bigcup_{H\in\mathcal A} (H\otimes \mathbb C).$$
It follows from \cite[Thm 1]{s87} that $\widehat\Si_\ca$ is homotopy equivalent to $M(\mathcal A\otimes \mathbb C)$, and so they have isomorphic fundamental groups. 

In the remaining part of this subsection, we assume that $W_\Lambda$ is a finite Coxeter group with its canonical representation $\rho
\colon W_\Lambda\to \mathbf{GL}(n,\mathbb R)$ \cite[Chap~6.12]{davis2012geometry}. A \emph{reflection} of $W_\Lambda$ is a conjugate of $s\in S$. Each reflection fixes a hyperplane in~$\mathbb R^n$, which we call a \emph{reflection hyperplane}. Let $\mathcal A$ be the family of all reflection hyperplanes. The hyperplane arrangement $\ca$ is the \emph{reflection arrangement} associated with $W_\Lambda$. We denote $\Si_\Lambda=\Si_\ca$ and $\od_\Lambda=\od_\ca$. Since $W_\Lambda$ permutes the elements of~$\mathcal A$, there is an induced action $W_\Lambda\act M(\ca\otimes \mathbb C)$ and an induced action $W_\Lambda\act \widehat\Si_\ca$, which are free.
The union of $\ca$ cuts the unit sphere of $\mathbb R^n$ into a simplicial complex, which is isomorphic to the Coxeter complex $\bC_\Lambda$ and dual to $\Si_\Lambda$.  The following are standard \cite[\S 3.2 and 3.3]{paris2014k}. 
\begin{itemize}
	\item $\pi_1M(\ca\otimes \mathbb C)=PA_\Lambda$ \cite{lek},
	\item $\pi_1(M(\ca\otimes \mathbb C)/ W_\Lambda)=\pi_1(\widehat\Si_\Lambda/W_\Lambda)=A_\Lambda$,
	\item $\widehat\Si^2_\Lambda/W_\Lambda$ is isomorphic to the presentation complex of $A_\Lambda$.
\end{itemize}

\begin{defin}
	\label{def:support} 
    Note that 
	$\od^1_\Lambda$ is isomorphic to the Cayley graph of $W_\Lambda$ (with edges appropriately oriented), and  $\Si^1_\Lambda$ is isomorphic to the unoriented Cayley graph of~$W_\Lambda$ (obtained by collapsing each double edge of the usual Cayley graph to a single edge). Thus the edges of $\od_\Lambda$ and $\Si_\Lambda$ are labelled by the elements of $S$. The \emph{type} of a face of $\Si_\Lambda$ or a standard subcomplex of $\od_\Lambda$ is defined to be the collection of the labels of edges of this subcomplex. 
\end{defin}

\begin{rem}[Alternative description of the Artin complex $\Delta_\Lambda$]
	\label{rmk:alternative}
	Let $\widetilde \Sigma_\Lambda$ be the universal cover of $\od_\Lambda$. Then by the last bullet point above, $\widetilde \Sigma_\Lambda^1$ can be identified with the Cayley graph of $A_\Lambda$. An  \emph{elevation} of a subcomplex of 
$\od_\Lambda$ to $\widetilde \Sigma_\Lambda$ is a connected component of the preimage of this subcomplex under the covering map. 
	Vertices of~$\Delta_\Lambda$ are in bijective correspondence with the elevations of the standard subcomplexes of type $\hat s$ for $s\in S$, since the vertex set of such an elevation is a left coset $gA_{\hat s}\subset A_\Lambda=\widetilde \Sigma_\Lambda^0$.
	Vertices of $\Delta_\Lambda$ span a simplex if their corresponding elevations have non-empty common intersection.  We will call these elevations \emph{standard subcomplexes of $\widetilde \Sigma_\Lambda.$} By \cite{lek}, the intersection of a collection of standard subcomplexes of $\widetilde \Sigma_\Lambda$ of types $S_i$ is empty or is a standard subcomplex of type $\bigcap_i S_i$.
\end{rem}

\subsection{Collapsing hyperplanes}
\label{ss:col}
Let $\mathcal A$ be a hyperplane arrangement and let $\mathcal A'\subset \mathcal A$. Note that each fan of~$\ca$ is contained in a unique fan of $\ca'$. Since the vertices of $b\Si_{\ca}$ (the barycentric subdivision of $\Si_{\ca}$) correspond to the fans of $\ca$, this gives a map from the vertex set of $b\Si_\ca$ to the vertex set of $b\Si_{\ca'}$. This map extends to a simplicial map $\kappa =\kappa_{\mathcal A'} \colon b\Si_\ca\to b\Si_{\ca'}$, which can also be viewed as a piecewise linear map from $\Si_\ca$ to $\Si_{\ca'}$.

By the description of the faces of $\Si_\ca$ in the terms of the simplices of $b\Si_\ca$ at the beginning of Section~\ref{subsec:zonotope}, 
$\kappa$ maps each face $F$ of~$\Si_\ca$ onto a face of $\Si_{\ca'}$ that we denote $\kappa(F)$. Note that if an edge $e$ of $\Si_\ca$ is dual to a hyperplane outside $\ca'$, then $\ka_{\ca'}(e)$ is a vertex, otherwise $\ka_{\ca'}(e)$ is an edge.

Furthermore, for $v,v'\in\Si^0_\ca$ satisfying $\Pi_F(v')=\Pi_F(v)$, we have $\Pi_{\kappa(F)}(\kappa(v'))=\Pi_{\kappa(F)}(\kappa(v))$.
Thus $\kappa$ 
induces a piecewise linear map $\hk\colon\od_\ca\to \od_{\ca'}$.

\subsection{Central arrangements}
\label{subsec:central}
Let $\ca$ be a hyperplane arrangement in $\mathbb R^n$ that is \emph{central}, that is, all its hyperplanes pass through the origin. 
Let $H\in \ca$ and let $\mathbb R^{n-1}\subset \mathbb R^{n}$ be parallel to and distinct from $H$. The \emph{deconing} $\ca_H$ of $\ca$ with respect to $H$ is the hyperplane arrangement in $\mathbb R^{n-1}$ consisting of the intersections of the elements of $\ca$ with $\mathbb R^{n-1}$. Note that $\ca_H$ is well-defined, since choosing a different parallel hyperplane $\mathbb R^{n-1}$ gives rise to a hyperplane arrangement differing from the first one by an affine isomorphism. It is well-known that $M(\mathcal A\otimes \mathbb C)$ is homeomorphic to $M(\ca_H\otimes \mathbb C)\times \mathbb C^*$, where $\mathbb C^*=\mathbb C\setminus\{0\}$, see e.g.\ \cite[Prop 5.1]{orlik2013arrangements}. Thus $\pi_1M(\mathcal A\otimes \mathbb C)\cong\pi_1M(\ca_H\otimes \mathbb C)\oplus\mathbb Z$. It is also possible to see this isomorphism on the level of the Salvetti complex, where we identify $\Si_{\mathcal A_H}$ with the subcomplex of $\Si_{\mathcal A}$ on one side of~$H$.

\begin{lem}
	\label{lem:injective}
	Let $\mathcal A$ be a central arrangement and $H\in\mathcal A$. Then the inclusion $i\colon\widehat \Si_{\mathcal A_H}\to \widehat \Si_{\mathcal A}$ is $\pi_1$-injective. Moreover, $\pi_1 \widehat\Si_{\mathcal A}= i_*(\pi_1\od_{\mathcal A_H})\oplus \mathbb Z$.
\end{lem}

\begin{proof}
Let $b\od_\ca$ denote the barycentric subdivision of $\od_\ca$.
Recall that, in \cite[pp.\ ~608]{s87}, Salvetti constructed a piecewise linear embedding $\phi: b\od_\ca\to M(\ca\otimes \mathbb C)$. He proved in \cite[pp.\ 611]{s87} that $\phi$ is a homotopy equivalence. Let $\mathbb R^{n-1}\subset \mathbb R^n$ be as above. 
Then $M(\ca_H\otimes \mathbb C)$ is a subspace of $\mathbb R^{n-1}\otimes \mathbb C$. We can homotopy $\phi$ so that $\phi(\od_{\ca_H})\subset M(\ca_H\otimes \mathbb C)$ and $\phi|_{\od_{\ca_H}}: \od_{\ca_H}\to M(\ca_H\otimes \mathbb C)$ is a homotopy equivalence. 
Thus the lemma follows from the paragraph preceding its statement. 
\end{proof}

\subsection{Line arrangements}
\label{subsec:line arrangement}
Let $\ca$ be a central arrangement of lines in $\mathbb R^2$. Let $\omega$ be a locally embedded edge-path in $\od_\ca$, and $\ell\in \ca$. An \emph{$\ell$-segment} of $\omega$ is a maximal subpath mapped to an edge dual to a fan in $\ell$ under $\od_\ca\to\Si_\ca$.

\begin{lem}[{\cite[Lem 3.6]{falk1995k}}]
	\label{lem:falk}
	Suppose that $P$ is a locally embedded homotopically trivial edge-loop in $\od_\ca$, and $\ell\in \ca$. Then $P$ contains at least two $\ell$-segments.
\end{lem}

A \emph{minimal positive path} in $\od_\ca$ is a minimal length path between its endpoints that is positively oriented. (The orientation of the edges was introduced at the beginning of Section~\ref{sec:Salvetti} and discussed in Definition~\ref{def:support}.)
Note that the boundary of each $2$-cell of $\od_\ca$ is a union of two minimal positive paths from some vertex $x$ to its antipodal vertex $y$. We call $x$ the \emph{source} of this $2$-cell, and $y$ the \emph{sink}. Let $\Delta_x$ be the concatenation of a minimal positive path from $x$ to $y$ and a minimal positive path from $y$ to $x$. The element represented by $\Delta_x$ in $\pi_1(\od_\ca,x)$ is independent of the choice of the paths. 

Below, we denote the edges of the two minimal length paths from $x$ to $y$ in $\Si_\ca^1$ by $e_1\cdots e_n$ and $d_n\cdots d_1$. For an edge $e_i$ of $\Si_\ca^1$, we label both edges of $\widehat e_i$ by $e_i$.  Note that they are oriented in opposite directions.
Let $z$ be the common vertex of~$e_1$ and~$e_2$.

\begin{lem}
	\label{lem:relations}
	\begin{enumerate}
		\item Let $P$ be an edge-path in $\od_\ca$ from $x_1$ to $x_2$.  
		Then $\Delta_{x_1}P$ is homotopic, relative to the endpoints, to $P\Delta_{x_2}$. In particular, $\Delta_{x_1}$ is central in $\pi_1(\od_\ca,x_1)$.
		\item Paths $\Delta_ze^{-1}_2e^{-1}_3\cdots e^{-1}_ne^{-1}_n\cdots e^{-1}_3e^{-1}_2$ and $e_1e_1$ represent the same element of $\pi_1(\od_\ca,z)$.
	\end{enumerate}
\end{lem}

\begin{proof}
	Assertion 1 is a consequence of \cite[Lem 1.26 and Prop 1.27]{deligne}. For Assertion 2, note that $\Delta_z\sim e_2\cdots e_nd_1d_1e_n\cdots e_2$, where $\sim$ stands for a homotopy relative to the endpoints. The union of the two 2-cells of $\od_{\ca}$ with sources the two endpoints of $e_1$ form a cylinder in~$\od_\ca$ with boundary paths~$e^2_1$ and~$d^2_1$. More precisely, $e^2_1\sim e_2e_3\cdots e_nd^2_1e^{-1}_n\cdots e^{-1}_3e^{-1}_2$. Thus
	$$
	\Delta_z\sim e_2e_3\cdots e_nd^2_1e^{-1}_n\cdots e^{-1}_3e^{-1}_2e_2e_3\cdots e_ne_n\cdots e_2\sim  e^2_1e_2e_3\cdots e_ne_n\cdots e_2.
	$$
\end{proof}

Let $e_1,d_1$ be dual to the fans in $\ell\in \mathcal A$. 
Let $i\colon \widehat\Sigma_{\ca_l}\to \widehat \Si_\ca$ be as in Lemma~\ref{lem:injective}.
\begin{lem}
	\label{lem:representative}
	We have a short exact sequence
	$$
	\pi_1(\widehat\Sigma_{\ca_l},z)\stackrel{i_*}{\to}\pi_1(\od_\ca,z)\stackrel{p_*}{\to}\pi_1(\widehat e_1,z),
	$$
	where 
	$p=\Pi_{\widehat e_1}$.
	In particular, if for a representative $P$ of  $\alpha\in \pi_1(z,\od_\ca)$ the path $\Pi_{\widehat e_1}(P)$ is homotopically trivial, then $\alpha$ can be represented by a loop in $\widehat\Sigma_{\ca_l}$.
\end{lem}

\begin{proof}
Since $\pi_1(\widehat e_1,z)$ is isomorphic to $\mathbb Z$, this follows from Lemma~\ref{lem:injective}, from $\mathrm{im}\, i_*<\ker p_*$, and from the surjectivity of $p_*$.
\end{proof}

\begin{lem}
\label{lem:collapse}
Let $\ca'=\ca \setminus \{l\}$, where $l$ is dual to $e_1$. Let $P$ be an edge-path in $\widehat e_3\cup \widehat e_4\cup \cdots \cup \widehat e_n$. If $\hk_{\ca'}(P)$ is homotopic, relative to the endpoints, into~$\hk_{\ca'}(\widehat e_j)$ in~$\od_{\ca'}$, for some $j\neq 1$, then $P$ is homotopic, relative to the  endpoints, into~$\widehat e_j$ in~$\od_{\ca}$.
\end{lem}

 For $j=2$ this means that $P$ is a homotopically trivial edge-loop by Lemma~\ref{lem:injective}.

\begin{proof}
Let $f_i=\ka_{\ca'}(e_i)$. 
By Lemma~\ref{lem:injective}, $\widehat f_3\cup \widehat f_4\cup \cdots \cup \widehat f_n\subset \od_{\ca'}$ is $\pi_1$-injective. Thus for $j\neq 2$ the edge-path $\hk_{\ca'}(P)$ is homotopic in $\widehat f_3\cup \widehat f_4\cup \cdots \cup \widehat f_n\subset \od_{\ca'}$ into $\widehat f_j$. Consequently, $P$ is homotopic in $\widehat e_3\cup \widehat e_4\cup \cdots \cup \widehat e_n$ into $\widehat e_j$. 

For $j=2$, since $P$ is contained in $\widehat e_3\cup \widehat e_4\cup \cdots \cup \widehat e_n$ and $\hk_{\ca'}(P)$ is homotopic, relative to the endpoints, into~$\widehat f_2$, we have that $P$ is an edge-loop. If $P$ is homotopically nontrivial in $\widehat e_3\cup \widehat e_4\cup \cdots \cup \widehat e_n$, then by Lemma~\ref{lem:injective} $\hk_{\ca'}(P)$ is homotopically nontrivial in~$\od_{\ca'}$. However, by considering the retraction $\Pi_{\widehat f_2}\colon\od_{\ca'}\to \widehat f_2$, we obtain that a loop in $\widehat f_2$ cannot be homotopic in $\od_{\ca'}$ to a homotopically nontrivial loop in $\widehat f_3\cup \widehat f_4\cup\cdots \cup\widehat f_n$, which is a contradiction.
\end{proof}

\section{Some sub-arrangements of the $H_3$-arrangement}
\label{sec:subarrangement}
Let $\Lambda$ be the Coxeter diagram of type~$H_3$, which is the linear graph with consecutive vertices $abc$ and $m_{ab}=3,m_{bc}=5$. 
Let $\ca$ be the reflection arrangement in $\mathbb R^3$ associated with~$W_\Lambda$.
Denoting the quotient map from the Artin complex $\Delta=\Delta_\Lambda$ to the Coxeter complex $\bC=\bC_\Lambda$ by $\pi$, we say that a vertex $x$ of $\Delta$ has \emph{face type} $C$, where $C$ is the face of $\Si$ dual to $\pi(x)$.

We start with describing a procedure of converting an $n$-cycle in $\Delta$ to a concatenation of $n$ words in $A_\Lambda$ (cf.\ \cite[Def~6.14]{huang2023labeled}). These $n$ words are well-defined up to an appropriate notion of equivalence.

\begin{constr}
	\label{def:ncycle}	
	Let $\omega=x_1\cdots x_n$ be a cycle in $\Delta$ of type $\hat s_1\cdots \hat s_n$. 
	For each $i\in \mathbb Z/n\mathbb Z$, consider a triangle containing $x_ix_{i+1}$ and corresponding $g_i\in A_\Lambda$. Then $g_i=g_{i-1}w_{i}$ for $w_i\in A_{\hat s_i}$ and $w_1\cdots w_n=1$.
	A different choice of such triangles would lead to a word $u_1\cdots u_n$ with $u_i=q^{-1}_{i-1}w_i q_i$ for some $q_i\in A_{S\setminus\{s_i,s_{i+1}\}}$. 
	In this case we say that the words $u_1\cdots u_n$ and $w_1\cdots w_n$ are \emph{equivalent}.

	Given $\omega$, we construct a homotopically trivial edge-loop $P=P_1\cdots P_n$ in $\od$ as follows.
	Let $\widetilde \Sigma$ be the universal cover of $\od=\od_\Lambda$, with standard subcomplexes~$\mathcal T_i$ corresponding to $x_i$. Let $\widetilde P_i$ be edge-paths in $\mathcal T_i$ from $g_{i-1}$ to $g_i$ representing~$w_i$. We define $P_i$ to be the image of $\widetilde P_i$ in $\od$. We have $P_i\subset \whC_i$, where $C_i$ is the face type of~$x_i$.  
	
	Conversely, consider
	a homotopically trivial edge-loop $P=P_1\cdots P_n$ in $\od$, with the $P_i$ contained in \emph{hosts} $\widehat C_i$. Then we can construct a cycle in $\Delta$ as follows.
	For any lift $\widetilde P$ of $P$ to $\widetilde \Sigma$, each $\widetilde P_i$ is contained in a standard subcomplex that is an elevation of $\widehat C_i$ corresponding to a vertex $x_i$ of~$\Delta$. Then $\omega=x_1\cdots x_n$ is a cycle of~$\Delta$.
\end{constr}

\begin{defin}
\label{def:dent}
We refer to Figure~\ref{fig:denting} for the following discussion. Let $\omega$ and $P$ be as in Construction~\ref{def:ncycle}. Suppose that $x_1\neq x_2\neq x_3$ are of types $\hat a,\hat c,\hat a$ or $\hat c,\hat a,\hat c$, and $P_2$ is also contained in 
$\widehat B$ with $B\subset\Si$ a square. 
Then there is a vertex $y\in \Delta^0$ of face type $B$ that is a neighbour of $x_1,x_2,x_3$. Let~$C$ be the face of $\Sigma$ intersecting $B$ along the edge opposite to $B\cap C_2$.  
Then there is a vertex $x\in \Delta^0$ of face type~$C$ that is a neighbour of $y$ and consequently of $x_1,x_3$ by Remark~\ref{rem:adj}. Replacing $x_2$ by~$x$ in $\omega$ is called \emph{denting $x_2$ to $C$}.
\end{defin}
\begin{figure}
    \centering
    \includegraphics[width=0.6\linewidth]{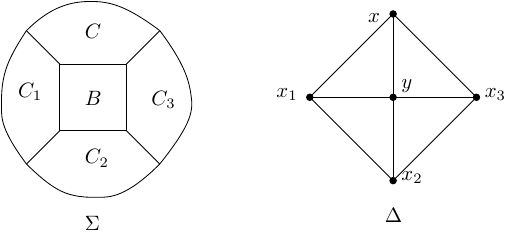}
    \caption{Denting}
    \label{fig:denting}
\end{figure}

\subsection{Sub-arrangement of type I}
\label{subsec:aug1}
\begin{defin}
	\label{def:arrangment}
Consider consecutive vertices $\theta_1, \theta_2,
\theta_3$ of $\bC$ of types $\hat a,\hat b,\hat a$ in a hyperplane of $\ca$. Let $\mathcal H\subset \ca$ be the collection of hyperplanes  
passing through at least one of the $\theta_i$, see Figure~\ref{fig:1}, left. The central arrangement $\mathcal H$ in $\mathbb R^3$ is called the \emph{sub-arrangement of type I}. Let $H\in\mathcal H$ be the hyperplane passing through $\theta_1$  represented as the boundary circle in Figure~\ref{fig:1}, left. Consider the deconing $\mathcal H'=\mathcal H_H$, which is a hyperplane arrangement in $\mathbb R^2$, see Figure~\ref{fig:1}, right.
\end{defin}

\begin{figure}[h]
	\centering
	\includegraphics[scale=0.7]{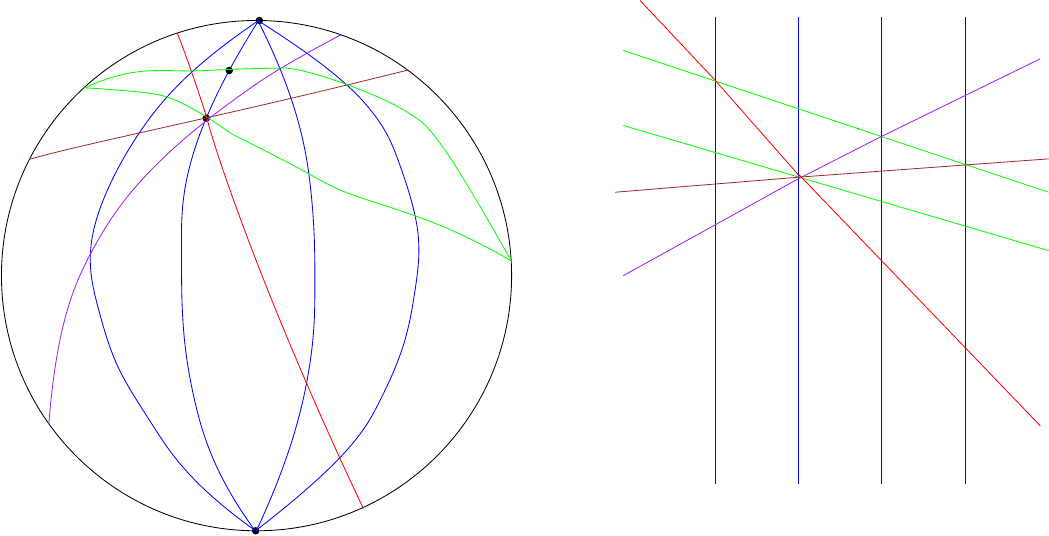}
	\caption{Sub-arrangement of type I}
	\label{fig:1}
\end{figure}

Let $X=\Sigma_{\mathcal H'}$ and $\wX=\widehat \Sigma_{\mathcal H'}$.  
Denote the four vertical hyperplanes of $\mathcal H'$ by $h_1,h_2,h_3,h_4$, from left to right. Let $X_i$ be the union of all the closed faces of $X$ that intersect~$h_i$. 
For $i=1,2,3$, let $\widehat Y_i=\widehat X_i\cap \widehat X_{i+1}$. We define subcomplexes $X_{ij}$ of $X_i$ for $j=1,2$ as follows. For $i=1,3$, let $X_{i1}$ be the subcomplex of $X_i$ coloured white in Figure~\ref{fig:2} (the hexagon), and let $X_{i2}$ be the subcomplex of $X_i$ coloured gray (the union of three squares). For $i=2,4$, let $X_{i1}$ be the subcomplex of $X_i$ coloured gray (the square on the top), and let $X_{i2}$ be the subcomplex of $X_i$ coloured white.

\begin{figure}[h]
	\centering
	\includegraphics[scale=0.85]{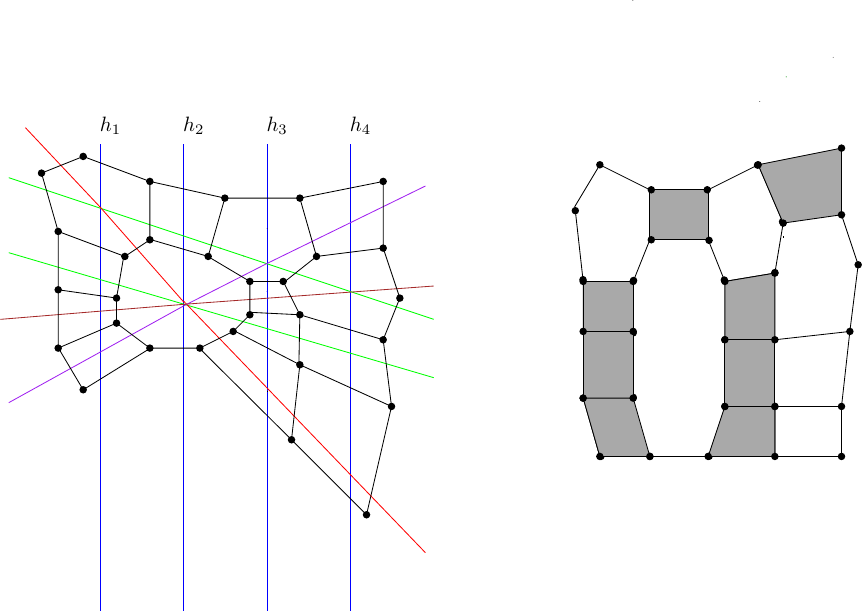}
	\caption{Dual complex}
	\label{fig:2}
\end{figure}	

We now define a simple complex of groups $\mathcal U_{12}$ (see \cite[II.12]{BridsonHaefliger1999}) with fundamental group $\pi_1 (\wX_1\cup \wX_2)$ as in Figure~\ref{fig:3}, whose underlying complex $U_{12}$ is the union of two triangles. The vertex groups and the edge groups are the fundamental groups of the subcomplexes of $\wX$ as labelled in Figure~\ref{fig:3}, and the remaining local groups are trivial. The morphisms between the local groups are induced by the inclusions of the associated subcomplexes, which are injective by Lemma~\ref{lem:injective}.
By \cite[\S6.1]{huang2024}, $\mathcal U_{12}$ is developable with $\pi_1 \mathcal U_{12}=\pi_1 (\wX_1\cup \wX_2)$.

\begin{figure}[h]
		\centering
	\includegraphics[scale=0.7]{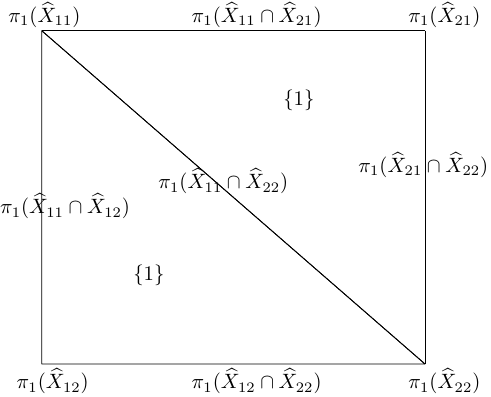}
	\caption{}
	\label{fig:3}
\end{figure}

\begin{defin}
	\label{def:bu}
	Let $\bU_{12}$ be the development of $\mathcal U_{12}$ (cf.\ \cite[II.12]{BridsonHaefliger1999}). Equivalently, the vertices of $\bU_{12}$ (of \emph{face type $\whX_{ij}$}) correspond to the elevations of $\whX_{ij}$ to the universal cover of $\whX_1\cup \whX_2$, which we also call \emph{standard subcomplexes of face type~$\whX_{ij}$}. By \cite[Lem~6.3]{huang2024}, the intersection of a pair of standard subcomplexes and in fact of any collection of standard subcomplexes) is either empty or connected. Vertices of $\bU_{12}$ are neighbours if their corresponding subcomplexes intersect. Vertices of~$\bU_{12}$ form a triangle, if their corresponding subcomplexes have non-empty common intersection (which is then a single vertex).
\end{defin}
Below, we consider the inclusion $\whX=\od_{\mathcal H'}\subset \od_{\mathcal H}$ introduced in Section~\ref{subsec:central}.

\begin{lem}[{\cite[Lem~6.6]{huang2024}}]
	\label{lem:transfer1}
	The inclusion $\whX_1\cup\whX_2\subset \od_{\ch}$ is $\pi_1$-injective.
\end{lem}

There is a natural action $\pi_1(\whX_1\cup \whX_2)\act \bU_{12}$, with quotient $U_{12}$. We equip $U_{12}$ with the piecewise Euclidean metric of the unit square.
It pulls back to a piecewise Euclidean metric on $\bU_{12}$.  

\begin{lem}[{\cite[Lem~6.4]{huang2024}}]
	\label{lem:CAT0}
	$\bU_{12}$ 
    is $\mathrm{CAT}(0)$.
\end{lem}

\subsection{Converting paths in $\whX_1\cup\whX_2$ to paths in $\bU_{12}$}
\begin{defin}
	\label{def:loop}
Let $\mathcal T=\{\whX_{11},\whX_{12},\whX_{21},\whX_{22}\}$. Let $P$ be a homotopically trivial edge-loop in $\whX_1\cup\whX_2$ that decomposes as a concatenation of edge-paths $P_i$ contained in \emph{hosts} $\mathcal T_i\in \mathcal T$. Let $\widetilde P$ be a lift of $P$ to the universal cover of $\whX_1\cup\whX_2$. Then each $\widetilde P_i$ is contained in a standard subcomplex of face type $\mathcal T_i$ corresponding to a vertex~$x_i$ of $\bU_{12}$. For each $i$, we have that $x_{i+1}$ is equal to, or a neighbour of~$x_i$. Thus $\omega=x_1x_2\cdots$ is a cycle in $\bU_{12}$ \emph{corresponding} to $P$ (or $\widetilde P$). Note that $\omega$ depends on the decomposition of $P$, the choice of the hosts and (least importantly) the choice of the lift of $P$. In practice, we will be looking for a minimal decomposition.
\end{defin}

\begin{defin}
	\label{def:triangle}
	Let $\widetilde P$ be an edge-path in the universal cover of $\whX_1\cup\whX_2$. We say that $\widetilde P$ \emph{starts} (resp.\ \emph{ends}) with a triangle $\sigma\subset \bU_{12}$ if $\widetilde P$ starts (resp.\ ends) with the vertex corresponding to $\sigma$.
\end{defin}

Under the same notation as in Definition~\ref{def:loop}, since non-empty intersections of standard subcomplexes were connected,  we have the following.
\begin{lem}
	\label{lem:link path}
	Suppose that $y_1\cdots y_k$ is a locally embedded edge-path of face type $\mathcal T_{i1}\cdots \mathcal T_{ik}$ in $\lk(x_i,\bU_{12})$ from $x_{i-1}$ to $x_{i+1}$. Then $\widetilde P_i$ is homotopic in its elevation of $\mathcal T_i$, relative to its endpoints, to a concatenation of locally embedded edge-paths $\widetilde P_{i1}\cdots \widetilde P_{ik}$, with $\widetilde P_{ij}$ projecting into $\mathcal T_i\cap\mathcal T_{ij}$. Moreover,
	\begin{enumerate}
		\item if $2\le j\le k-1$, then $\widetilde P_{ij}$ is nontrivial in the sense that its endpoints are distinct,
		\item if $\widetilde P_i$ starts with the triangle $x_ix_{i-1}y_2$ and ends with the triangle $x_iy_{k-1}x_{i+1}$, then $\widetilde P_{i1}$ and $\widetilde P_{ik}$ are trivial.
	\end{enumerate}	
\end{lem}

 Note that an analogous result holds for $\bU_{12}, \widehat X_1\cup \widehat X_2$ replaced by $\Delta, \widehat \Sigma$.

\begin{defin}
	\label{def:noncollapsed}
   Let $\widehat \kappa \colon \od \to \od_{\mathcal H}$ be as in Section~\ref{ss:col}. Let $\widetilde\kappa$ be the induced map between the universal covers. We view $\widehat X$ as a subcomplex $\od_{\mathcal H}$ as in Section~\ref{subsec:central}. Let~$E$ be an elevation of $\whX_1\cup \whX_2$ to the universal cover of $\od_{\mathcal H}$, which is the universal cover of $\whX_1\cup \whX_2$ by Lemma~\ref{lem:transfer1}.
		Consider the subcomplex $\Si^*\subset \Si$ (depending on~$\cH$) that is the union of $2$-cells $C$ with $\ka(C)$ a $2$-cell of $X_1\cup X_2$.    
		Let $\Delta^*$ be the subcomplex of $\Delta$ spanned by the vertices of face type in $\Si^*$ whose corresponding standard subcomplexes map under $\widetilde\kappa$ into $E$. 
		Then $\widetilde\kappa$ induces a simplicial map $\ka^* \colon \Delta^*\to \bU_{12}$. 
	For a vertex $x$ of~$\Delta^*$, we denote $x^{\mathcal H}=\ka^*(x)$. Whenever the dependence on $\cH$ is relevant, we write $\Delta^*_\ch, \kappa_\ch$ instead of $\Delta^*,\kappa$.
	
	We say a that a $2$-cell $C$ of~$\Si^*$ and its image $\ka(C)$ in $X_1\cup X_2$ are \emph{non-collapsed} if $\ka(C)=X_{11},X_{12},$ or $X_{22}$. In particular, $\hk_{|\whC}$ is a homeomorphism.
	A vertex of~$\Delta^*$ (resp.\ $\bU_{12}$) is \emph{non-collapsed} if its face type is non-collapsed. Let $\Delta^\mathrm{nc}$ (resp.\ $\bU^\mathrm{nc}_{12}$) be the subcomplex of $\Delta^*$ (resp.\ $\bU_{12}$) spanned on non-collapsed vertices.
\end{defin}

\begin{lem}
	\label{lem:lift}
 Let $x\in \Delta^*$ be non-collapsed of face type $C$. 
	If $\ka(C)= X_{11}$ or $X_{21}$, then the map $\lk(x,\Delta^*)\to \lk(x^\cH, \bU_{12})$ induced by $\kappa^*$ is an isomorphism. 
	Furthermore, if $\ka(C)= X_{22}$, then the map $\lk(x,\Delta^\mathrm{nc})\to \lk(x^\cH, \bU_{12}^{\mathrm{nc}})$ induced by $\kappa^*$ is an isomorphism. 
\end{lem}

\begin{proof}  By the the description of edges and triangles in $\Delta$ and $\bU_{12}$ in Remark~\ref{rmk:alternative} and Definition~\ref{def:bu}, all the relevant neighbours of $x,x^\cH$ correspond to lines in the isomorphic standard subcomplexes corresponding to $x,x^\cH$. Two such neighbours span an edge of the link exactly when these lines intersect, which is invariant under the isomorphism.
\end{proof}

\subsection{Sub-arrangement of type II}
\label{subsec:AuII}

\begin{defin}
	\label{def:arrangment1}
Let $\ca$ be as at the beginning of the Section~\ref{sec:subarrangement}.
	Consider consecutive vertices $\theta_1,\ldots, \theta_4$ of $\bC$ of types $\hat a,\hat{c},\hat b,\hat c$ in a hyperplane of $\ca$. Let $\mathcal K\subset \ca$ be the collection of hyperplanes passing through at least one of the $\theta_i$. See Figure~\ref{fig:5}, left. The central arrangement $\mathcal K$ in $\mathbb R^3$ is called the \emph{sub-arrangement of type II}. Let $H\in\mathcal K$ be the hyperplane passing through~$\theta_1$ represented as the boundary circle in Figure~\ref{fig:5}, left. We consider the deconing $\mathcal K'=\mathcal K_H$, which is a hyperplane arrangement in $\mathbb R^2$, see Figure~\ref{fig:5}, right.
\end{defin}

\begin{figure}[h]
	\centering
	\includegraphics[scale=0.7]{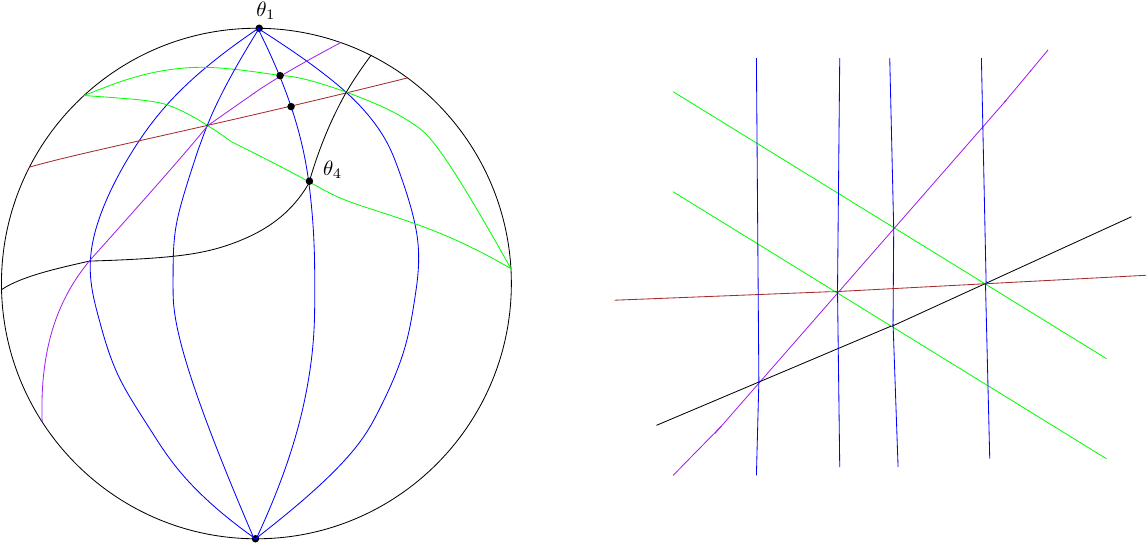}
	\caption{Sub-arrangement of type II}
	\label{fig:5}
\end{figure}

Let $\sX=\Sigma_{\mathcal K'}$ and $\wsX=\widehat \Sigma_{\mathcal K'}$. We view $\od_{\mathcal K'}$ as a subcomplex of $\od_{\mathcal K}$, as in Section~\ref{subsec:central}.  
Denote the four vertical hyperplanes of $\mathcal K'$ by $h_1,h_2,h_3,h_4$ (from left to right in Figure~\ref{fig:6}(I)). Let $\sX_i$ be the union of all the closed faces of $\sX$ that intersect $h_i$. For $2\le i\le 4$, $\sX_i$ is formed of three $2$-cells, denoted from top to bottom by $\sX_{i1},\sX_{i2}$ and~$\sX_{i3}$. The following lemma follows from Lemma~\ref{lem:injective} and \cite[Lem~6.8]{huang2024}.

\begin{lem}
	\label{lem:embed1}
	Inclusions $\wsX_2\cup\wsX_3\subset \od_{\cK}$, $\wsX_3\cup\wsX_4\subset \od_{\cK}$, and $\wsX_2\cup\wsX_3\cup\wsX_4\subset \od_{\cK}$ are $\pi_1$-injective.
\end{lem}

\begin{figure}[h]
	\centering
	\includegraphics[scale=0.88]{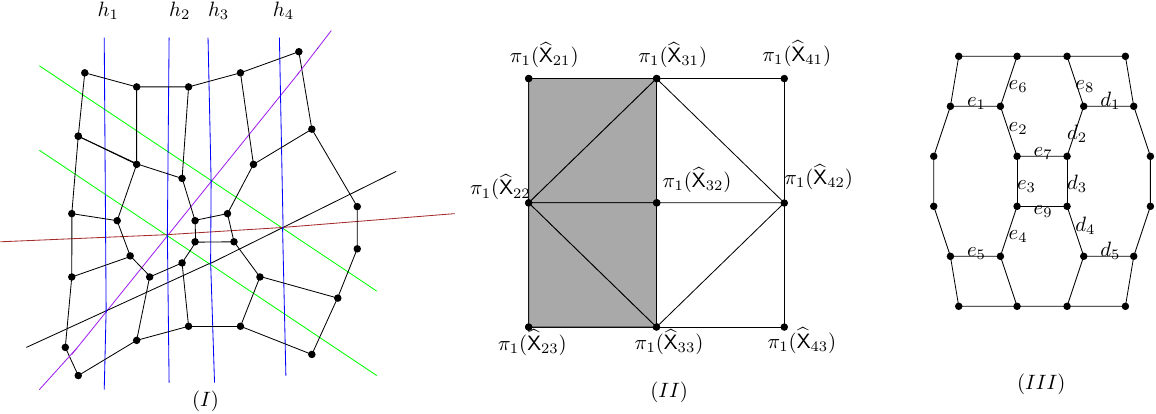}
	\caption{Dual complex}
	\label{fig:6}
\end{figure}

Let $W_{234}$ be the Coxeter group of type $B_3$, and let $\ca_{234}$ be its reflection arrangement. 
Namely, $\ca_{234}$ has the following hyperplanes: $x_i=0$ for $1\le i\le 3$, and $x_i\pm x_j=0$ for $1\le i\neq j\le 3$.
Let $\Si_{234}$ and $\od_{234}$ be the associated dual polyhedron and the Salvetti complex.
Let $\ca'_{234}$ be the deconing of $\ca_{234}$ with respect to $x_1=0$. Then we have isomorphisms of combinatorial complexes
$$
\Si_{\ca'_{234}}\cong \sX_2\cup\sX_3\cup\sX_4\ \mathrm{and}\ \od_{\ca'_{234}}\cong \wsX_2\cup\wsX_3\cup\wsX_4.
$$

Let $\mathcal V_{234}$ be the simple complex of groups with the underlying complex $V_{234}$ described in Figure~\ref{fig:6}(II). Note that $\pi_1\mathcal V_{234}=\pi_1\od_{\ca'_{234}}$. Hence the local groups embed in $\pi_1\mathcal V_{234}$, and so $\mathcal V_{234}$ is developable. Its development is called the \emph{Falk complex}~$\bV_{234}$.  
Let $V_{23}$ be the gray subcomplex of $V_{234}$ in Figure~\ref{fig:6}(II). Let $\mathcal V_{23}$ be the simple complex of groups with the underlying complex $V_{23}$ induced from $\mathcal V_{234}$. Then $\pi_1(\mathcal V_{23})= \pi_1(\wsX_2\cup\wsX_3)$. Let $\bV_{23}$ be the development of $\mathcal V_{23}$. We analogously define~$\bV_{34}$.

Let $\widetilde K_{234}$ (resp.\ $\widetilde K_{23}$) be the universal cover of $\wsX_2\cup\wsX_3\cup\wsX_4$ (resp.\ $\wsX_2\cup\wsX_3$). Vertices of $\bV_{234}$ are in bijective correspondence with the elevations of $\wsX_{ij}$ in $\widetilde K_{234}$, for $2\le i\le 4,1\le j\le 3$, called \emph{standard subcomplexes of face type $\wsX_{ij}$}. The \emph{face type} of the corresponding vertex of $\bV_{234}$ is also $\wsX_{ij}$. We can describe edges and triangles of~$\bV_{234}$ (and of~$\bV_{23}$ and $\bV_{34}$) using these standard subcomplexes in the same way as in Definition~\ref{def:bu}. We equip~$V_{23}$ with the piecewise Euclidean metric of a rectangle with sides of lengths $1$ and~$2$. It pulls back to a piecewise Euclidean metric on~$\bV_{23}$. We define an analogous piecewise Euclidean metric on $\bV_{34}$. 
 We will use the vocabulary from Definitions~\ref{def:loop} and~\ref{def:triangle} in the context of $\bV_{234}$ as well.

\begin{defin}
	 Consider the subcomplex $\Si^*\subset \Si$ (depending on $\cK)$ that is the union of $2$-cells $C$ with $\ka(C)$ a $2$-cell of $\sX_2\cup \sX_3\cup \sX_4$. Using $\widehat \kappa \colon \od \to \od_{\mathcal K}$, we can define a subcomplex $\Delta^*\subset\Delta$ arising from $\Si^*$, and a simplicial map $\ka^* \colon \Delta^*\to \bV_{234}$ in a similar way as in Definition~\ref{def:noncollapsed}. We say a that a $2$-cell $C$ of~$\Si^*$ and its image $\ka(C)$ in $\sX_2\cup \sX_3\cup \sX_4$ are \emph{non-collapsed} if $\ka(C)=\sX_{21},\sX_{31},\sX_{41},\sX_{32},$ or $\sX_{33}$ --- these are the faces for which $\ka|_{C}$ is a homeomorphism.
	A vertex of~$\Delta^*$ (resp.\ $\bV_{234}$) is \emph{non-collapsed} if its face type is non-collapsed.
	\end{defin}

The following has the same proof as Lemma~\ref{lem:lift}.

\begin{lem}
	\label{lem:lift2}
	
Let $x\in \Delta^*$ be non-collapsed of face type $C$.
If $\kappa(C)\neq \sX_{33}$,
then the map $\lk(x,\Delta^*)\to \lk(x^\cK, \bV_{234})$ induced by~$\kappa^*$ is an isomorphism. 
If $\ka(C)= \sX_{33}$, then this map is an isomorphism onto the subcomplex spanned by the vertices of face types $\wsX_{22},\wsX_{32},$ and $\wsX_{42}.$ 
\end{lem}

\section{Filling cycles in $\bV_{234}$}
\label{sec:splitting system}

Let $\Lambda$ be the Coxeter diagram of type~$B_3$, which is the linear graph with consecutive vertices $s_1s_2s_3$ and $m_{s_1s_2}=3,m_{s_2s_3}=4$, and total order $s_1<s_2<s_3$. We shortly write $\Lambda=234$. Let $\bV_{234}, V_{234}$, $\sX_i$, $\wsX_i$ and $\wsX_{ij}$ be as in Section~\ref{subsec:AuII}.

The goal of this section is to establish the properties of certain 8-cycles and 10-cycles in $\bV_{234}$, namely Propositions~\ref{prop:8cycle diagram}, \ref{prop:10cycle diagram1}, and \ref{prop:10cycle diagram2}. We will start with lemmas on vertex links in Section~\ref{subsec:link}, which will be used to study the cycles in Section~\ref{subsec:cycles234}.

\subsection{Vertex links in $\bV_{234}$}
\label{subsec:link}

Let $\sY=\sX_2\cup\sX_3\cup\sX_4$.
We label the edges of $\sY$ (and~$\wsY$) as in Figure~\ref{fig:6}(III). 
Identifying $\od^1_{234}$ with the Cayley graph of $W_{234}$, for each edge $e$ of $\sY$, the two edges of $\widehat e\subset \wsY\subset \od_{234}$ are oriented in opposite directions. 

\begin{lem}
	\label{cor:deadends}
	Let $\omega$ be a locally embedded cycle in the link of a vertex of type $\hat s_3$ (resp.\ $\hat s_1$) in $\bV_{234}$. Then $\omega$ contains at least two vertices of face type $\wsX_{32}$ (resp.\ of type~$\hat s_2$ but not of face type $\wsX_{32}$).  
\end{lem}
\begin{proof} 
	Suppose that $\omega$ lies in the link of a vertex of face type $\whX_{22}$.
	We apply Lemma~\ref{lem:link path} to $\omega$ to produce a locally embedded edge-loop $P$ in $\wsX_{22}$. By Lemma~\ref{lem:falk} applied to $P$, the cycle $\omega$ has at least two vertices of face type in $\{\wsX_{21},\wsX_{23}\}$. Other cases are analogous.
\end{proof}

\begin{lem}
	\label{lem:n=8}
	Let $x$ be a vertex of $\bV_{234}$ of face type $\wsX_{22}$.
	Let $\omega$ be a locally embedded $n$-cycle in $\lk(x,\bV_{234})$. Then $n\ge 8$. Moreover, the equality holds if and only if $\omega$ corresponds, in the sense of Lemma~\ref{lem:link path}, up to a cyclic permutation of vertices, to an edge-loop in $\wsX_{22}$ of form $e^{2k}_1e_2e_3e_4e^{-2k}_5e^{-1}_4e^{-1}_3e^{-1}_2$ or $e^{2k}_1e^{-1}_2e^{-1}_3e^{-1}_4e^{-2k}_5e_4e_3e_2$. An analogous statement holds for $x$ of face type $\wsX_{42}$, with $e_i$ replaced by $d_i$.
\end{lem}

\begin{proof}
	As before, we apply Lemma~\ref{lem:link path} to $\omega$ to produce a locally embedded edge-loop in~$\wsX_{22}$, and then we apply Lemma~\ref{lem:falk} to this edge-loop to deduce that $\omega$ has at least two vertices whose face types belong to $\{\wsX_{21},\wsX_{23}\}$, at least two vertices with face type $\wsX_{31}$, at least two vertices with face type $\wsX_{32}$, and at least two vertices with face type $\wsX_{33}$. Hence $n\ge 8$. When $n=8$, up to a cyclic permutation, the only possible face type of $\omega$ is $\wsX_{21}\wsX_{31}\wsX_{32}\wsX_{33}\wsX_{23}\wsX_{33}\wsX_{32}\wsX_{31}$. Thus the corresponding edge-loop in~$\wsX_{22}$ is of form $e^{2k_1}_1e^*_2e^*_3e^*_4e^{2k_2}_5e^*_4e^*_3e^*_2$, where $k_i$ and $*$ are non-zero integers. Since $\wsX_{22}$ is the cover of the presentation complex of a dihedral Artin group corresponding to the pure Artin group, it remains to apply \cite[Lem~39]{crisp2005automorphisms}.
\end{proof}

We record the following corollary, which will be used in the later sections.
\begin{cor}
	\label{lem:bv23 cat0}
	$\bV_{23}$ (and $\bV_{34}$) are $\mathrm{CAT}(0)$. 
\end{cor}

\begin{proof}
	Since $\bV_{23}$ is the development of a complex of groups, it is simply connected. It remains to that show that each $\lk(x,\bV_{23})$ is $\mathrm{CAT}(1)$, i.e.\ each embedded cycle in $\lk(x,\bV_{23})$ has length $\ge 2\pi$. This is clear if $x$ has face type $\wsX_{21}$, $\wsX_{32}$, or $\wsX_{23}$, as its link is a bipartite graph with edge length $\frac{\pi}{2}$. The case where $x$ has face type $\wsX_{22}$ follows from Lemma~\ref{lem:n=8}. 
It remains to consider $x$ of face type $\wsX_{31}$ (or $\wsX_{33}$). 
By Lemma~\ref{lem:link path} and Lemma~\ref{lem:falk}, any embedded cycle in $\lk(x,\bV_{31})$ has at least two vertices of face type $\wsX_{21}$, 
and at least two vertices of face type $\wsX_{32}$. Any such cycle has $\ge 8$ edges (of length $\frac{\pi}{4}$), as desired. 
\end{proof}

The following lemma has the proof analogous to Lemma~\ref{lem:n=8}.

\begin{lem}
	\label{lem:n=6}
	Let $x$ be a vertex of $\bV_{234}$ of face type $\wsX_{31}$. Let $\omega$ be a locally embedded $n$-cycle in $\lk(x,\bV_{234})$. Then $n\ge 6$. If $n=6$, then $\omega$ corresponds, up to a cyclic permutation of vertices, to an edge-loop in $\wsX_{31}$ of form $e^{2k}_6e_2e_7d^{-2k}_2e^{-1}_7e^{-1}_2, e^{2k}_6e^{-1}_2e^{-1}_7d^{-2k}_2e_7e_2,$ $e^{2k}_2e_7d_2e^{-2k}_8d^{-1}_2e^{-1}_7,$ or $e^{2k}_2e^{-1}_7d^{-1}_2e^{-2k}_8d_2e_7$. 
	An analogous statement holds for $x$ of face type~$\wsX_{33}$.
\end{lem}

\begin{lem}
	\label{lem:12}
	Let $\omega$ be a locally embedded cycle in the link of a vertex of type $\hat s_1$. If $\omega$ contains a subpath of type $\hat s_2 \hat s_3 \hat s_2$, where none of the type $\hat s_2$ vertices are of face type $\wsX_{32}$, then $|\omega|\geq 12$.
\end{lem}
\begin{proof} 
	We can assume that $\omega$ lies in the link of a vertex of face type $\whX_{22}$, and contains a subpath of face type $\whX_{21}\whX_{31}\whX_{21}$. By Lemma~\ref{lem:link path} and Lemma~\ref{lem:falk}, $\omega$ has at least two vertices of face type $\whX_{33}$. So if $|\omega|<12$, then it is a 10-cycle of face type $\wsX_{31}\wsX_{21}\wsX_{31}\wsX_{21}\wsX_{31}\wsX_{32}\wsX_{33}\wsX_{32}\wsX_{33}\wsX_{32}$ or $\wsX_{31}\wsX_{21}\wsX_{31}\wsX_{21}\wsX_{31}\wsX_{32}\wsX_{33}\wsX_{23}\wsX_{33}\wsX_{32}$.
	
	In the first case, by Lemma~\ref{lem:link path} we obtain a locally embedded edge-loop in $\wsX_{22}$ of form $P_{\omega}=e^{2k_1}_1e^{2m}_2e^{2k_2}_1e^*_2e^*_3e^*_4e^*_3e^*_4e^*_3e^*_2$, where  $k_1,m,k_2$ and $*$ are non-zero integers.  
	By considering $\Pi_{\widehat e_1}\colon \wsX_{22}\to\widehat e_1$ (see Definition~\ref{def:retraction}), we obtain $k_1+k_2=0$. By Lemma~\ref{lem:relations}, $e^{2k_1}_1e^{2m}_2e^{2k_2}_1$ is homotopic in $\wsX_{22}$ to $(e^{-1}_2e^{-1}_3e^{-2}_4e^{-1}_3e^{-1}_2)^{k_1}e^{2m}_2(e_2e_3e^2_4e_3e_2)^{k_1}$. Indeed, since $k_1+k_2=0$, by Lemma~\ref{lem:relations}(1), the terms from Lemma~\ref{lem:relations}(2) involving $\Delta$ will cancel. Since $P_\omega$ is homotopically trivial in $\wsX_{22}$, and the inclusion $\widehat e_2\cup\widehat e_3\cup\widehat e_4\to \wsX_{22}$ is $\pi_1$-injective (Lemma~\ref{lem:injective}), we obtain that $P_1=(e^{-1}_2e^{-1}_3e^{-2}_4e^{-1}_3e^{-1}_2)^{k_1}e^{2m}_2(e_2e_3e^2_4e_3e_2)^{k_1}$ and $P_2=e^*_2e^*_3e^*_4e^*_3e^*_4e^*_3e^*_2$ are homotopic in the graph $\Gamma_{234}=\widehat e_2\cup\widehat e_3\cup\widehat e_4$. Given an edge-path $P$ in $\Gamma_{234}$, its \emph{reduced representative} is the unique locally embedded edge-path in $\Gamma_{234}$ homotopic to $P$ in $\Gamma_{234}$. However, the reduced representative of $P_1$ is distinct from $P_2$, which is already reduced, which is a contradiction.
	
	In the second case, by Lemma~\ref{lem:link path} we obtain a locally embedded edge-loop in~$\wsX_{22}$ of form $P_{\omega}=e^{2k_1}_1e^{2m}_2e^{2k_2}_1e^*_2e^*_3e^*_4e^{2k_3}_5e^*_4e^*_3e^*_2$, where  $k_1,m,k_2$ and $*$ are non-zero integers.  By considering $\Pi_{\widehat e_1}\colon\wsX_{22}\to\widehat e_1$, we obtain $k_1+k_2+k_3=0$. By Lemma~\ref{lem:relations}, $P_{\omega}$ is homotopic in $\whX_{22}$ to $$(e^{-1}_2e^{-1}_3e^{-2}_4e^{-1}_3e^{-1}_2)^{k_1}e^{2m}_2(e^{-1}_2e^{-1}_3e^{-2}_4e^{-1}_3e^{-1}_2)^{k_2}e^*_2e^*_3e^*_4(e^{-1}_4e^{-1}_3e^{-2}_2e^{-1}_3e^{-1}_4)^{k_3}e^*_4e^*_3e^*_2.$$
	This edge-loop is homotopically trivial in $\whX_{22}$, hence it is homotopically trivial in~$\Gamma_{234}$. Thus the reduced representative of $$e^*_4e^*_3e^*_2(e^{-1}_2e^{-1}_3e^{-2}_4e^{-1}_3e^{-1}_2)^{k_1}e^{2m}_2(e^{-1}_2e^{-1}_3e^{-2}_4e^{-1}_3e^{-1}_2)^{k_2}e^*_2e^*_3e^*_4$$ is $(e_4e_3e^{2}_2e_3e_4)^{k_3}$, which contradicts $m\neq 0$.
\end{proof}

\begin{lem}
	\label{lem:8466}
	Let $D\to \bV_{234}$ be a minimal disc diagram with an edge $xy$ of type $\hat s_1 \wsX_{32}$ lying in triangles $xyz,xyz'$ of $D$. Then $x,y,z,z'$ cannot be simultaneously interior vertices of $D$ with degrees $8,4,6,6$.
\end{lem}

\begin{proof}
	 For triangles $\delta_1, \delta_2$ of $\bV_{234}$ sharing an edge $\tau$, and corresponding vertices $x_1,x_2$ of $\widetilde K_{234}$, let $P(\delta_1,\delta_2)$ denote the image in $\widehat\Si_{234}$ of the embedded edge-path from $x_1$ to~$x_2$ in the line of $\widetilde K_{234}$ corresponding to $\tau$. We assume without loss of generality that $x$ has face type $\wsX_{22}$. We argue by contradiction and refer to Figure~\ref{fig:3cycle1}(I). 
	Without loss of generality, we can assume that $z'$ has face type $\wsX_{33}$. Then Lemma~\ref{lem:n=8} implies that $z$ has face type $\wsX_{31}$. Moreover, either $P(\delta_0,\delta_1)=e_4$, $P(\delta_1,\delta_2)=e_3$, and $P(\delta_2,\delta_3)=e_2$; or $P(\delta_0,\delta_1)=e^{-1}_4$, $P(\delta_1,\delta_2)=e^{-1}_3$, and $P(\delta_2,\delta_3)=e^{-1}_2$. We only discuss the former case, since the latter is similar. Applying Lemma~\ref{lem:n=6} to the 6-cycles in the link of $z$ and~$z'$ implies that $P(\delta_4,\delta_2)=e_7$ and $P(\delta_1,\delta_5)=e_{9}$. On the other hand, since $\wsX_{32}$ is a product of two oriented circles, 
and the degree of $y$ in $D$ is $4$, $P(\delta_4,\delta_2)=e_7$ implies  $P(\delta_5,\delta_1)=e_{9}$, which is a contradiction.
\end{proof}

\subsection{Filling special cycles in $\bV_{234}$}
\label{subsec:cycles234}
We induce the partial order on the vertex set~$\bV^0_{234}$ from $\Delta^0_{234}$ via the inclusion $\bV_{234}\subset\Delta_{234}$. The map $\pi\colon \Delta_{234}\to \bC$ sends $\bV_{234}$ to~$V_{234}$.
\begin{lem}
	\label{lem:bowtie}
	$\bV^0_{234}$ is bowtie free.
\end{lem} 

\begin{proof}
	Given distinct $x_1,x_2,y_1,y_2\in\bV^0_{234}$ with $x_i\le y_j$ for $1\le i,j\le 2$, there is $z\in \Delta^0_{234}$ such that $x_1,x_2\le z\le y_1,y_2,$ by Theorem~\ref{thm:flag}. Since $\pi(z)$ is a neighbour or equal to each of $\pi(x_1),\pi(x_2),\pi(y_1),\pi(y_2)\in V^0_{234}$, we have $\pi(z)\in V^0_{234}$. Hence $z\in \bV^0_{234},$ as desired.
\end{proof}

\begin{lem}
	\label{lem:restriction order}
	Let $\omega=(x_i)_{i=1}^6$ be a locally embedded cycle in $\bV_{234}$ of type $\hat s_1\hat s_3\hat s_1\hat s_3\hat s_2\hat s_3$ or  $\hat s_1\hat s_3\hat s_2\hat s_3\hat s_2\hat s_3$, angle $\pi$ at $x_5$, and angle $\ge \frac{3\pi}{4}$ at $x_6$. 
	Then $\omega$ is embedded and bounds a diagram in~$\bV_{234}$ as in Figure~\ref{fig:3cycle1}~(II). Furthermore, there is no locally embedded cycle in~$\bV_{234}$ of type 
	$\hat s_2\hat s_3\hat s_2\hat s_3\hat s_2\hat s_3$.
\end{lem}

\begin{proof}
	For the first assertion, by the upward flag property in Theorem~\ref{thm:tripleBn}, there is $z\in\Delta^0_{234}$ of type~$\hat s_3$ that is a common upper bound for $x_1,x_3,$ and $x_5$. If $z\neq x_6$, then by the bowtie free property in Theorem~\ref{thm:tripleBn} applied to $x_1zx_5x_6$, we obtain that  $x_1$ is a neighbour of~$x_5$, contradicting the angle assumption at $x_6$. Thus $z=x_6$. By Lemma~\ref{lem:bowtie} applied to $x_3x_4x_5x_6$ and to $x_1x_2x_3x_6,$ we obtain that $x_3$ is a neighbour of $x_5$, and there is $w\in \bV^0_{234}$ of type $\hat s_2$ that is neighbour of each of $x_1,x_2,x_3,x_6$. Then the first assertion follows.
	The furthermore assertion is proved similarly. \end{proof}

\begin{lem}
	\label{lem:restriction order2}
	Let $\omega=(x_i)_{i=1}^6$ be a cycle in $\bV_{234}$ of type
	$\hat s_1\hat s_3\hat s_1\hat s_3\hat s_1\hat s_3$ such that the face type of $x_3$ is distinct from that of $x_1$ and $x_5$.
	Then 
    \begin{enumerate}
    \item $x_2=x_4$, or
    \item $x_2$ and $x_4$ are connected in $\lk(x_3,\bV_{234})$ by a locally embedded path of length~$2$ with middle vertex of face type $\wsX_{32}$, or
    \item There is a vertex $z$ of $\bV_{234}$ such that the cycle obtained from $\omega$ by replacing~$x_6$ with~$z$ bounds a reduced disc diagram in Figure~\ref{fig:6cycle} on the right, with the interior vertices of face type $\wsX_{32}$. 
    \end{enumerate}
\end{lem}
\begin{proof} Suppose $x_2\neq x_4$. By the upward flag property in Theorem~\ref{thm:tripleBn}, there is a vertex~$z$ of $\Delta_{234}$ of type $\hat s_3$ that is a neighbour of all $x_1,x_3,x_5$. Since $\pi(z)$ is a neighbour of $\pi(x_1)$ and $\pi(x_3)$, it has face type $\wsX_{31}$ or $\wsX_{33}$, 
and so it belongs fo $V^0_{234}$. Consequently, we have $z\in \bV^0_{234}$.
If $z=x_2$ or $x_4$, then we have (2) by Lemma~\ref{lem:bowtie}. 
Otherwise, still by Lemma~\ref{lem:bowtie}, we have the disc diagram in Figure~\ref{fig:6cycle} on the left. If the subdiagram on the right is not reduced, then the two interior vertices are equal and so we have~(2). Otherwise, we have (3).
\end{proof}

\begin{figure}[h]
	\centering
	\includegraphics[scale=0.9]{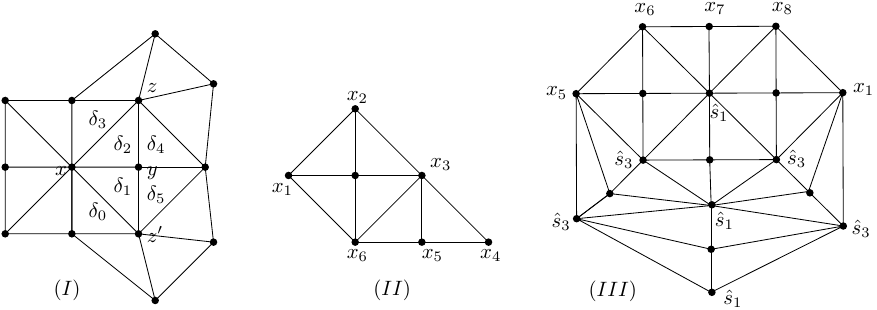}
	\caption{}
	\label{fig:3cycle1}
\end{figure}

\begin{figure}[h]
	\centering
	\includegraphics[scale=1.2]{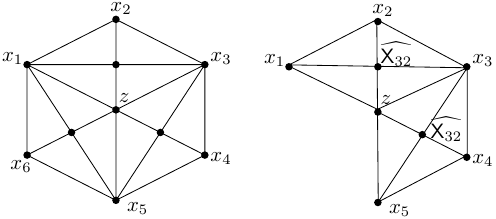}
	\caption{}
	\label{fig:6cycle}
\end{figure}

The following propositions will be proved simultaneously.

\begin{prop}
	\label{prop:8cycle diagram}
	Let $\omega=(x_i)_{i=1}^8$ be a cycle in $\bV_{234}$ of type $\hat s_1\hat s_3\hat s_1\hat s_3\hat s_1\hat s_3\hat s_2\hat s_3$.  
    Suppose that 
	\begin{enumerate}
		\item $\omega$ has angle $\ge \frac{3\pi}{4}$ at $x_6$ and $x_8$,
		\item $\omega$ has angle $\pi$ at $x_7$, and
		\item $\omega$ has angle $\ge \frac{\pi}{2}$ at $x_1$ and $x_5$.
	\end{enumerate}
	Then $\omega$ is embedded, and it bounds a minimal disc diagram $D\to\bV_{234}$ such that $D$ embeds as a subdiagram of Figure~\ref{fig:3cycle1}(III) with $x_5x_6x_7x_8x_1$ mapping to the indicated path.
\end{prop}

\begin{prop}
	\label{prop:10cycle diagram1}
	Let $\omega=(x_i)_{i=1}^{10}$ be a cycle in $\bV_{234}$ of type $\hat s_1\hat s_3\hat s_2\hat s_3\hat s_2\hat s_3\hat s_1\hat s_3\hat s_2\hat s_3$. 
    Then the following properties cannot hold simultaneously: 
	\begin{enumerate}
		\item $\omega$ has angle $\geq \frac{\pi}{2}$ at $x_1$ and $x_7$,
		\item $\omega$ has angle $\pi$ at $x_3,x_4,x_5$, and $x_9$, and 
		\item $\omega$ has angle $\ge \frac{3\pi}{4}$ at $x_2,x_6,x_8$ and $x_{10}$.
	\end{enumerate}
\end{prop}

\begin{prop}
	\label{prop:10cycle diagram2}
	Let $\omega=(x_i)_{i=1}^{10}$ be a locally embedded cycle in $\bV_{234}$ of type $\hat s_3\hat s_2\hat s_3\hat s_2\hat s_3\hat s_2\hat s_3\hat s_1\hat s_3\hat s_1$.  Assume that $x_2$ has face type $\wsX_{32}$, but $x_4$ and $x_6$ do not have face type~$\wsX_{32}$.
	Suppose that
	$\omega$ has angle $\pi$ at $x_2,x_3$, $x_4$ and $x_6$.  
	Then $\omega$ has angle $\frac{\pi}{4}$ at~$x_1$.
\end{prop}

To prove Propositions~\ref{prop:8cycle diagram}, \ref{prop:10cycle diagram1}, and \ref{prop:10cycle diagram2}, we will consider minimal disc diagrams $D\to \bV_{234}$ with boundary $\omega$. First note that $\omega$ is embedded by Lemma~\ref{lem:restriction order}. Thus $D$ is homeomorphic to a disc.

\begin{defin} 
	\label{def:splitting}
	A \emph{splitting system} of a minimal disc diagram $D\to \bV_{234}$ is the preimage under $D\to \bV_{234}$ of all straight line segments in the triangles $xyz$ of $\bV_{234}$ of type $\hat s_1\hat s_2\hat s_3$ joining the midpoint of $xz$
	with the midpoint of $xy$, for $y$ of face type $\wsX_{32}$, or with the midpoint of $yz$, for $y$ not of face type $\wsX_{32}$. Equivalently, we can define the splitting system in the following way. Consider the complex~$V_{234}$ illustrated in Figure~\ref{fig:splitting}(I), where the vertices of type $\hat s_i$ are labelled $i$ and the vertices of face type~$\wsX_{32}$ are circled. Then the splitting system of $D\to \bV_{234}$ is the preimage of the dashed lines  
under the composition $D\to \bV_{234}\to V_{234}$. Note that the splitting system is a union of arcs, starting and ending on~$\partial D$, and (possibly) circles. 
	
	The union of all the edges of $D$ disjoint from the splitting is the \emph{core graph} of $D\to \bV_{234}$. In other words, the core graph of $D\to \bV_{234}$ is the preimage of the thickened lines in Figure~\ref{fig:splitting}(I) under the composition $D\to \bV_{234}\to V_{234}$.
\end{defin}

\begin{figure}[h]
	\centering
	\includegraphics[scale=1.2]{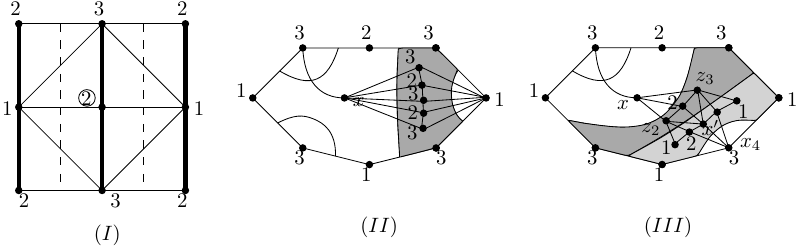}
	\caption{}
	\label{fig:splitting}
\end{figure}

\begin{rem}
	\label{rem:deadends}
	By Lemma~\ref{cor:deadends}, each vertex of the core graph lying in $\mathrm{int}D$ has degree $\geq 2$ in the core graph. In other words, all leaves of the core graph lie in $\partial D$.
\end{rem}

\begin{lem}
	\label{lem:splitting}
	\begin{enumerate}[(i)]
		\item The splitting system contains no circles.
		\item The core graph is a forest. 
		\item Let $x_i$ be a vertex of $\partial D$ of type $\hat s_2$ with distinct neighbours $x_{i-1},x_{i+1}$ both of type $\hat s_1$ or $\hat s_3$. Then there is no arc in the splitting system joining the midpoints of $x_{i-1}x_i$ and $x_ix_{i+1}$. 
		\item Let $x_i$ be a vertex of $\partial D$ of type $\hat s_1$ or $\hat s_3$. 
		If there is an arc $\beta$ in the splitting system joining the midpoints of $x_{i-1}x_i$ and $x_ix_{i+1}$, then the intersection of the core graph with the connected component $R$ of $D\setminus \beta$ containing $x_i$ consists only of $x_i$.
		\item Let $x_{i-1}x_ix_{i+1}$ be a path 
		of $\partial D$ of type $\hat s_3\wsX_{32}\hat s_3$. If there is an arc $\beta$ in the splitting system joining the midpoints of $x_{i-2}x_{i-1}$ and $x_{i+1}x_{i+2}$, then the intersection of the core graph with the connected component of $D\setminus \beta$ containing $x_i$ consists only of $x_{i-1}x_ix_{i+1}$. Similarly, if $x_{i-2}\cdots x_{i+2}$ is a path 
		of $\partial D$ of type $\hat s_3\wsX_{32}\hat s_3\wsX_{32}\hat s_3$, and there is an arc $\beta$ in the splitting system joining the midpoints of $x_{i-3}x_{i-2}$ and $x_{i+2}x_{i+3}$, then the intersection of the core graph with the connected component of $D\setminus \beta$ containing $x_i$ consists only of $x_{i-2}\cdots x_{i+2}$.
		\item If a connected component $Q$ of the complement in $D$ of the splitting system contains exactly two vertices of $\partial D$ and both of them are of type $\hat s_3$, then the intersection of the core graph with $Q$ is an arc ending at these vertices.
	\end{enumerate}
\end{lem}
\begin{proof}
	To prove (i) and (ii), consider an innermost cycle $\beta$ in either the splitting system or the core graph. Note that the open region $R\subset D$ bounded $\beta$ contains a point of the core graph or the splitting system. Since all connected components of the splitting system in $R$ are circles, by the innermost assumption we have that $\beta$ lies in the splitting system, and each connected component of the core graph in $R$ is a tree. This contradicts Remark~\ref{rem:deadends}.
	
	For (iii), assume without loss of generality that $x_i$ has type $\hat s_2$ but not $\wsX_{32}$ and $x_{i-1},x_{i+1}$ are of type $\hat s_3$. If $\beta$ were such an arc, consider the connected component~$R$ of $D\setminus \beta$ containing $x_i$. By (ii), each connected component of the core graph in~$R$ is a tree. By Remark~\ref{rem:deadends}, this connected component equals $x_i$. Hence $x_i$ does not have a neighbour of type $\hat s_1$, which is impossible for $x_{i-1}\neq x_{i+1}$. The proofs of (iv),(v), and (vi) are analogous.
\end{proof}

Lemma~\ref{lem:splitting}(i) gives a bound on the number of the connected components of the splitting system, since each of them is an arc starting and ending in $\partial D$. In Propositions~\ref{prop:8cycle diagram}, \ref{prop:10cycle diagram1}, and \ref{prop:10cycle diagram2}, the number of points in the intersection of the splitting system with $\omega$ is $\leq 10$. Up to a homeomorphism of $D$, each splitting system corresponds to a perfect non-crossing matching of these points. We illustrate the ones satisfying Lemma~\ref{lem:splitting}(ii,iii) in Figure~\ref{fig:3cyclediagram} and Figure~\ref{fig:3cyclediagram1} below.  In Proposition~\ref{prop:8cycle diagram} we consider cases A and F, depending on whether the vertex $x_7$ has type $\wsX_{32}$ (then it is circled) or not. Similarly,
in Proposition~\ref{prop:10cycle diagram1} we distinguish cases B, C, D, G,H, and I, depending on which vertices of $\omega$ are of face type $\wsX_{32}$ (they are circled). 
We will now gradually analyse all these 42 diagrams, excluding most of them.

\begin{figure}
	\centering
	\includegraphics[scale=0.9]{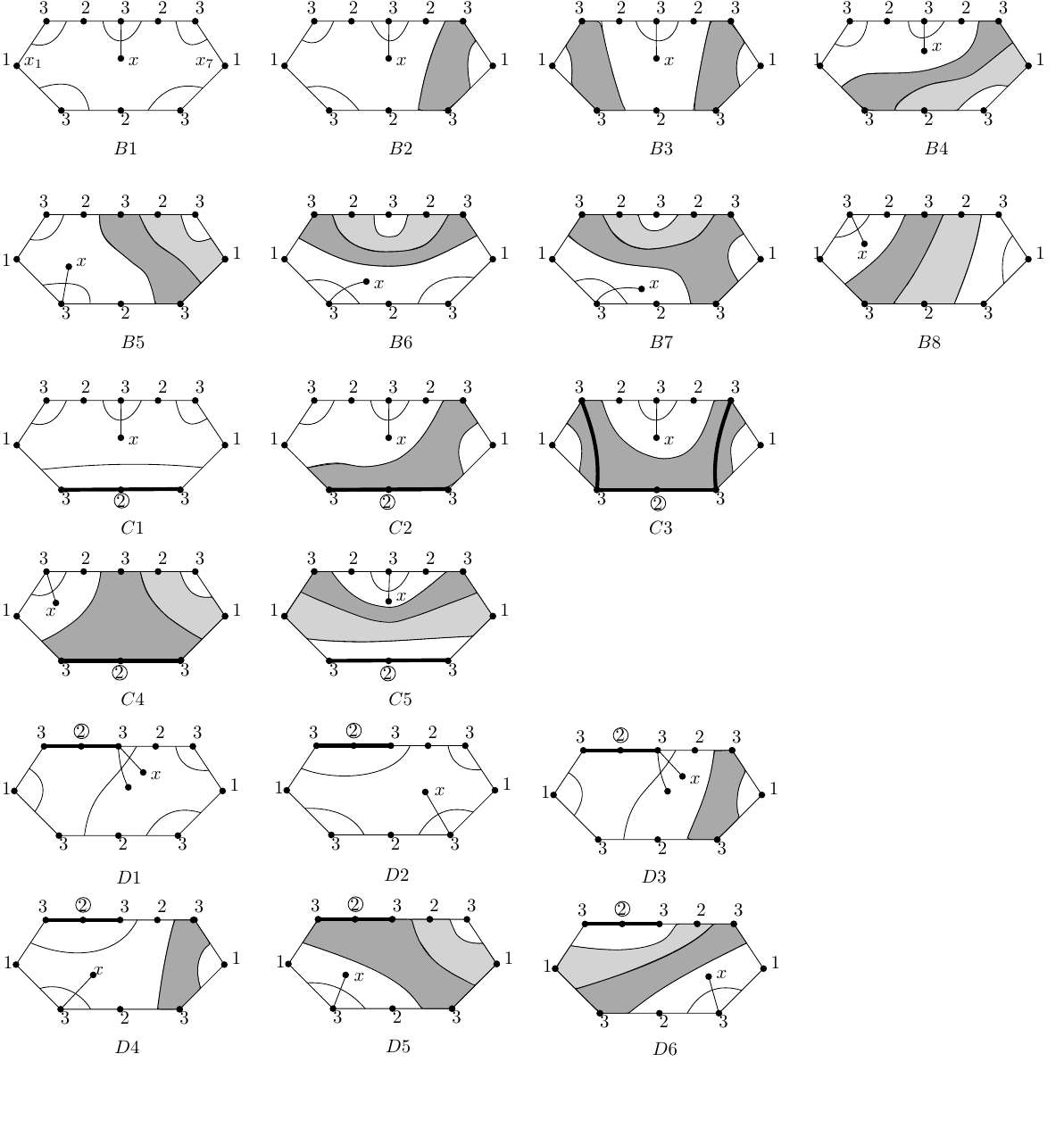}
	\caption{}
	\label{fig:3cyclediagram}
\end{figure}

\begin{figure}
	\centering
	\includegraphics[scale=0.9]{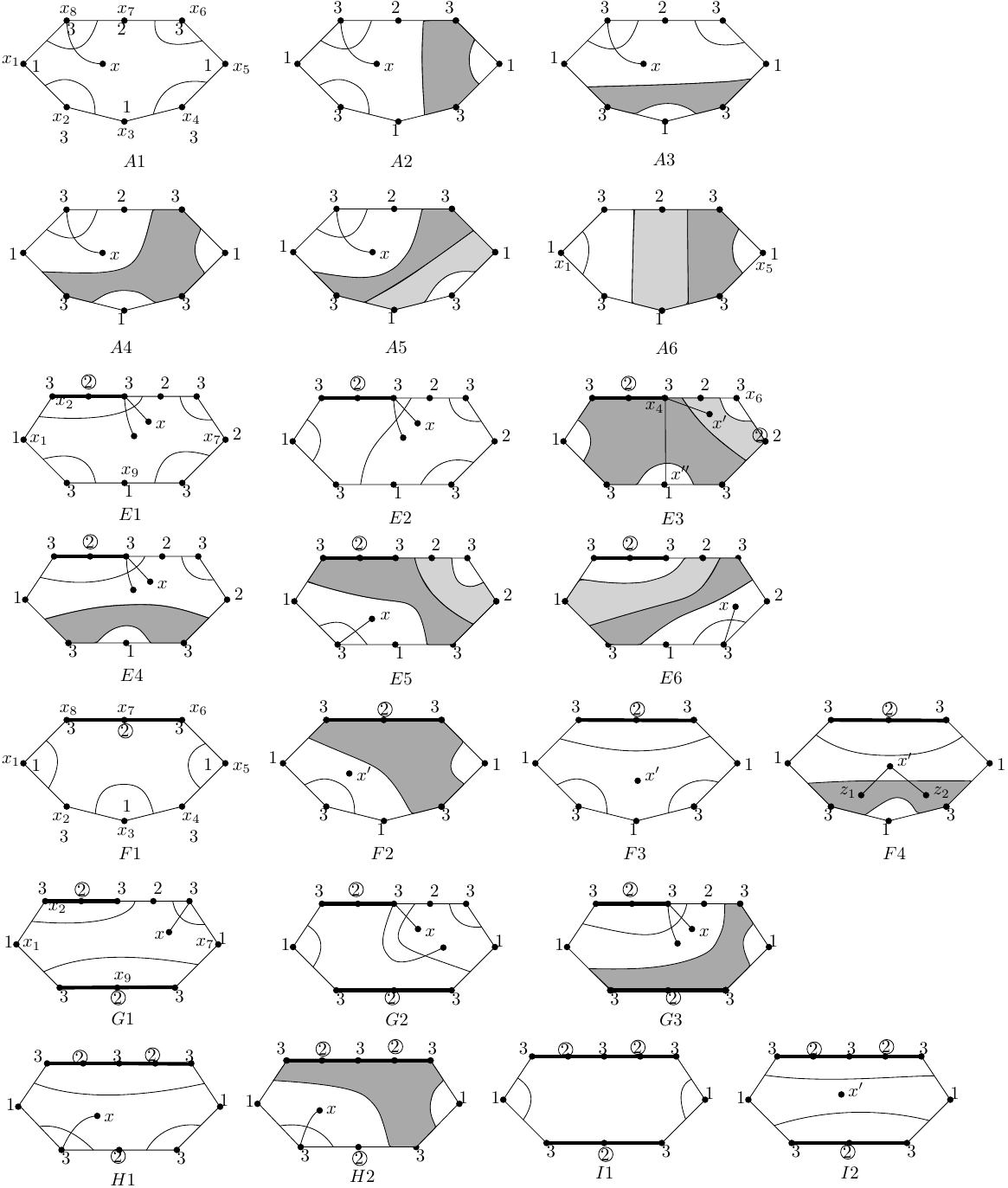}
	\caption{}
	\label{fig:3cyclediagram1}
\end{figure}

\begin{proof}[Proof of Propositions~\ref{prop:8cycle diagram}, \ref{prop:10cycle diagram1}, and \ref{prop:10cycle diagram2}]
In Proposition \ref{prop:10cycle diagram1}, assume by contradiction that all (1)-(3) hold. In Proposition \ref{prop:10cycle diagram2}, assume that $\omega$ has angle $\ge \frac{3\pi}{4}$ at~$x_1$. Consider a minimal disc diagram $D\to \bV_{234}$ with boundary $\omega$. We will reach a contradiction for all the diagrams illustrated in Figures~\ref{fig:3cyclediagram} and~\ref{fig:3cyclediagram1}, except for diagrams~A6, F3, and F4.
	
	In diagram C3, the core graph in the shaded region cannot have a leaf at $x_8$ (or at $x_{10}$). Otherwise, considering the triangle $yx_7x_8$ of~$D$, by assumption~(3) of Proposition~\ref{prop:10cycle diagram1}, the vertex $y$ would not be of face type $\wsX_{32}$. Consequently, the edge $x_8y$ would intersect a splitting curve that also intersects the edge $x_7y$, and so the edge $x_7y$ would intersect two splitting curves, which is a contradiction. Thus, by Lemma~\ref{lem:splitting}(ii) the core graph in the shaded region is of the form indicated by the thickened line in Figure~\ref{fig:3cyclediagram}.

	In most of the diagrams, we indicated an edge (or edges)  $x_ix$ with $x\neq x_{i-1},x_{i+1}$ of type $\hat s_1$, which exists by the assumption on the angles. In diagrams D1, D3, E1, E2, E4, G2, and G3, there are at least two such edges and we denote by $x_4x$ the first one in the order around~$x_4$ indicated in Figure~\ref{fig:3cyclediagram} and Figure~\ref{fig:3cyclediagram1}. Let $x_4y$ be the second such edge. Note that $x$ and $y$ lie on the same side of the arc of the splitting system intersecting the edge $x_4x_5$ (which is clear for diagrams E1, E4, and G3). Otherwise, for diagrams D1, D3, E2, and G2, considering the arc of the splitting system intersecting $x_4y$, we would obtain $y=x_1$, contradicting Lemma~\ref{lem:bowtie}. 
	
	If $x\in \partial D$, then we can appeal to Lemma~\ref{lem:bowtie} and Lemma~\ref{lem:restriction order} to reach the conclusion of Proposition~\ref{prop:8cycle diagram}, or a contradiction with one of the assumptions on the angles. Thus from now on we assume $x\in \mathrm{int} D$.  This excludes diagrams F1 and~I1, where $x_8$ cannot have an interior neighbour $x$ of type $\hat s_1$. By Lemma~\ref{lem:12}, the degree of $x$ is at least $12$. In other words, $x$ has at least $6$ neighbours of type $\hat s_3$. 
	
	By Lemma~\ref{lem:splitting}(iv,v), since any edge between $x$ and a vertex of type $\hat s_3$ intersects an arc of the splitting system, in diagrams A1, B1, C1, D1, D2, E1, E2, G1, G2, and~H1, the vertex~$x$ can have at most $5$ neighbours of type~$\hat s_3$, which is a contradiction. 
	
	In diagram E3, we consider the first two edges $x_3x',x_3x''$ of type $\hat s_3\hat s_1$ in the order around~$x_4$ indicated in Figure~\ref{fig:3cyclediagram1}.  We claim that the vertex $x''$ equals to $x_9$. Otherwise,~$x''$  lies in the light-shaded region. Since all the vertices in the light-shaded region are neighbours of $x_6$, the vertex $x$ has $4$ neighbours, which contradicts the $n\ge 8$ part of Lemma~\ref{lem:n=8}, and justifies the claim. Since $x'$ has at least $8$ neighbours, it has at least $4$ neighbours of type $\hat s_3$, all of which, except for $x_6$, lie in the shaded region. Consequently, $x'$ has a neighbour $z$ of type~$\hat s_3$ in the interior of the shaded region. Since the two neighbours of $z$ in the core graph are neighbours of both~$x''$ and~$x'$, the vertex $z$ has $4$ neighbours, contradicting the $n\ge 6$ part of Lemma~\ref{lem:n=6}. 
	
	In diagrams A2, A3, D3, E4 (resp.\ B2, B3, D4), the vertex $x$ has at least $4$ (resp.\ at least $3$) consecutive type $\hat s_3$ neighbours $z_j$ in one of the shaded regions, labelled according to their order around $x$. By Lemma~\ref{lem:splitting}(vi), only the first and the last of~$z_j$ might lie in $\partial D$. Thus, except for the first and the last one, any $z_j$ has at most two type~$\hat s_1$ neighbours (one of which is $x$), and so  $z_j$ has $4$ neighbours, contradicting Lemma~\ref{lem:n=6}.  See Figure~\ref{fig:splitting}(II) for the A2 case. Similarly, in diagrams~C2 (resp.\ A4, C3, G3, H2), the vertex $x$ has at least $4$ (resp.\ at least~$5$) consecutive type $\hat s_3$ neighbours in the shaded region, one of which contradicts Lemma~\ref{lem:n=6}. Note that in diagram C3 such a type $\hat s_3$ neighbour cannot be simultaneously a neighbour of both $x_1$ and $x_7$, by the shape of the connected component of the core graph we established earlier.
	
	In diagrams A5, B4, B7, C4, D5, E5, (resp.\ B5, B6) the vertex $x$ has at least~$5$ (resp.\ at least $4$) consecutive type $\hat s_3$ neighbours $z_j$ in the shaded region.  Except for the first and the last one, and the second or next-to-last one in diagrams C4, D5, E5, and B7, all $z_j$ lie in the interior of the shaded region, and so there are at least two such consecutive $z_j$. By Lemma~\ref{lem:n=6}, each such $z_j$ has at least $6$ neighbours, so it has at least two type~$\hat s_1$ neighbours in the light-shaded region.
Thus we can find~$x'$ of type $\hat s_1$ in the light-shaded region that is a common neighbour of two such~$z_j$. See Figure~\ref{fig:splitting}(III) for the A5 case. Then $x'$ has at most $3$ neighbours of type~$\hat s_3$ (two of which are among $z_j$), implying that~$x'$ has at most $6$ neighbours, which contradicts Lemma~\ref{lem:n=8}. 
	
	In diagrams C5, D6, and E6, the vertex $x$ has at least $5$ consecutive type $\hat s_3$ neighbours~$z_j$ in the shaded region. Except for the first and the last one, each $z_j$ is interior and by the $n\ge 6$ part of Lemma~\ref{lem:n=6} has at least two type~$\hat s_1$ neighbours in the light-shaded region. We can assume that they all have exactly two such neighbours. Indeed, if $z_j$ had $3$ consecutive type~$\hat s_1$ neighbours in the light-shaded region, then the middle one would have at most 6 neighbours (one in the shaded region, two in the light-shaded region, and $3$ in the thickened part of~$\partial D$), contradicting Lemma~\ref{lem:n=8}. Let $x'\neq x$ be a common neighbour of type $\hat s_1$ of two such~$z_j$. Then $x'$ has degree $8$, which contradicts Lemma~\ref{lem:8466}.
	
	Consider now diagram A6. We claim that $x_1$ has at most one interior neighbour of type $\hat s_3$. Otherwise, if we had such consecutive $z,z'$, by Lemma~\ref{lem:n=6} each of them would have at least two type $\hat s_1$ neighbours in the light-shaded region. If one of them, say $z$, had degree $>6$, then it would have at least $3$ type $\hat s_1$ neighbours in the light-shaded region. Except for the first and last one, any such neighbour~$x'$ would have degree $\geq 12$ by Lemma~\ref{lem:12}. Then one if its type $\hat s_3$ neighbours in the shaded region  would have degree $4$, contradicting Lemma~\ref{lem:n=6}. We can thus assume that the degrees of $z$ and $z'$ are equal to $6$.
	Let $x'$ be the common neighbour of type~$\hat s_1$ of $z,z'$ in the light-shaded region. Then $x'$ has no common neighbours of type~$\hat s_3$ with $x_1$ except for $z,z'$. By Lemma~\ref{lem:8466}, we have that $x'$ has at least $3$ common neighbours of type $\hat s_3$ with $x_5$. This contradicts Lemma~\ref{lem:n=6} for the middle one of these neighbours and justifies the claim. Analogously, $x_5$ has at most one interior neighbour of type~$\hat s_3$. This implies that the length of the core graph component in the light shaded region is $\leq 5$ and so it implies the conclusion of Proposition~\ref{prop:8cycle diagram}.
	
	In diagram B8, the vertex $x$ has at least $5$ consecutive type $\hat s_3$ neighbours $z_j$ in the shaded region. Except for the first and the last one, each $z_j$ is interior and by Lemma~\ref{lem:n=6} has at least two type $\hat s_1$ neighbours in the light-shaded region. We can assume that they all have exactly two such neighbours since otherwise by Lemma~\ref{lem:12} one of these neighbours would have degree $\geq 12$ and it would have a neighbour of type $\hat s_3$  outside the shaded region violating Lemma~\ref{lem:n=6}. Let $x'\neq x$ be a common neighbour of type $\hat s_1$ of two such $z_j$. By Lemma~\ref{lem:8466}, the vertex $x'$ has degree $>8$ and so it contains at least $3$ type $\hat s_3$ neighbours  outside the shaded region. This contradicts Lemma~\ref{lem:n=6} for the middle one of these neighbours.

In diagrams I2, F2, F3, and F4, since the angle at $x_8$ is $\geq \frac{3\pi}{4}$, there is neighbour~$x'$ of $x_8$, of type $\hat s_1$, in the indicated region. In diagram F2 we have that $x'$ is also a neighbour of $x_2$ and so we obtain a contradiction as in diagram H2. In diagrams I2, F3, and F4, by Lemma~\ref{lem:n=8}, the vertex $x'$ has $\geq 4$ neighbours of type $\hat s_3$. In diagram~F3, these can be only $x_2,x_4,x_6,$ and $x_8$, which, by Lemma~\ref{lem:bowtie}, implies the conclusion of Proposition~\ref{prop:8cycle diagram}. In diagram I2, by Lemma~\ref{lem:12}, both $x_8$ and $x_{10}$ are neighbours of~$x'$. Since the same holds for the interior neighbours of type $\hat s_1$ of $x_2$ and $x_6$, we have that both of these neighbours equal $x'$. But then $x_4$ has at most one interior vertex of type $\hat s_1$, contradiction. In diagram F4, $x'$ has at most two neighbours $z_j$ in the shaded region, since otherwise one of them would have only $4$ neighbours, contradicting Lemma~\ref{lem:n=6}. Thus, by Lemma~\ref{lem:n=8}, $x'$ must be also a neighbour of $x_6$ and $x_8$. In particular, $x'$ is the only interior vertex of type $\hat s_1$ in its region. Consequently, if we have $z_1\neq x_2,$ then, by Lemma~\ref{lem:n=6}, $z_1$ is a neighbour of $x_1$ and $x_3$. Analogously, if we have $z_2\neq x_4$, then $z_1$ is a neighour of $x_3$ and $x_5$. By Lemma~\ref{lem:bowtie}, this implies the conclusion of Proposition~\ref{prop:8cycle diagram}.
\end{proof}

\begin{cor} 
\label{cor:remove}
Proposition~\ref{prop:10cycle diagram1} remains valid without assumption (1). 
\end{cor}
\begin{proof} Suppose $x_6=x_8$. If $x_5=x_9$, then, by Lemma~\ref{lem:restriction order}, we have $x_2=x_{10}$ or $x_4=x_{10}$. By Lemma~\ref{lem:bowtie}, this contradicts assumptions (2) or (3). Thus we can assume $x_5\neq x_9$. If $x_2=x_{10}$, then, by Lemmas~\ref{lem:restriction order} and~\ref{lem:bowtie}, this contradicts assumption (2). Thus we can assume $x_2\neq x_{10}$. Let $x_3',x_9'$ be type $\hat a$ neighbours of $x_3,x_9$. The cycle $\omega_8=x_9'x_{10}x_1x_2x'_3x_4x_5x_6$ satisfies assumptions (2) and (3) of Proposition~\ref{prop:8cycle diagram}. For assumption (1), if $\omega_8$ has angle $\frac{\pi}{4}$ at $x_4$, then this contradicts assumption (2) for $\omega$. If $\omega_8$ has angle $\frac{\pi}{4}$ at $x_6$, then by Lemma~\ref{lem:bowtie} applied to $x_9'x_{10}x_1x_2x_3x_4$, we obtain that $x_9'$ is a neighbour of $x_3$, contradicting (2) for $\omega$ as well. Thus  
by Proposition~\ref{prop:8cycle diagram} we have that $\omega_8$ bounds a minimal disc diagram that is a subdiagram of Figure~\ref{fig:3cycle1}(III). By Lemma~\ref{lem:n=8}, the vertex $x_3$ lies in the image of that disc diagram. Thus there is a neighbour of type $\hat a$ of $x_3$ and $x_5$, which again contradicts (2) for $\omega$.
\end{proof}

\section{Critical $8$-cycles}
\label{sec:typeI}
Let $\Lambda$ be the linear graph $abc$ with $m_{ab}=3$, $m_{bc}=5$, as in Section~\ref{sec:subarrangement}.
A \emph{critical $8$-cycle} in $\Delta$ has type $\hat a\hat c\hat a\hat c\hat a\hat c\hat b\hat c$ (or, shortly, $(\hat a\hat c)^3\hat b\hat c$).

\begin{defin}
	\label{def:adm I}
	An embedded critical $8$-cycle $(x_i)$ is \emph{admissible} if $x_7$ is a neighbour of
	\begin{enumerate}
		\item $x_1,x_5$, or
		\item $x_3$.
	\end{enumerate}
\end{defin}

Note that in Case (2), the vertex $x_3$ is a neighbour of both $x_6$ and $x_8$ by Remark~\ref{rem:adj}.

\begin{lem}
	\label{lem:adjacent adm}
	Let $\omega$ be an embedded critical $8$-cycle. Under any of the following conditions, $\omega$ is admissible.
		\begin{enumerate}
		\item \label{38_old1} The vertex $x_3$ is a neighbour of $x_8$ (or $x_6$). 
		\item \label{x247_old2} There is a vertex $x$ of type $\hat a$ that is a neighbour of $x_2,x_4,$ and $x_7$.
		\item \label{replace24_old5} Replacing in $\omega$ the vertex $x_2$ by $z_2$
		results in a critical cycle $\omega_0$ that is not embedded or is admissible.
        \item \label{4x8_old3} There is a vertex $x$ of type $\hat a$ that is a neighbour of $x_4$ and $x_8$ 
		(or of $x_2$ and~$x_6$). 
		\item \label{3yz7_old4} There is a vertex $z$ of type $\hat c$ and a vertex $x$ of type $\hat a$ such that $z$ a neighbour of $x_3$ and $x,$ and $x$ is a neighbour of $x_7$.
		\item \label{replace1_old6} Replacing in $\omega$ the vertex $x_1$ by $x$ 
		results in a critical cycle $\omega_0$ that is not embedded or is admissible.
		\item \label{replace3} Replacing in $\omega$ the vertex $x_3$ by $x$  
		results in a critical cycle $\omega_0$ that is not embedded or is admissible.
        \end{enumerate}
\end{lem}

\begin{proof}
	For (\ref{38_old1}), 
	note that $x_3,x_5,x_7$ are pairwise upper bounded. 
	By Theorem~\ref{thm:flag}, there is $z\in \Delta^0$ of type $\hat c$ that is their common upper bound. 
	Applying the bowtie freeness from Theorem~\ref{thm:flag} to $x_3zx_7x_8$, we obtain that $x_3$ is a neighbour of $x_7$ or $z=x_8$. In the latter case, $x_8$ is a neighbour of $x_5$. Applying the bowtie freeness to $x_5x_6x_7x_8$, we obtain that $x_5$ is a neighbour of  $x_7$, as desired.
		
	For (\ref{x247_old2}), we can assume $x\neq x_1,x_3,x_5$. By Remark~\ref{rem:adj}, $x$ is a neighbour of both $x_8$ and $x_6$. Applying the bowtie freeness 
	to $x_1x_2xx_8$, we obtain their common neighbour $y_1$. Analogously, we obtain a common neighbour $y_2$ of $x,x_4,x_5,x_6$, and a common neighbour $y$ of $x_2,x_3,x_4,x$. Then we have an $8$-cycle $x_8y_1x_2yx_4y_2x_6x_7$ in $\lk(x,\Delta)$. Since $\lk(x,\Delta)$ has girth $\ge 10$, this $8$-cycle is not locally embedded at one of $y_1,x_8,x_7,x_6,y_2$. Since $\omega$ is embedded, this $8$-cycle is not locally embedded at $x_8$ or $x_6$, which implies that $\omega$ is admissible.

  For (\ref{replace24_old5}), if $\omega_0$ is not embedded, then since $\omega$ is embedded, the only possibility is that $z_2$ equals $x_6,x_8,$ or $x_4$.
    In the first two cases, $\omega$ is admissible by (\ref{38_old1}). If $z_2=x_4$, then,
    since $x_1,x_5,x_7$ are pairwise upper bounded, they have a common upper bound~$z$ of type $\hat c$. If $z\neq x_6$, then by the bowtie freeness 
	applied to $x_5x_6x_7z$ we obtain that $x_5$ and $x_7$ are neighbours. If $z=x_6$, then $z\neq x_8$, and analogously $x_1$ and~$x_7$ are neighbours.
	If $\omega_0$ is admissible, then so is $\omega$ since they share the vertices $x_1,x_3,x_5,x_7$.

For (\ref{4x8_old3}), 
since $x,x_5,x_7$ are pairwise upper bounded, 
they have a common upper bound $z$ of type $\hat c$. 
We can assume that $x_5$ and $x_7$ are not neighbours, and so applying the bowtie freeness to $zx_5x_6x_7$, we obtain $z=x_6$, i.e.\ $x$ is a neighbour of~$x_6$. Applying the bowtie freeness to $x_6x_7x_8x$, we obtain that $x_7$ is a neighbour of~$x$. Since $x_1,x_3,x$ are pairwise upper bounded, 
they have a common upper bound~$z_2$ of type $\hat c$. Since $x$ is a neighbour of $z_2,x_4,x_7$, the critical cycle obtained from $\omega$ by replacing $x_2$ with~$z_2$ is either not embedded or is admissible by (\ref{x247_old2}). Thus we are done by~(\ref{replace24_old5}).
	
For (\ref{3yz7_old4}), 
	since $x_3,x_5,x$ are pairwise upper bounded, they
	have a common upper bound $z_4$ of type $\hat c$. The critical cycle obtained from $\omega$ by replacing $x_4$ with $z_4$ is either not embedded or is admissible by (\ref{4x8_old3}). Thus we are done by (\ref{replace24_old5}).

	For (\ref{replace1_old6}), if $\omega_0$ is not embedded, then either $x=x_3$, in which case $\omega$ is admissible by (\ref{38_old1}), or $x=x_5$, in which case $x_5$ is a neighbour of $x_7$ by applying the bowtie freeness to $x_5x_6x_7x_8$.  
	Now assume that $\omega_0$ is embedded. Then $x_7$ is a neighbour of one of $x_3,x_5,x$. In the last case, $x$ is a neighbour of $x_6$, and so $\omega$ is admissible by~(\ref{4x8_old3}). 

 For (\ref{replace3}), if $\omega_0$ is not embedded, then, say, $x=x_1$, and so $\omega$ is admissible by (\ref{4x8_old3}). If $\omega_0$ is admissible and satisfies Definition~\ref{def:adm I}(1), then so does $\omega$. If $\omega_0$ satisfies Definition~\ref{def:adm I}(2), then $\omega$ is admissible by (\ref{x247_old2}).
\end{proof}

In the remaining part of this section, let $\omega=x_1\cdots x_8$ be an embedded critical $8$-cycle. 
Let $w_i, C_i, P_i$ be as in Construction~\ref{def:ncycle}.

The goal of this section is to prove:
\begin{prop}
	\label{prop:Ia}
	Each embedded critical $8$-cycle is admissible.
\end{prop}
Proposition~\ref{prop:Ia} follows from Propositions~\ref{prop:Ia1}, \ref{prop:Ia2}, and~\ref{prop:Ia3}, which are proved in Subsections~\ref{subsec:type1a1}, \ref{subsec:type1a2}, and~\ref{subsec:type1a3}.

\subsection{Case of one decagon}
\label{subsec:type1a1}
\begin{prop}
	\label{prop:Ia1}
	Let $\omega$ be an embedded critical $8$-cycle with only one decagon among the $C_i$. Then $\omega$ is admissible.
\end{prop}

Since $C_1=C_3=C_5$, we have that $C_i$ intersects $C_1$ for all even $i$. We will also assume that $C_7$ intersects $C_1$. Indeed, otherwise
we have $C_6=C_8$, and there is a hyperplane dual to an edge of $C_7$ and disjoint from all the remaining $C_i$. Thus, by Lemma~\ref{lem:retraction property}, we have $\Pi_{\whC_7}(P_i)\subset \whC_7\cap\whC_6$ for $i\neq 7$. Since $\Pi_{\whC_7}(P)$ is homotopically trivial in $\whC_7$, this implies that $P_7$ is homotopic in $\whC_7$ to a path inside $\whC_7\cap \whC_6$. Thus $x_6=x_8$,
contradicting the assumption that $\omega$ is embedded. 

\medskip
\noindent
\underline{Case 1: $C_2=C_4=C_6=C_8$.}  Let $B\neq C_7$ be the other square intersecting $C_1$ and~$C_2$. Let $\ch$ be the type I sub-arrangement of $\ca$ with $\kappa_\ch(C_1)=X_{22}, \kappa_\ch(C_2)=X_{11}$, and $\kappa_\ch(C_7)=X_{21}$. 
Let $\omega^\cH=\ka^*_\ch(\omega)$.  
Since $\omega$ is locally embedded, and $C_i$ are non-collapsed, $\omega^\cH$ is locally embedded by Lemma~\ref{lem:lift}. In particular, the angle of $\omega^\cH$ at $x^\cH_7$ equals $\pi$.
If the angle of $\omega^\cH$ at $x^\cH_6$ or $x^\cH_8$ equals $\frac{\pi}{4}$, then $x^\cH_7$ is a neighbour of  $x^\cH_5$ or~$x^\cH_1$. By Lemma~\ref{lem:lift}, $x_7$ is a neighbour of $x_5$ or $x_1$, as desired. Thus we can assume that the angles of $\omega^\cH$ at $x^\cH_6$ or $x^\cH_8$ are $\geq \frac{3\pi}{4}$.

\begin{figure}[h]
	\centering
	\includegraphics[scale=0.9]{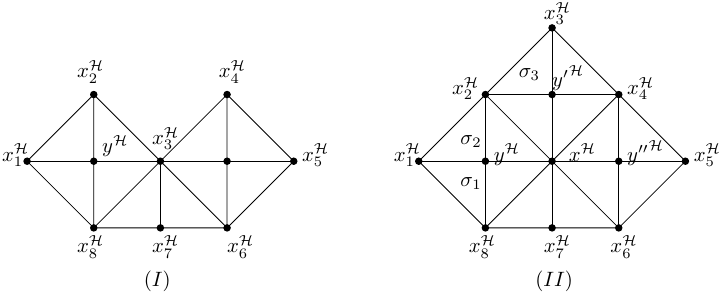}
	\caption{}
	\label{1atypeicase1}
\end{figure}

Since $\bU_{12}$ is $\mathrm{CAT}(0)$ (Lemma~\ref{lem:CAT0}),
$\omega^\cH$ bounds one of the diagrams in Figure~\ref{1atypeicase1}. 
Suppose first that 
$y^\cH$ has face type~$\whX_{21}$. Then, by Lemma~\ref{lem:lift}, we can lift $y^\cH$ to $y\in \lk(x_1,\Delta)^0$ that is a neighbour of $x_2,x_8$. Furthermore, we can lift $x_3^\cH$ (in case~(I)) or~$x^\cH$ (in case~(II)) to $x\in \lk(y,\Delta)^0$ that is a neighbour of $x_2$ and $x_8$. Thus we can replace in~$\omega$ the vertex $x_1$ by  $x$ to form another critical $8$-cycle $\omega_0$. Note that $\omega_0$ is not embedded or is admissible, since $x$ and $x_7$ are neighbours by Lemma~\ref{lem:lift}. Thus by Lemma~\ref{lem:adjacent adm}(\ref{replace1_old6}), $\omega$ is admissible. Hence we can assume that $y^\cH$ has face type~$\whX_{12}$. 
Thus, by Lemma~\ref{lem:lift}, the vertices $x_1$ and $x_7$ are connected in $\lk(x_8,\Delta^*_\cH)$ by a locally embedded path of length three with an interior vertex of face type $B$. 

Let $\mathcal J$ be the type I sub-arrangement of $\ca$ with $\ka_\cj(C_1)=X_{22}, \ka_\cj(C_2)=X_{11},$ and $\ka_\cj(B)=X_{21}$. 
Since $\widehat\ka_\cj$ maps $\widehat C_7 \ \pi_1$-injectively into $\widehat X_{12}$, we have $x^\cj_6\neq x^\cj_8$. Thus we can again assume that $\omega^\cj$ bounds a minimal disc diagram in Figure~\ref{1atypeicase1}, with $\cH$ replaced by~$\cj$. 
As before, we can assume that $y^\cj$ 
has face type $\whX_{12}$. Thus by Lemma~\ref{lem:lift}, the vertices $x_1$ and~$x_7$ are connected in $\lk(x_8,\Delta^*_\cj)$ by a locally embedded path of length three with an interior vertex of face type $C_7$. 
Since $B\neq C_7$, this contradicts the fact that the girth of $\lk(x_8,\Delta^*_\cj)=\lk(x_8,\Delta^*_\cH)$ is $8$. 

\medskip
\noindent
\underline{Case 2: There are exactly two distinct hexagons among the $C_i$.}  Denote these hexa- gons by $D_1$ and $D_2$. To start with, we consider the case where $D_1$ and $D_2$ do not intersect a common square. Then $C_6=C_8$. Assume without loss of generality $C_6=D_1$. Let $B$ be the square such that $B\cap C_1$ and $D_2\cap C_1$ are opposite edges of $C_1$. Note that $B$ intersects~$D_1$. Let $\ch$ be the type I sub-arrangement of $\ca$ with $\ka_\ch(C_1)=X_{22}, \ka_\ch(D_1)=X_{11},$ and $\ka_\ch(B)=X_{21}$. 
Note that the vertices of $\omega$ of face type $D_2$ do not belong to~$\Delta^*_\ch$. 
However, $\ka_\ch(D_2)$ is an edge of $X_{22}$, and thus 
$\hk_\ch(P)\subset \whX_1\cup\whX_2$. Thus declaring that $\widehat \ka_\ch(P_i)$ is hosted by $\whX_{22}$ whenever $i\in \{2,4\}$ satisfies $C_i=D_2$, we obtain $\omega^\cH$ in~$\bU_{12}$ as in Definition~\ref{def:loop}.
If $C_2=D_2$, then $x^\cH_1=x^\cH_3$, and 
if $C_4=D_2$, then $x^\cH_3=x^\cH_5$. Since at least one of $C_2,C_4$ equals $D_2$, we have $|x^\cH_1,x^\cH_5|\le 2\sqrt{2}$. On the other hand, if $\omega^\cH$ has angle $\ge \frac{3\pi}{4}$ at both $x^\cH_6$ and~$x^\cH_8$, since by Lemma~\ref{lem:lift} it also has angle $\pi$ at $x^\cH_7$,  and $\bU_{12}$ is $\mathrm{CAT}(0)$, the endpoints of the path $x^\cH_5x^\cH_6x^\cH_7x^\cH_8x^\cH_1$ are at distance $\ge 4$, which is a contradiction. 
Thus $\omega^\cH$ has angle $\frac{\pi}{4}$ at one of $x^\cH_6,x^\cH_8$. As before we can deduce that $x_7$ is a neighbour of $x_5$ or $x_1$, as desired. 

\begin{figure}[h]
	\centering
	\includegraphics[scale=0.85]{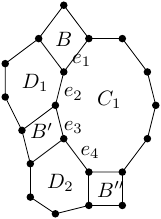}
	\caption{}
	\label{fig:1atypeicase22}
\end{figure}

It remains to assume that $D_1$ and $D_2$ intersect a common square $B'$, see 
Figure~\ref{fig:1atypeicase22}. Assume without loss of generality $C_7\in \{B',B\}$. Let $\mathcal K$ be the type II sub-arrangement of $\ca$ with $\ka_\cK(D_1)=\sX_{31}, \ka_\cK(B')=\sX_{32},$ and $\ka_\cK(C_1)=\sX_{42}$. Let $\omega^\cK=\ka^*_\cK(\omega)\subset \bV_{34}$.
As before, we can assume that the angles at $x^\cK_6,x^\cK_8$ are $\geq \frac{3\pi}{4}$. Since $\bV_{34}$ is $\mathrm{CAT}(0)$ (Corollary~\ref{lem:bv23 cat0}), we obtain that $\omega^\cK$ bounds, up to a symmetry, one of the minimal disc diagrams in Figure~\ref{1atypeicase1}, with $\cH$ replaced by $\cK$, or Figure~\ref{fig:2atypeicase3}. If $\omega$ is not admissible, consider such~$\omega$ with the smallest area of the disc diagram. 

\begin{figure}[h]
	\centering
	\includegraphics[scale=0.85]{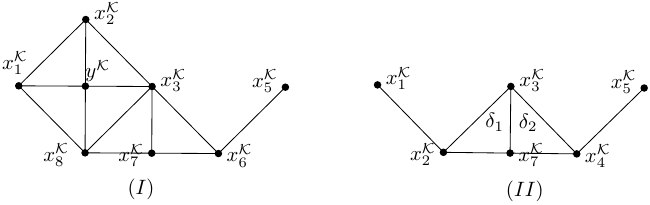}
	\caption{}
	\label{fig:2atypeicase3}
\end{figure}

We claim that $x^\cK_2$ does not have a single interior neighbour $y^\cK$. Indeed, otherwise by Lemma~\ref{lem:lift2} we can lift $y^\cK$ to $y\in \lk(x_2,\Delta)^0$, and we can lift the neighbour of $y^\cK$ opposite to $x_2^\cK$ to $z_2\in \lk(y,\Delta)^0$. The critical cycle obtained from $\omega$ by replacing $x_2$ with~$z_2$ has appropriate diagram with smaller area and so it is not embedded or it is admissible. Thus $\omega$ is admissible by 
Lemma~\ref{lem:adjacent adm}(\ref{replace24_old5}). This justifies the claim and excludes Figures~\ref{1atypeicase1}(I) and~\ref{fig:2atypeicase3}(I). In Figure~\ref{1atypeicase1}(II), by Lemma~\ref{lem:n=8}, we have, up to a symmetry, 
\begin{enumerate}[(i)]
\item
$C_2=C_8=D_1, C_7=B', C_4=C_6=D_2,$ or
\item
$C_6=C_8=D_1, C_7=B, C_2=C_4=D_2.$
\end{enumerate}
In Figure~\ref{fig:2atypeicase3}(II) we must have case (i).

Let $\mathcal H$ be the type I sub-arrangement of $\ca$ with $\ka_\ch(C_1)=X_{22}, \ka_\ch(D_2)=X_{11},$ and $\ka_\ch(B'')=X_{21}$. 
 Note that, to achieve that, we need to reflect Figure~\ref{fig:2} with respect to, say, the line $h_2$, and then apply an orientation-preserving isometry carrying it appropriately to Figure~\ref{fig:1atypeicase22}. In case~(i), let $\omega^\ch=\ka^*_\ch(\omega)\subset \bU_{12}$. Note that $x^\ch_7=x^\ch_8$ and so $x^\ch_1$ is a neighbour of $x^\ch_6$, 
and $x^\ch_2$ has type $\hat b$ (since it has face type $\whX_{12}$). 
 Furthermore, the $6$-cycle $x^\ch_1x^\ch_2x^\ch_3x^\ch_4x^\ch_5x^\ch_6$ is locally embedded at $x_4^\ch, x_5^\ch,$ and $x_6^\ch$ by Lemma~\ref{lem:lift}. Since $\bU_{12}$ is $\mathrm{CAT}(0)$, there is a common neighbour~$y^\cH$ of $x_4^\ch$ and~$x_6^\ch$ in $\lk(x_5^\ch,\bU_{12})$ (otherwise $x^\ch_4x^\ch_5x^\ch_6$ is a geodesic and this 6-cycle cannot `close up' in $\bU_{12}$). See Figure~\ref{fig:1atypeicase23} for all the possible minimal disc diagrams bounded by this 6-cycle.

\begin{figure}[h]
	\centering
	\includegraphics[scale=1.2]{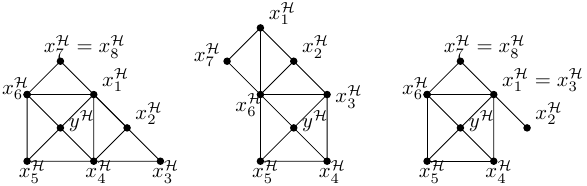}
	\caption{}
	\label{fig:1atypeicase23}
\end{figure}

If $y^\cH$ is of face type $\wX_{21}$, then by Lemma~\ref{lem:lift} we can lift it to a common neighbour $y\in \lk(x_5,\Delta)$ of $x_4$ and $x_6$. If $x^\cH_3$ is a neighbour of $x^\cH_6$, then by Lemma~\ref{lem:lift} $x_3$ is a neighbour of $x_6$, and so $\omega$ is admissible by Lemma~\ref{lem:adjacent adm}(\ref{38_old1}). Otherwise, $x^\cH_4$ is a neighbour of $x^\cH_1$, and so we can lift $x^\cH_1$ to a common neighbour $x\in\lk(y,\Delta)$ of $x_4$ and $x_6$.
Since $x^\cH_1$ is a neighbour of $x_7^\cH$, by Lemma~\ref{lem:lift} we have that $x$ is a neighbour of $x_7$. Hence the cycle obtained from $\omega$ by replacing $x_5$
with $x$ is not embedded or is admissible, and so $\omega$ is admissible by Lemma~\ref{lem:adjacent adm}(\ref{replace1_old6}).

If $y^\cH$ is of face type $\wX_{12}$, then, by Lemma~\ref{lem:lift}:
\begin{enumerate}
\item $x_3$ and $x_5$ are connected in $\lk(x_4,\Delta)$ by a locally embedded path of face type $C_1B'C_1$ or $C_1B'C_1B'C_1$, and
\item $x_4$ and $x_6$ are connected in $\lk(x_5,\Delta)$ by a locally embedded nontrivial path all of whose interior vertices have face types $B,D_1,$ or $B'$.\end{enumerate}

In Figure~\ref{1atypeicase1}(II), $x_3$ and $x_5$ are connected in $\lk(x_4,\Delta)$ by a locally embedded path of face type $C_1B'C_1B''C_1$, which contradicts (1) or Lemma~\ref{lem:falk}. 
To conclude discussing case (i), we consider Figure~\ref{fig:2atypeicase3}(II). Let $f_i=\ka_{\mathcal K}(e_i)$. Then $\hk_{\mathcal K}(P_5)$ is homotopic in $\wsX_{42}$ into $\widehat f_4$. On the other hand, by (2), after possibly replacing the~$w_i$ by equivalent words, we can choose $P_5$ to be an edge-loop in $\widehat e_1\cup \widehat e_2\cup\widehat e_3$ that is homotopically nontrivial in $\whC_1$. This contradicts Lemma~\ref{lem:collapse}. 

In case~(ii), note that $\ka_\ch(B)$ is an edge of $X_{12}$. Thus declaring that $\hk_\ch(P_7)$ is hosted by $\whX_{12}$, 
we obtain $\omega^\cH$ in $\bU_{12}$ as in Definition~\ref{def:loop}. Note that $x^\ch_6=x^\ch_7=x^\ch_8$. The $6$-cycle $x^\ch_1x^\ch_2x^\ch_3x^\ch_4x^\ch_5x^\ch_6$ is locally embedded at $x_2^\ch, x_3^\ch,$ and $x_4^\ch$. As before, since $\bU_{12}$ is $\mathrm{CAT}(0)$, there is a common neighbour $y^\cH$ of $x_2^\ch$ and $x_4^\ch$ in $\lk(x_3^\ch,\bU_{12})$. Moreover, this 6-cycle has angle $\frac{\pi}{2}$ at $x^\cH_4$, in which case $x_2^\ch,x_5^\ch$ are neighbours, or angle $\pi$ at $x^\cH_4$, in which case  $x_1^\ch,x_4^\ch$ are neighbours.

If $y^\cH$ is of face type $\wX_{21}$, then by Lemma~\ref{lem:lift} we can lift it to a common neighbour $y\in \lk(x_5,\Delta)$ of $x_4$ and $x_6$. 
The link of $y^\cH$ contains neighbours $x_2^\ch,x_5^\ch$ or $x_1^\ch,x_4^\ch$. By Lemma~\ref{lem:lift}, $x_2,x_5$ 
are neighbours or $x_1,x_4$ are neighbours, and so $\omega$ is admissible by Lemma~\ref{lem:adjacent adm}(\ref{4x8_old3}).

If $y^\cH$ is of face type $\wX_{12}$, then $x_3$ and $x_5$ are connected in $\lk(x_4,\Delta)$ by a locally embedded path of face type $C_1B'C_1$ or $C_1B'C_1B'C_1$. On the other hand, in Figure~\ref{1atypeicase1}(II), $x_3$ and $x_5$ are connected in $\lk(x_4,\Delta)$ by a locally embedded path of face type $C_1B''C_1B'C_1$, which contradicts Lemma~\ref{lem:falk} as before.

\medskip
\noindent
\underline{Case 3: There are at least three distinct hexagons among the $C_i$.} Hexagons intersecting $C_1$ are \emph{consecutive} if they intersect a common square. We claim that either
\begin{enumerate}[(i)]
\item
among $C_2,C_4,C_6,C_8$ there is $C$ that equals $C_i$ for a unique $i$, and such that for $C',C''$ consecutive with $C$, there is at most one $j$ with $C_j\in \{C',C''\}$ or there are two such $j$, and they equal $6$ and $8$, or 
\item
up to a symmetry, $C_6,C_8$ are as in Figure~\ref{fig:1atypeicase3}(I) and $\{C_2,C_4\}=\{C_6,C\}$.
\end{enumerate}

\begin{figure}[h]
	\centering
	\includegraphics[scale=1]{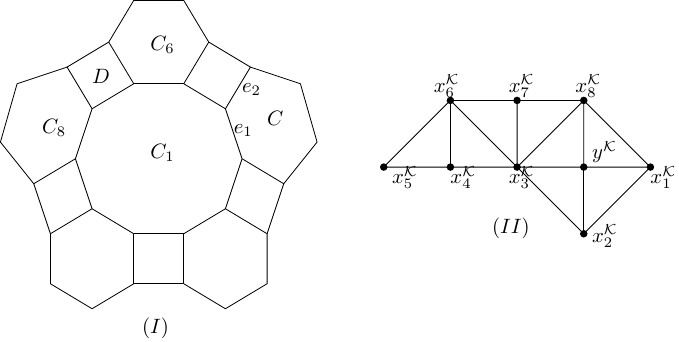}
	\caption{}
	\label{fig:1atypeicase3}
\end{figure}

To justify the claim, start with any $C$ that equals $C_i$ for a unique $i$. If both $C',C''$ equal to some $C_j$, then (i) is satisfied with $C$ replaced by $C'$ or $C''$. If, say, $C''$ is distinct from all $C_j$, but $C$ is not as required in (i), then without loss of generality $C'=C_6$ and $C'=C_2$ or $C_4$. If the remaining $C_k$ is consecutive with $C'$, this brings us to (ii). Otherwise, we have (i) with $C$ replaced by $C_k$. This justifies the claim.

If (i) holds, first note that if $\{C_6,C_8\}\subset\{C',C''\}$, then $C_6=C_8$ since these faces are equal or consecutive.
Furthermore, $P_i$ is not homotopic in $\whC=\whC_i$ to a path in~$\whC_1$, since $\omega$ is embedded. Consequently, by considering  $\Pi_{\whC}(P)=\Pi_{\whC}(P_1)\cdots \Pi_{\whC}(P_8)$, which is homotopically trivial in $\whC$, we deduce, via Lemma~\ref{lem:retraction property}, that there exists $C_j$ consecutive with~$C$. Moreover, we obtain $P_i=e^*$, where $e\subset C$ is the edge contained in the square $B$ intersecting $C$ and $C_j$. If $i=6$ or $8$, then $x_7$ is a neighbour of $x_5$ or $x_1$, as desired.  Otherwise, if, say, $i=2$, then the critical cycle obtained from $\omega$ by denting $x_2$ to $C_j$ (see Definition~\ref{def:dent}) is not embedded or is admissible by Case 2. Consequently, $\omega$ is admissible by Lemma~\ref{lem:adjacent adm}(\ref{replace24_old5}).

If (ii) holds, then by considering $\Pi_{\whC}(P)$, after possibly replacing the~$w_i$ by equivalent words, we can choose $P_i=e^*_2e^*_1e^*_2$ for $C_i=C$. Let $\cK$ be the type II sub-arrangement of $\ca$ with $\ka_\cK(C_6)=\sX_{31},\ka_\cK(D)=\sX_{32},\ka_\cK(C_1)=\sX_{42}$. Since $\ka_\cK(e_1)$ is a vertex, we have $\hk_\cK(P_i)\subset \wsX_{41}$. 
Declaring that $\hk_\cK(P_i)$ is hosted by~$\wsX_{41}$, we obtain $\omega^\cK$ in $\bV_{34}$ as in Definition~\ref{def:loop}. 
As before, we can assume that $\omega^\cK$ has angle $\ge \frac{3\pi}{4}$ at both $x^\cK_6$ and $x^\cK_8$. 
If $C_2=C$, then $x^\cK_2$ has face type~$\wsX_{41}$. Since $x^\cK_8$ has face type~$\wsX_{33}$, we have that $\omega^\cK$ has angle $\ge \frac{3\pi}{4}$ at $x^\cK_1$. Since $\bV_{34}$ is $\mathrm{CAT}(0)$, the endpoints of $x^\cK_5x^\cK_6x^\cK_7x^\cK_8x^\cK_1x^\cK_2$ are at distance $\geq 4$. This contradicts the fact that the path  $x^\cK_2x^\cK_3x^\cK_4x^\cK_5$ has length $2\sqrt{2}+1$.
Thus we have $C_2=C_6$. Since $x^\cK_2$ and~$x^\cK_8$ have distinct face types, $\omega^\cK$ is locally embedded at $x^\cK_1$. Since $\bV_{34}$ is $\mathrm{CAT}(0)$, we obtain that $\omega^\cK$ bounds the minimal disc diagram in Figure~\ref{fig:1atypeicase3}(II). Since $x^\cK_2$ and $x^\cK_8$ have face types $\wsX_{31}$ and $\wsX_{33}$, the vertex $y^\cK$ has face type $\wsX_{32}$. By Lemma~\ref{lem:lift2}, we can lift $y^\cK$ to $y\in \lk (x_2,\Delta)^0$ and $x_8^\cK$ to $z\in \lk(y,\Delta)^0$ of face type $C_8$. Thus the critical cycle obtained from $\omega$ by replacing $x_2$ by $z$ satisfies (i), and so it is not embedded or it is admissible. Consequently, $\omega$ is admissible by Lemma~\ref{lem:adjacent adm}(\ref{replace24_old5}).

\subsection{Case of two decagons}
\label{subsec:type1a2}
\begin{prop}
	\label{prop:Ia2}
	Let $\omega$ be an embedded critical $8$-cycle with exactly two decagons among the $C_i$. Then $\omega$ is admissible.
\end{prop}

\begin{figure}[h]
	\centering
	\includegraphics[scale=0.91]{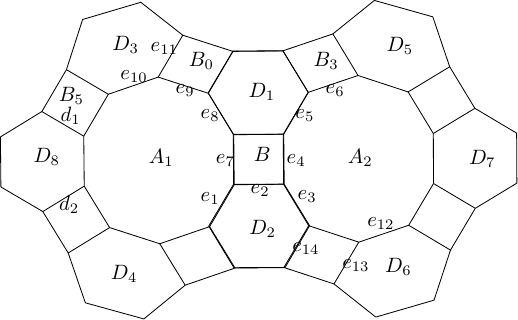}
	\caption{}
	\label{fig:2atypeicase1}
\end{figure}

Denote these decagons $A_1,A_2$. There is a square $B$ intersecting $A_1$ and $A_2$, see Figure~\ref{fig:2atypeicase1}. Note that any $C_i$ intersects $A_1$ or $A_2$. This is immediate for all $C_i$ except for $C_7$, where otherwise we can obtain $x_6=x_8$ by a similar argument as at the beginning of Section~\ref{subsec:type1a1}.

\underline{Case 1: All $C_i$ belong, up to a symmetry, to $\{A_1,A_2,D_1,D_2,B_0,B,B_3\}$}. Let $\Omega$ be the collection of all embedded critical $8$-cycles with all their $C_i$ belonging to the above set, and satisfying an extra condition ($*$):  
\begin{itemize}
 \item if $C_8=D_2$, then $C_1\neq C_3$, and 
 \item if $C_6=D_2$, then $C_3\neq C_5$.
\end{itemize} 
Note that condition ($*$), and hence the class $\Omega$, is invariant under the involution  $\mathcal I$ on the set of critical $8$-cycles sending $x_1\cdots x_8$ to $x_5x_4\cdots x_1x_8x_7x_6$, which still has type $\hat a \hat c \hat a \hat c \hat a \hat c \hat b \hat c$.
 
It suffices to show that the critical cycles in $\Omega$ are admissible. Indeed, if $\omega$ is a critical cycle with, say, $C_1=C_3$ (we cannot have simultaneously $C_3=C_5$) and $C_8=D_2$,  then $C_7=B$. Thus we can apply a symmetry of $\Si$ interchanging $D_1$ with~$D_2$,  which fixes $B$, to send $\omega$ to an element of $\Omega$.

Let $\cK$ be the type II sub-arrangement of $\ca$ with $\ka_\cK(D_1)=\sX_{31},\ka_\cK(B)=\sX_{32},$ $\ka_\cK(A_2)=\sX_{42}.$
Then $\omega^\cK=\ka^*_{\cK}(\omega)\subset \bV_{234}$. Denote $f_i=\ka_\cK(e_i)$. As before, we can assume that the angles at $x^
\cK_6,x^\cK_8$ are $\geq \frac{3\pi}{4}$. If there is $\omega\in \Omega$ that is not admissible, consider such~$\omega$ with $\omega^{\cK}$ bounding a minimal disc diagram in $\bV_{234}$ of the smallest possible area. Proposition~\ref{prop:8cycle diagram} and Lemmas~\ref{lem:restriction order} and~\ref{lem:bowtie} imply that, up to the involution $\mathcal I$, $\omega^\cK$
bounds one of the minimal disc diagrams in Figures~\ref{1atypeicase1} (with $\ch$ replaced by~$\cK$), \ref{fig:2atypeicase3}, or~\ref{fig:2atypeicase22}.

\begin{figure}[h]
	\centering
	\includegraphics[scale=0.9]{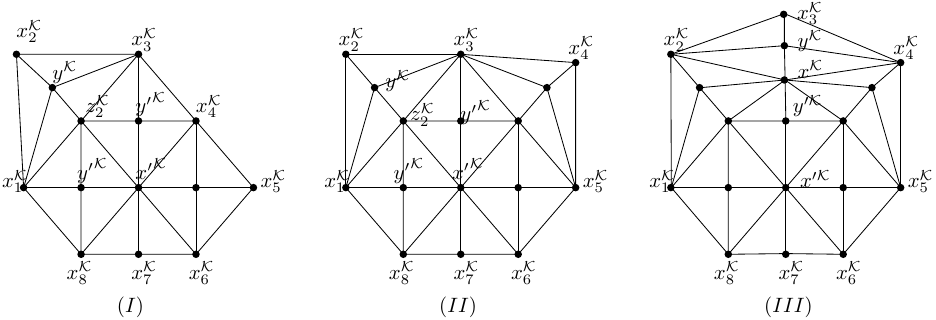}
	\caption{}
	\label{fig:2atypeicase22}
\end{figure}

We claim that $x^\cK_2$ does not have a single interior neighbour $y^\cK$.  To justify the claim, we first verify that such $y^\cK$ would have face type distinct from $\wsX_{23}$ and $\wsX_{43}$. 

For contradiction, suppose that such $y^\cK$ has face type $\wsX_{23}$ or $\wsX_{43}$. In Figure~\ref{fig:2atypeicase22}(I,II), by Lemma~\ref{lem:n=6} applied with $x=z_2^\cK$, the two vertices labelled ${y'}^\cK$ would have face type~$\wsX_{32}$, contradicting Lemma~\ref{lem:n=8} applied with $x={x'}^\cK$. 
In Figures~\ref{1atypeicase1}(I) and~\ref{fig:2atypeicase3}(I), the vertex ${y}^\cK$ is a neighbour of $x_8^\cK$, $x_1^{\cK}$ and $x_3^{\cK}$. Thus $C_1=C_3=A_1$, and $x_8$ has face type $D_2$.
This contradicts condition ($*$).

This confirms that the face type of $y^\cK$ is distinct from $\wsX_{23}$ and $\wsX_{43}$. By Lemma~\ref{lem:lift2}, we can lift $y^\cK$ to $y\in \lk(x_2,\Delta)^0$, and we can lift the neighbour of~$y^\cK$ opposite to~$x_2^\cK$ to $z_2\in \lk(y,\Delta)^0$. The critical cycle obtained from $\omega$ by replacing $x_2$ with~$z_2$ has appropriate diagram with smaller area, and still satisfies condition ($*$), and so it is not embedded or it is admissible. Thus $\omega$ is admissible by 
Lemma~\ref{lem:adjacent adm}(\ref{replace24_old5}). This justifies the claim and excludes Figures~\ref{1atypeicase1}(I), \ref{fig:2atypeicase3}(I), and~\ref{fig:2atypeicase22}(I,II). 

Up to a symmetry, we can assume $A_2=C_i$ for exactly one $i$. 
By considering $\Pi_{\widehat A_2}(P)$, for any choice of the $P_i$ we have $P_i\subset \widehat e_3\cup \widehat e_4\cup \widehat e_5$ or $P_i\subset \widehat e_4\cup \widehat e_5\cup \widehat e_6$, or we are in a special case with $C_7=B_3, C_6=C_8=D_1,$ and $D_2=C_2$ or $C_4$.

In this special case, we have that $x^\cK_2$ or $x^\cK_4$, say $x^\cK_2$, has face type distinct from that of $x^\cK_6,x^\cK_8$, which excludes Figure~\ref{fig:2atypeicase3}(II). In Figure~\ref{fig:2atypeicase22}(III), by Lemma~\ref{lem:n=8}, ${y'}^\cK$ (resp.\ ${y}^\cK$) has face type $\wsX_{43}$ (resp.\ $\wsX_{41}$). 
Thus $x^\cK_2$ and $x^\cK_4$ have the same face type as $x^\cK_6,x^\cK_8$, which is a contradiction. 
 Consequently, we have Figure~\ref{1atypeicase1}(II). Again, by Lemma~\ref{lem:n=8}, 
${y'}^\cK$ has face type~$\wsX_{43}$ and so $x_3^\cK$ has face type~$\wsX_{42}$, which implies $C_3=A_2$, and consequently $C_1=C_5=A_1$.
 After possibly replacing the $w_i$ by equivalent words, we can assume that $P^\cK_1$ starts at $\si_1$ and ends at $\si_2$, and $P^\cK_2$ starts at $\si_2$ and ends at $\si_3$ (see Definition~\ref{def:triangle}). Since $y^\cK$ has face type $\wsX_{32},$ by Lemma~\ref{lem:link path} we have $P^\cK_1=f^*_7$ and $P^\cK_2=f^*_2f^*_3f^*_{14}$ with $*$ non-zero, where $f_j=\kappa_\cK(e_j)$. Since $\Pi_{\widehat f_1}(P^\cK_1)$ is homotopically trivial, we obtain that $\Pi_{\widehat e_1}(P_1)$ is homotopically trivial. Since $x^\cK_2$ is non-collapsed, we have $P_2=e^*_2e^*_3e^*_{14}$ with $*$ non-zero. Thus $\Pi_{\widehat e_1}(P_1P_2)$ is homotopically nontrivial. Then $\Pi_{\widehat D_6}(P)=e^*_{13}e^*_{12}e^*_{13}e^*_{12}$, where the first $e^*_{13}$ and last~$e^*_{12}$ come from $\Pi_{\widehat D_6}(P_1P_2)$ and $\Pi_{\widehat D_6}(P_7)$, which are homotopically nontrivial. This contradicts Lemma~\ref{lem:injective} applied to $\widehat e_{13}\cup\widehat e_{12}\subset\widehat D_6$, 
and finishes the discussion of the special case.

In Figure~\ref{fig:2atypeicase22}(III), by Lemma~\ref{lem:8466} applied with the edge $x^\cK{y'}^{\cK}$ 
playing the role of~$xy$, 
we have that ${y'}^\cK$ is not of face type $\wsX_{32}$. Hence, by Lemma~\ref{lem:n=8}, the vertices $y^\cK, x_7^\cK$ have the same face types  (which are thus $\wsX_{41}$) and the vertices $x^\cK,{x'}^\cK$ have the same face types.
Thus, by Lemma~\ref{lem:n=6}, the vertices $x^\cK_1,x^\cK_5$ have the same face types (distinct from that of $x^\cK,{x'}^\cK$), implying $C_1=C_5$, and consequently $C_3=A_2$. Thus 
$x^\cK_i=x^\cK_3$ has a single interior neighbour~$y^\cK$, which has face type $\wsX_{41}$. 
Since we have covered already the special case, we have $P_3\subset \widehat e_3\cup \widehat e_4\cup \widehat e_5$ or $P_3\subset \widehat e_4\cup \widehat e_5\cup \widehat e_6$.
Furthermore, $P^\cK_3$ has the form $ f_j^*f_6^*f_l^*$, 
where the first and the third term can be removed after possibly replacing the $w_i$ by equivalent words.
By Lemma~\ref{lem:collapse}, we can lift $y^\cK$ to $y\in \lk(x_3,\Delta)^0$. We also lift $x^\cK$ to $x\in \lk(y,\Delta)^0$, which still has face type $A_2$. The critical cycle obtained from $\omega$ by replacing $x_3$ with~$x$ has appropriate diagram with smaller area, and still satisfies condition ($*$),
and so it is not embedded or it is admissible. Thus $\omega$ is admissible by 
Lemma~\ref{lem:adjacent adm}(\ref{replace3}).

In Figure~\ref{1atypeicase1}(II), suppose first $i=3$. If $x_7$ has face type $B_3$, then, by Lemma~\ref{lem:n=8}, ${y'}^\cK$ has face type $\wsX_{43}$. Thus $C_2$ has face type $D_2$, which is the special case that we have already covered. If $x_7$ has face type $B_0$, then ${y'}^\cK$ has face type $\wsX_{23}$, and so
${x}^\cK_3$ has face type $\wsX_{22}$, which is a contradiction. If $x_7$ has face type $B$, then ${y}^\cK$ has face type $\wsX_{21}$ or $\wsX_{23}$, and ${y'}^\cK,x^\cK$ have face types $\wsX_{32},\wsX_{22}$. 
By the same reasoning as in the previous paragraph, we can lift~${y'}^\cK$ to ${y'}\in \lk(x_3,\Delta)^0$. We also lift $x^\cK$ 
to $x\in \lk(y,\Delta)^0$, which has face type~$A_1$.
The critical cycle obtained from $\omega$ by replacing~$x_3$ with~$x$ is not embedded or is admissible by Proposition~\ref{prop:Ia}. Thus $\omega$ is admissible by 
Lemma~\ref{lem:adjacent adm}(\ref{replace3}).

Second, suppose in Figure~\ref{1atypeicase1}(II) that we have $i\neq 3$, say $i=1$. If $y^\cK$ has face type $\wsX_{32}$, then we lift $y^\cK, x^\cK,$ and we proceed as in the last case with~$y$ and~$y'$ interchanged. If $y^\cK$ has face type distinct from $\wsX_{32}$, then, by Lemma~\ref{lem:n=8}, ${y''}^\cK$ has face type $\wsX_{41}$ or $\wsX_{43}$, thus $x^\cK_1,x^\cK_5$ have the same face type, which is a contradiction.

It remains to consider Figure~\ref{fig:2atypeicase3}(II). Since $\omega$ is embedded, by Lemma~\ref{lem:collapse} we have $i=3$.
Furthermore, if $C_7=B$, then the critical cycle obtained from $\omega$ by denting $x_3$ to $A_1$ is not embedded or is admissible by Proposition~\ref{prop:Ia1}. Thus $\omega$ is admissible by Lemma~\ref{lem:adjacent adm}(\ref{replace3}). 

Thus we can assume $C_7=B_3$, where all even $C_i$ equal $D_1$. Moreover, we can choose $P_3=e_6^*$ with $P^\cK_3$ starting at $\delta_1$ and ending at $\delta_2$. Let $H$ be the hyperplane in~$\ca$ dual to $e_8$, and let $K\subset\Si$ be the union of the faces on the side of $H$ containing~$e_6$. We will justify that we can choose all $P_j$ inside $\widehat K$. Indeed, except $j=3$, all $C_j$ intersect $H$. Starting with $j=4$, and applying Lemma~\ref{lem:representative}, $P_j$ is homotopic in $\whC_j$, relative to the endpoints, to $P_{j1}P_{j2}$ with $P_{j1}\subset \widehat K$ and $P_{j2}\subset \whC_{j}\cap\whC_{j+1}$. We replace $P_j$ by $P_{j1}$ and $P_{j+1}$ by $P_{j2}P_{j+1}$. We repeat the same procedure for $j=5,6,7,8,1$.  
Since $\Pi_{\widehat e_8}(P)$ is homotopically trivial, ending this procedure with $j=1$ yields $P_2\subset \widehat K$.

Then $P^\cK_4$ starts at $\delta_2$ and ends at a triangle $\delta_3$ of $\bV_{234}$ containing the edge $x^\cK_4x^\cK_5$. Since $x^\cK_4=x^\cK_6$, we have that $P^\cK_5$ is homotopic in $\wsX_{22}$ to $f^*_8$. Since $P_5\subset \widehat K$, we conclude that $P^\cK_5$ is homotopically trivial in $\wsX_{22}$ and so it both starts and ends at~$\delta_3$. Then $P^\cK_6$ starts at $\delta_3$ and ends at a triangle $\delta_4$ containing the edge $x^\cK_6x^\cK_7$. Note that $\delta_4=\delta_2$, since otherwise $P^\cK_4P^\cK_6$ is a path in $\wsX_{31}$ with nontrivial image under $\Pi_{\widehat f_8}$ contradicting $P_4,P_6\subset \widehat K$. Thus we can assume $P^\cK_6={P^\cK_4}^{-1}$, and so $P_6={P_4}^{-1}$ by Lemma~\ref{lem:lift2}. Analogously, $P_8= P_2^{-1}$, and similarly $P_7=P_3^{-1}$.  
Considering $\Pi_{\widehat A_1}(P)$, and noticing that $\Pi_{\widehat A_1}(P_i)$ are homotopically trivial loops in $\widehat A_1$ for $i\neq 1,5$, we obtain $P_1=P^{-1}_5$.
It follows that $w_1$ commutes with $w_2w_3w_4$. 

Let $\mathsf P$ be the parabolic closure of $w_1$ (i.e.\ the smallest parabolic subgroup of $A_\Lambda$ containing $w_1$, which exists by \cite{cumplido2019parabolic}). Note that $\mathsf P=A_{bc}$, since otherwise we would have $w_1=gb^*g^{-1}$ or $gc^*g^{-1}$ for some $g\in A_{bc}$. Hence there would exist $j\in \{1,7,8,9,10\}$ with $\Pi_{\widehat e_j}(P_1)$ homotopically nontrivial. However, $j\neq 8$ since $P_1\subset \widehat K$. 
Furthermore, since $\Pi_{\widehat D}(P)=k^*_1k^*_2k^*_1k^*_2$ is homotopically trivial in $\widehat D$ (see Figure~\ref{fig:2atypeicase3copy}), hence 
in $\widehat k_1\cup\widehat k_2$ by Lemma~\ref{lem:injective}, and the $*$ over $k_2$ are non-zero, we have that the remaining $*$ are zero, and so $j\neq 10$. 
The remaining~$j$ are excluded since $\omega^\cK$ is not locally embedded at $x^\cK_1$.

\begin{figure}[h]
	\centering
	\includegraphics[scale=0.9]{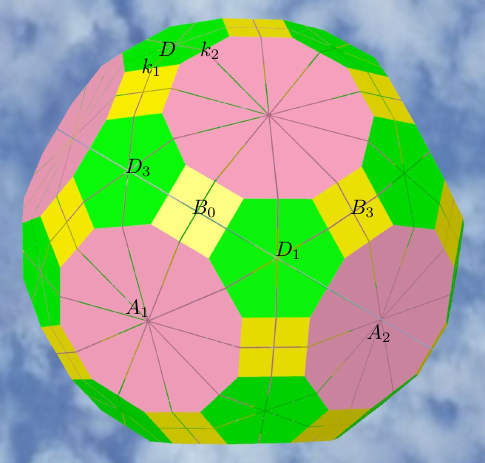}
	\caption{}
	\label{fig:2atypeicase3copy}
\end{figure}

Since 
$w_2w_3w_4$  commutes with $w_1$, it normalises $\mathsf P=A_{bc}$.
By \cite[Thm~5.2 (5)]{paris1997parabolic}, the edge-loop $P_2P_3P_4$ is homotopic in $\od$ to an edge-loop of the form $Q^{n_1}(Q_1Q_2)^{n_2}Q_3$ 
defined as follows. Let 
$p_2$ be the antipodal vertex to the basepoint $p_1$ of $P_1$ in $\Si$. Let $A_1'\subset \Si$ be the opposite face to $A_1$, and let $p_3$ be the projection of $p_1$ onto~$A_1'$. We define $Q_1$ to be the minimal positive path from $p_1$ to $p_3$, $Q_2$ to be the minimal positive path from $p_3$ to~$p_1$, and $Q$ to be the concatenation of a minimal positive path from $p_1$ to $p_2$ and a minimal positive path from $p_2$ to~$p_1$. We allow any $Q_3\subset \widehat A_1$. Since $\Pi_{\widehat e_6}(P_2P_3P_4)=\Pi_{\widehat e_6}(P_7^{-1})$ is homotopically nontrivial, we have $n_1+n_2\neq 0$. On the other hand, for an edge~$e$ whose dual hyperplane does not intersect $A_1,D_1$ or $B_3$, we have that $\Pi_{\widehat e}(P_2P_3P_4)$ is trivial, implying $n_1+n_2=0$, contradiction.

\medskip

\textbf{Suppose now that the condition of Case 1 is not satisfied. We assume without loss of generality $C_5=A_1$.}

\medskip
\noindent
\underline{Case 2: $C_1=C_5$.}
Then all of the $C_i$ intersect $A_1$. Suppose that one of the $C_i$, say~$C_6$, equals~$D_8$. If we also have $C_8=D_8$, then by considering $\Pi_{\widehat A_2}(P)$, we can choose $P_3=e^*_4$. By Proposition~\ref{prop:Ia1}, the critical cycle obtained from $\omega$ by denting~$x_3$ to~$A_1$ is not embedded or is admissible. Thus $\omega$ is admissible by Lemma~\ref{lem:adjacent adm}(\ref{replace3}). If $C_8\neq D_8$, then by considering $\Pi_{\widehat D_8}(P)$, we obtain that  $x_7$ is a neighbour of~$x_5$. Thus we can assume that none of the $C_i$ equals $D_8$.

Up to a symmetry, it remains to assume $C_6=D_3$. Suppose first $C_8=D_1$. By considering $\Pi_{\widehat D_3}(P)$, we can choose $P_6\subset \widehat e_{10}\cup \widehat e_{11}$. Let $\cK$ be the sub-arrangement of $\ca$ as in Case 1. All the vertices of $\omega$ lie in $\Delta^*_\cK$, except for $x_6$.
However, $\ka_\cK(e_{10})$ is a vertex, so we have $\hk_\cK(P_6)\subset \wsX_{21}$. Thus declaring that $\widehat \ka_\cK(P_6)$ is hosted by~$\wsX_{21}$, we obtain $\omega^\cK$ in~$\bV_{234}$ as in Definition~\ref{def:loop}. Note that $x^\cK_6=x^\cK_7$, and so $x^\cK_5$ and~$x^\cK_8$ are neighbours. 

By considering $\Pi_{\widehat A_2}(P)$, we obtain that $P_3$ is contained in $\widehat e_3\cup\widehat e_4\cup\widehat e_5$, which will allow us in a moment to apply Lemma~\ref{lem:collapse} to $P_3$. We apply Lemma~\ref{lem:restriction order2} to the $6$-cycle $x_1^\cK\cdots x_5^\cK x_8^\cK$. If we have Lemma~\ref{lem:restriction order2}(1), i.e.\ $x^\cK_2=x^\cK_4$, then, by Lemma~\ref{lem:collapse}, we obtain $x_2=x_4$, which is a contradiction. If we have Lemma~\ref{lem:restriction order2}(2), i.e.\ $x^\cK_2\neq x^\cK_4$ have a common neighbour $y^\cK$ in $\lk(x^\cK_3,\bV_{234})$, then, as in Case 1, by Lemma~\ref{lem:collapse}, the vertex $y^\cK$ can be lifted to a neighbour $y$ of $x_2,x_4$ in $\lk(x_3,\Delta)$ of face type $B$. Let $x$ be a neighbour of $y$ of face type $A_1$. The critical cycle obtained from $\omega$ by replacing $x_3$ with $x$ is not embedded or is admissible by Proposition~\ref{prop:Ia1}. Thus $\omega$ is admissible by Lemma~\ref{lem:adjacent adm}(\ref{replace3}).

If we have Lemma~\ref{lem:restriction order2}(3), then let ${y'}^\cK$ be the interior vertex of the disc diagram in (3) that is a neighbour of $x^\cK_2$. 
By Lemma~\ref{lem:lift2}, we can lift $y'^\cK$ to a neighbour $y'$ of $x_1,x_3$ in $\lk(x_2,\Delta)$ of face type $B$. Let $z$ be a lift of $z^\cK$ to a neighbour of $x_1,x_3$ in $\lk(y',\Delta)$. Then the critical cycle obtained from $\omega$ by replacing $x_2$ with $z$ is not embedded or is admissible, since it satisfies Lemma~\ref{lem:restriction order2}(2). Hence $\omega$ is admissible by Lemma~\ref{lem:adjacent adm}(\ref{replace24_old5}).

Second, suppose $C_8=D_3$. Let $\ch$ be the type I sub-arrangement of $\ca$ with $\ka_\ch(A_1)=X_{22}, \ka_\ch(D_3)=X_{11},$ and $\ka_\ch(B_5)=X_{21}$. Note that again we need to reflect Figure~\ref{fig:2} before comparing it with 	Figure~\ref{fig:2atypeicase1}.
Then all $x_i$ belong to $\Delta^*_\ch$ (note that $\kappa(A_2)$ is a square of $X_{12}$), except for the ones of face type $D_2$. However, declaring that such $\hk_\ch(P_i)$ are hosted by $X_{22}$, we obtain $\omega^\ch$ in~$\bU_{12}$ as in Definition~\ref{def:loop}. Note that $x_3^\ch$ is a neighbour of $x^\ch_1$ and $x_5^\ch$. Since $\bU_{12}$ is $\mathrm{CAT(0)}$, we have that $\omega^\ch$ has angle $\frac{\pi}{4}$ at $x^\ch_8$ or $x^\ch_6$, and so $\omega$ is admissible as before.

\medskip
\noindent
\underline{Case 3: $C_3=C_5$.} Then none of the $C_i$ equals $D_7$. If one of the $C_i$ equals $D_5$, then $i=8$ and, 
considering $\Pi_{\widehat D_5}(P)$, we deduce that $x_7$ is a neighbour of  $x_1$. Thus we can assume that none of the $C_i$ equals~$D_5$ or $D_6$. We can also assume $C_i\neq D_8$, since otherwise $i=4$, and, by considering $\Pi_{\widehat D_8}(P)$, we can choose $P_4=d_1^*$ or $d_2^*$, say $d^*_1$. Then, by Lemma~\ref{lem:adjacent adm}(\ref{replace24_old5}), we can replace~$\omega$ by the critical cycle obtained by denting $x_4$ to $D_3$.

Up to a symmetry, it remains to assume $D_3\in \{C_4,C_6\}$. If $C_4=D_3\neq C_6$, then $\Pi_{\widehat D_3}(P_6\cup P_7\cup P_8)=e_{10}^*$ or $e_{11}^*$. Thus, by considering $\Pi_{\widehat D_3}(P)$, we can choose $P_4=e_{11}^*$. By denting $x_4$ to $D_1$ 
we obtain a critical cycle that is admissible either by Case 1 or the case $C_6=D_3\neq C_4$, which will be discussed in a moment. 
Thus $\omega$ is admissible by Lemma~\ref{lem:adjacent adm}(\ref{replace24_old5}). If $C_6=D_3\neq C_4$, then, by considering $\Pi_{\widehat D_3}(P)$, we can choose $P_6\subset \widehat e_{10}\cup \widehat e_{11}$. Let $\cK$ be the sub-arrangement of $\ca$ as in Case 1. Declaring that $\widehat \ka_\cK(P_6)$ is hosted by $\wsX_{21}$, we obtain $\omega^\cK$ in~$\bV_{234}$ as in Definition~\ref{def:loop}. As before, $x^\cK_5$ and $x^\cK_8$ are neighbours. We apply Lemma~\ref{lem:restriction order2} to the $6$-cycle $x_5^\cK x_8^\cK x_1^\cK x_2^\cK x_3^\cK x_4^\cK$, with $x_1^\cK$ playing the role of $x_3$,  and we finish as in Case~2.

If $C_6=C_4=D_3$, then let $\ch$ be the type I sub-arrangement of $\ca$ as in Case 2. Then all $x_i$ belong to $\Delta^*_\ch$, except possibly for $x_2$ if it has face type $D_2$. However, declaring then that $\widehat \ka_\ch(P_2)$ is hosted by $X_{22}$, we obtain $\omega^\ch$ in~$\bU_{12}$ as in Definition~\ref{def:loop}. Note that $x^\ch_3$ is a neighbour of $x_1^\ch=x^\ch_8=x^\ch_7$, and hence of $x_6^\ch$. Since $\bU_{12}$ is $\mathrm{CAT(0)}$, we have that the $4$-cycle $x_3^\ch x_4^\ch x_5^\ch x_6^\ch$ has a common neighbour $y^\ch$. If $y^\ch$ has face type $X_{21}$, then by Lemma~\ref{lem:lift} we obtain that $x_3$ and $x_6$ are neighbours, and so $\omega$ is admissible by Lemma~\ref{lem:adjacent adm}(\ref{38_old1}). If $y^\ch$ has face type $X_{12}$, then by Lemma~\ref{lem:lift} $x_4$ can be dented to $D_1$, reducing to the case $C_4\neq D_3$ by Lemma~\ref{lem:adjacent adm}(\ref{replace24_old5}).

\subsection{Case of three decagons}
\label{subsec:type1a3}
\begin{prop}
	\label{prop:Ia3}
	Let $\omega$ be an embedded critical 8-cycle with three decagons
among the $C_i$. Then $\omega$ is admissible.
\end{prop}

\begin{figure}[h]
	\centering
	\includegraphics[scale=0.85]{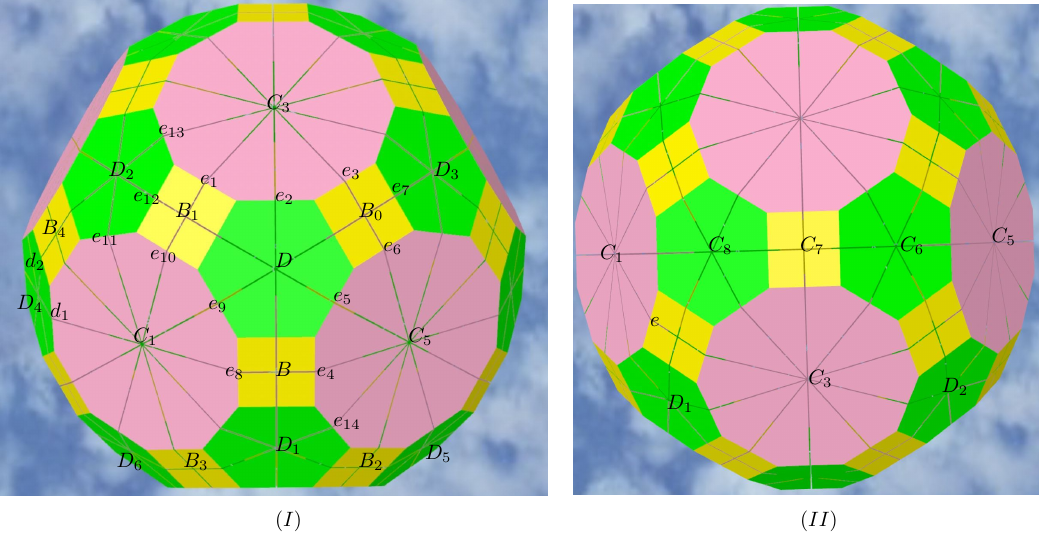}
	\caption{}
	\label{fig:3atypeicase0}
\end{figure}

Up to a symmetry, we have Figure~\ref{fig:3atypeicase0}(I) or (II). 
In (II), we have $C_2\in \{C_8,D_1\}$ and $C_4\in \{C_6,D_2\}$. 
By considering $\Pi_{\widehat C_1}(P)$, we can choose $P_1=e^*$. 
The critical cycle obtained from $\omega$ by denting $x_1$ to $C_3$ is not embedded or is admissible by Proposition~\ref{prop:Ia2}. Thus $\omega$ is admissible by  Lemma~\ref{lem:adjacent adm}(\ref{replace1_old6}). Hence in the remainder of the subsection, we assume~(I).

\medskip
\noindent
\underline{Case 1: One of $C_6,C_8$ belongs to $\{D_2,D_3\}$.} Then we have either $C_8=D_2$ and $C_6=D,$ or $C_6=D_3$ and $C_8=D$. 
 However, in the latter case, reflecting along the hyperplane of $\Si$ intersecting $C_3$ and $D$, and applying the involution $\mathcal I$ from Case~1, brings us to the former one. Thus we can assume
$C_8=D_2$ and $C_6=D$. First, we assume $C_4=D$. By considering $\Pi_{\widehat C_5}(P)$, we can choose $P_5=e^*_6e^*_5e^*_4$. We can assume that the last $*$ is non-zero, since otherwise we can choose $P_5=e^*_6$ and by Lemma~\ref{lem:adjacent adm}(\ref{replace1_old6}) reduce to Proposition~\ref{prop:Ia2} by denting $x_5$ to $C_3$. Then $\Pi_{\widehat D_4}(P_5)=d^*_1$ is nontrivial. Since $\Pi_{\widehat D_4}(P)=d^*_1d^*_2d^*_1d^*_2$ is homotopically trivial in $\widehat D_4$, by Lemma~\ref{lem:injective} we obtain that the $*$ over $d_2$ are zero. In other words, $\Pi_{\widehat e_{13}}(P_2P_3)$ and $\Pi_{\widehat e_{13}}(P_8)$ are homotopically trivial. By considering $\Pi_{\widehat C_3}(P)$, we can assume $P_3\subset \widehat e_1\cup\widehat e_2\cup\widehat e_3$.

Let $\ch$ be the type I sub-arrangement of $\ca$ with $\ka_\ch(C_1)=X_{22}, \ka_\ch(D_2)=X_{11},$ and $\ka_\ch(B_4)=X_{21}.$ 
Here again we reflect Figure~\ref{fig:2}.
Let $P^\ch=\hk_{\ch}(P)$. Note that $\ka_\ch(e_3)$ is a vertex. 
Since $C_4=D$, declaring that $P^\ch_i$ is hosted by $\whX_{12}$ for $i=3,\ldots, 7$, and for $i=2$ when $C_2=D$, we obtain a $4$- or $3$-cycle $\omega^\ch$ in~$\bU_{12}$ as in Definition~\ref{def:loop}. By Lemma~\ref{lem:lift}, $\omega^\ch$ is locally embedded at $x^\ch_1$. Since $\bU_{12}$ is $\mathrm{CAT}(0)$, we obtain that $\omega^\ch$ has angle $\frac{\pi}{4}$ at~$x^\ch_8$. Thus $x^
\ch_1$ and $x^\ch_7$ are neighbours and so $x_1$ and $x_7$ are neighbours. If $C_4=D_3$, then, by considering $\Pi_{\widehat D_3}(P)$, we can choose $P_4=e_7^*$. Thus, by Lemma~\ref{lem:adjacent adm}(\ref{replace24_old5}), we can reduce to the case $C_4=D$ by denting $x_4$ to $D$. 

\medskip
\textbf{In the remaining cases, we assume that none of $C_6,C_8$ belongs to $\{D_2,D_3\}$.}

\medskip
\noindent
\underline{Case 2: $C_2=C_4=D$, and at least one of $C_6,C_8$ equals $D$.} 

\medskip
\noindent
\underline{Case 2.1: $C_7=B$.} By considering $\Pi_{\whC_3}(P)$, we can choose $P_3=e^*_1e^*_2e^*_3$. 
Let $\cK$ be the type II sub-arrangement of $\ca$ with $\ka_\cK(D_1)=\sX_{31},\ka_\cK(B)=\sX_{32},\ka_\cK(C_1)=\sX_{42}$.
 Then $\ka_{\cK}(e_1)$ and $\ka_{\cK}(e_3)$ are vertices. 
Declaring that $P_3^\ch$ is hosted by~$\wsX_{33}$ we obtain a cycle $\omega^\ch$ in~$\bV_{234}$ as in Definition~\ref{def:loop} with $x^\ch_2=x^\ch_3=x^\ch_4$. Note that $\omega^\ch$ is locally embedded at $x^\ch_6,x^
\ch_7,$ and $x^\ch_8$. Thus $\omega^\ch$ is locally embedded at one of $x^\ch_1,x^\ch_5$, say~$x^\ch_1$. 
If $\omega^\ch$ is not locally embedded at $x^\ch_5$ or $x^\ch_4$, then by Lemma~\ref{lem:bowtie} applied to $x^\ch_1x^\ch_6x^\ch_7x^\ch_8$ we have that $x^\ch_1$ is a neighbour of $x^\ch_7$ and so $x_1$ is a neighbour of $x_7$. If $\omega^\ch$ is locally embedded at $x^\ch_5$ and $x^\ch_4$, then by Lemma~\ref{lem:restriction order} applied to  $x^\ch_1x^\ch_2x^\ch_5x^\ch_6x^\ch_7x^\ch_8$ we obtain that $x^\ch_7$ is a neighbour of $x^
\ch_1$ or~$x^\ch_5$ and we finish as before.

\medskip
\noindent
\underline{Case 2.2: $C_7=B_0$ or $B_1$.} By the boldface assumption at the end of Case 1, we have $C_6=C_8=D$. By the argument similar to the one at the beginning of Case~1, we can assume $C_7=B_0$. By considering $\Pi_{\whC_1}(P)$, we can choose $P_1=e^*_{8}e^*_{9}e^*_{10}$. 
We can assume that the $*$ are non-zero, since otherwise we can reduce to Proposition~\ref{prop:Ia2} by denting $x_1$ to $C_3$ or $C_5$. Consider the path $\eta=x_8x_{11}x_{12}x_{13}x_2$ in $\lk(x_1,\Delta)$ of face type $DBDB_1D$ corresponding to $e^*_{8}e^*_{9}e^*_{10}$.
Let $\cK$ be the type II sub-arrangement of~$\ca$ with $\ka_\cK(D)=\sX_{31},\ka_\cK(B_0)=\sX_{32},\ka_\cK(C_3)=\sX_{42}$.
Then the cycle $\omega_0$ obtained from~$\omega$ by replacing $x_8x_1x_2$ with $\eta$ lies in $\Delta^*_\cK$. Let $\omega_0^\cK=\ka^*_\cK(\omega_0)$. 

We first consider the case where $\omega_0^\cK$ has angle $\pi$ at $x^\cK_8$. By considering $\Pi_{\whC_3}(P)$, we obtain $P_3\subset \widehat e_1\cup\widehat e_2\cup\widehat e_3$. Thus by Lemma~\ref{lem:collapse} $\omega_0^\cK$ is locally embedded at~$x^\cK_3$. Analogously $\omega_0^\cK$ is locally embedded at $x^\cK_5$ and so it is locally embedded.
By Proposition~\ref{prop:10cycle diagram2}, with $x^\cK_6x^\cK_7x^\cK_8x^\cK_{11}x^\cK_{12}\cdots$ playing the role of $x_1x_2x_3x_4x_5\cdots$ (note that $x^\cK_7$ has face type $\wsX_{32}$, but $x^\cK_{11}$ and $x^\cK_{13}$ do not have face type $\wsX_{32}$), $\omega_0^\cK$ has angle~$\frac{\pi}{4}$ at $x^\cK_6$, and so $x_5$ and $x_7$ are neighbours.

Second, assume that $\omega_0^\cK$ has angle $\frac{\pi}{2}$ at $x^\cK_8$, and so there is a common neighbour~$x^\cK$ of $x_7^\cK$ and $x^\cK_{11}$ in $\lk(x^\cK_8, \bV_{234})$. By Lemma~\ref{lem:lift2}, we can lift $x^\cK$ to a a common neighbour $x$ of $x_7$ and $x_{11}$ in $\lk(x_8,\Delta)$. Since $x_7x_8x_{11}$ has face type $B_0DB$, we have that $x$ has face type $C_5$. Let $\omega'=x_3x_4x_5x_6xx_{12}x_{13}x_2$. 
If $\omega'$ is not embedded, then either $x=x_5$, in which case $x_5$ and $x_7$ are neighbours, 
or $x_{12}=x_6$, in which case $x_1$ is a neighbour of both $x_6$ and $x_2$ and $\omega$ is admissible by Lemma~\ref{lem:adjacent adm}(\ref{4x8_old3}), or $x_{12}=x_4$, in which case $x_{12}$ is a neighbour of $x_3$ and $\omega$ is admissible by Lemma~\ref{lem:adjacent adm}(\ref{3yz7_old4}). If $\omega'$ is embedded, then by Proposition~\ref{prop:Ia2} it is admissible, and so $x_{13}$ is a neighbour of $x_3$ (other possibilities are excluded since the faces $B_1$ and $C_5$ are disjoint). The critical cycle obtained from $\omega$ by replacing $x_2$ with $x_{12}$ is not embedded or is admissible since it fits the case of  $P_1=e^*_{8}e^*_{9}e^*_{10}$ with one of the $*$ zero.
Thus $\omega$ is admissible by Lemma~\ref{lem:adjacent adm}(\ref{replace24_old5}).

\medskip
\noindent
\underline{Case 3: $C_2=C_4=D$, and none of $C_6,C_8$ equals $D$.} 
If $C_6=D_5$ and $C_8=D_1$, then by considering $\Pi_{\widehat D_5}(P)$ 
we obtain that $x_7$ is is a neighbour of $x_5$. The case $C_6=D_1$ and $C_8=D_6$ is analogous. 
It remains to assume $C_6=C_8=D_1$. If $C_7=B_3$, then by considering $\Pi_{\whC_5}(P)$ 
we can choose $P_5=e^*_6e^*_5e^*_4$. 
 Let $\ch$ be the type I sub-arrangement of $\ca$ with $\ka_\ch(C_1)=X_{22},\ka_\ch(D_1)=X_{11},\ka_\ch(B_3)=X_{21}$. 
Declaring that $P_i^\ch$ are hosted by~$\wX_{12}$ for $i=2,\ldots, 5$, we obtain a cycle $\omega^\ch$ in~$\bU_{12}$ as in Definition~\ref{def:loop} with $x^\ch_2=x^\ch_3=x^\ch_4=x^\ch_5$.
By Lemma~\ref{lem:lift}, $x^\ch_6x^\ch_7x^\ch_8$ is a geodesic. Since $\bU_{12}$ is $\mathrm{CAT}(0)$, $\omega^\ch$ has angle $\frac{\pi}{4}$ at $x^\ch_8$ 
and so $x_7$ and $x_1$ are neighbours as usual. 
The case $C_7=B_2$ is analogous.

We now assume $C_7=B$. By considering $\Pi_{\whC_5}(P)$,  
we can assume $P_5=P_{51}P_{52}P_{53}$ with $P_{51}=e^*_6e^*_5e^*_4$, $P_{52}=e^*_{ 14}$ and $P_{53}=e^*_4$. We assume that $P_{52}$ and $P_{53}$ are nontrivial, since otherwise we can proceed as in the previous paragraph. 
Declaring that $P_i^\ch$ are hosted by~$\wX_{12}$ for $i=2,3,4, 51,53$, and $P^\ch_{52}$ is hosted by $\whX_{11}$, we obtain a cycle $\omega^\ch$ in~$\bU_{12}$ as in Definition~\ref{def:loop}, with $x^\ch_2=x^\ch_3=x^\ch_4=x^\ch_{51}$.
Since $P^\ch_{53}$ is nontrivial, $\omega^\ch$ has angle $\pi$ at $x^\ch_{53}$. Thus $x^\ch_6x^\ch_7x^\ch_8$ and $x^\ch_6x^\ch_{53}x^\ch_{52}$ are geodesics meeting at an angle $\ge \frac{\pi}{2}$. Since $\bU_{12}$ is $\mathrm{CAT}(0)$, 
it follows that the angle of $\omega^\ch$ at $x_8^\ch$ is $\frac{\pi}{4}$, and so $x_7$ and $x_1$ are neighbours as usual.

\medskip
\noindent
\underline{Case 4: $C_2=D_2$ or $C_4=D_3$.} If $C_2=D_2$, then $\Pi_{\widehat D_2}(P_6\cup P_7\cup P_8)=e^*_{12}$ or $e^*_{11}$. Thus,
by considering 
$\Pi_{\widehat D_2}(P)$, 
we can choose $P_2=e^*_{12}$. By Lemma~\ref{lem:adjacent adm}(\ref{replace24_old5}), by denting $x_2$ to $D$ we can reduce to the case where $C_2=D$. 
Analogously, if $C_4=D_3$, then by denting $x_4$ to~$D$ we can reduce to the case where $C_4=D$.

\section{Critical $10$-cycles}
\label{sec:typeII}

Let $\Lambda$ be the linear graph $abc$ with $m_{ab}=3$, $m_{bc}=5$, as in Section~\ref{sec:subarrangement}.
A \emph{critical $10$-cycle} in $\Delta$ has type $\hat a\hat c\hat b\hat c\hat b\hat c\hat a\hat c\hat b\hat c$ (or, shortly, $\hat a\hat c(\hat b\hat c)^2\hat a\hat c\hat b\hat c$).

\begin{defin}
	\label{def:adm II}
	An embedded critical $10$ cycle $(x_i)$ is \emph{admissible} if
	\begin{enumerate}
		\item $x_1$ is a neighbour of $x_3$ or $x_9$, or $x_7$ is a neighbour of $x_5$ or $x_9$, or
		\item there is a vertex of type $\hat a$ that is a neighbour of $x_3,x_5,$ and $x_9$.
	\end{enumerate}
\end{defin}

The goal of this section is to prove:
\begin{prop}
	\label{prop:IIa}
	Each embedded critical $10$-cycle is admissible.
\end{prop}
Proposition~\ref{prop:IIa} follows from Propositions~\ref{prop:IIa1} and~\ref{prop:II2a} below, which are proved in Subsections~\ref{subsec:type2a1} and~\ref{subsec:type2a2}.
In the remaining part of this section, let $\omega=x_1\cdots x_{10}$ be an embedded critical $10$-cycle. 
Let $w_i,C_i,P_i$ be as in Construction~\ref{def:ncycle}.

\begin{lem}
	\label{lem:admissible}
	Let $\omega$ be an embedded critical $10$-cycle. Under any of the following conditions $\omega$ is admissible.
    \begin{enumerate}[(1)]
    \item \label{3x5}
     There is a vertex $x$ of type $\hat a$ that is a neighbour of $x_3$ and $x_5$.
     \item \label{3x9}
     There is a vertex $x$ of type $\hat a$ that is a neighbour of $x_3$ and $x_9$.
     \item \label{replacing1}
     Replacing in $\omega$ the vertex $x_1$ by $x'_1$ results in a critical cycle $\omega_0$ that is not embedded or is admissible.
     \item \label{replacing2}
     Replacing in $\omega$ the vertex $x_2$ or $x_{10}$ results in a critical cycle $\omega_0$ that is not embedded or is admissible.
    \end{enumerate}
\end{lem}

\begin{proof}
	In (1), by Remark~\ref{rem:adj}, $x$ is a neighbour of $x_2$ and $x_6$.
	Let $\omega_8$ be the critical $8$-cycle $x_1x_2xx_6x_7x_8x_9x_{10}.$  By Proposition~\ref{prop:Ia}, $\omega_8$ is not embedded, or is admissible. If $\omega_8$ is not embedded, then, since $\omega$ is embedded, we have $x=x_1$ or $x=x_7$, which implies that $\omega$ is admissible. If $\omega_8$ is embedded and satisfies Definition~\ref{def:adm I}(1), then $\omega$ satisfies Definition~\ref{def:adm II}(1). If $\omega_8$ satisfies Definition~\ref{def:adm I}(2), then $\omega$ satisfies Definition~\ref{def:adm II}(2).
    
    In (2), by Remark~\ref{rem:adj}, $x$ is a neighbour of $x_4$ and $x_8$.
	By Theorem~\ref{thm:flag}, there is a common upper bound $z\in \Delta^0$ of type $\hat c$ of $x,x_5,x_7$. If $z=x_6$, then applying the bowtie freeness from Theorem~\ref{thm:flag} to $xx_4x_5x_6$, we obtain that $x$ is a neighbour of $x_5$, and so $\omega$ satisfies Definition~\ref{def:adm II}(2). If $z\neq x_6$, then applying the bowtie freeness to $zx_5x_6x_7$, we obtain that $x_7$ is a neighbour of $x_5$, and so $\omega$ satisfies Definition~\ref{def:adm II}(1).
    
	In (3), if $\omega_0$ is not embedded, then $x'_1=x_7$. Applying the bowtie freeness from Theorem~\ref{thm:flag} to $x_7x_8x_9x_{10}$, we obtain that $x_7$ is a neighbour of $x_9$.
	Thus we can assume that $\omega_0$ is embedded. The admissibility of  $\omega$ follows immediately from the admissibility of $\omega_0$ unless $x_1'$ is a neighbour of (i) $x_3$ or (ii) $x_9$. 
	
	In (i), let $x_9'$ be a neighbour of type $\hat a$ of $x_9$ (and hence of $x_8$ and $x_{10}$ by Remark~\ref{rem:adj}). Let $\omega_8$ be the critical $8$-cycle $x_7x_8x'_9x_{10}x'_1x_4x_5x_6$. Note that $\omega_8$ is embedded, since otherwise $x_7=x_9'$, which is a neighbour of $x_9$, or $x'_9=x'_1$, and so $\omega$ is admissible by (2) applied with $x=x_1'$. By Proposition~\ref{prop:Ia}, $\omega_8$ is admissible, and so $x_5$ is a neighbour of $x_7,x'_1,$ or $x_9'.$ 
In the second case, $\omega$ is admissible by (1) applied with $x=x_1'$. In the third case,  $\omega$ is admissible by (2) applied with $x=x_9'$.

In (ii), let $x_3'$ be a neighbour of type $\hat a$ of $x_3$. We consider the critical $8$-cycle $\omega_8=x_7x_8x'_1x_2x_3'x_4x_5x_6$, and we proceed analogously as in (i). 

In (4), if $\omega_0$ is not embedded, then, by Lemma~\ref{lem:bowtie}, $\omega$ satisfies Definition~\ref{def:adm II}(1). If $\omega_0$ is admissible and satisfies Definition~\ref{def:adm II}(1) (resp.\ (2)), then $\omega$ satisfies Definition~\ref{def:adm II}(1) (resp.\ (2)).
\end{proof}
\subsection{Case of one decagon}
\label{subsec:type2a1}
\begin{prop}
	\label{prop:IIa1}
	Let $\omega$ be an embedded critical $10$-cycle with $C_1=C_7$. Then $\omega$ is admissible.
\end{prop}

\begin{figure}[h]
	\centering
	\includegraphics[scale=0.7]{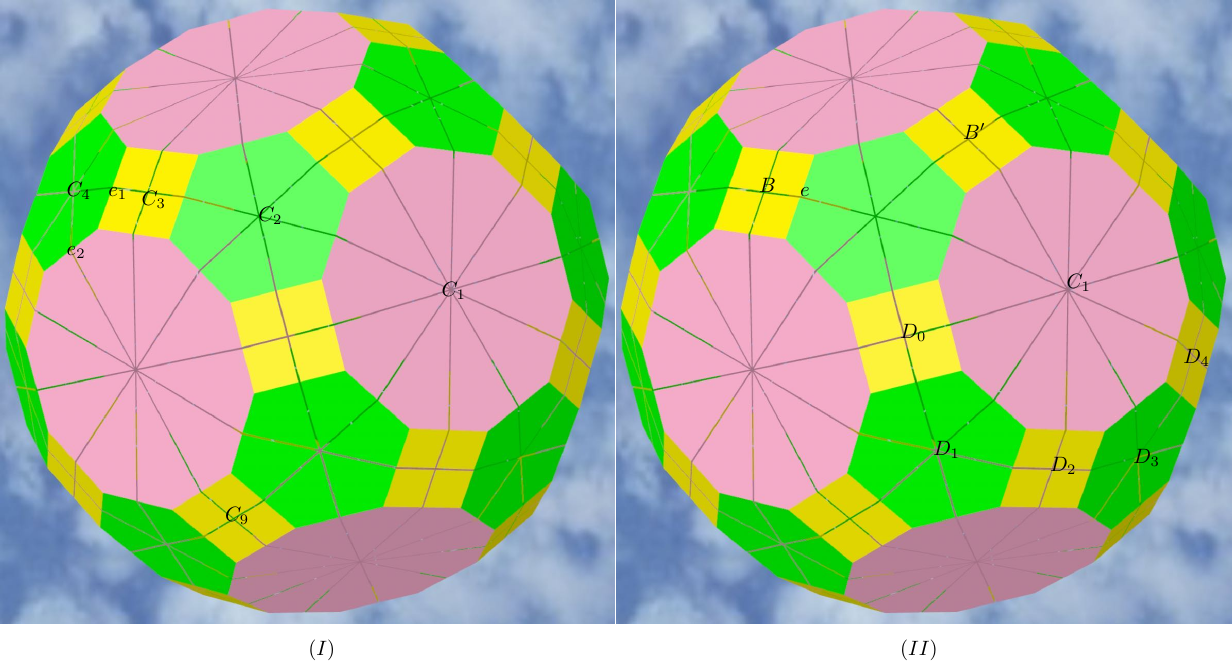}
	\caption{}
	\label{fig:1atypeiicase01}
\end{figure}

In the discussion below, we consider two kinds of symmetries. One kind are the symmetries of $\Sigma$. The second is the involution $\mathcal I'$ on the set of critical $10$-cycles sending $x_1\cdots x_{10}$ to $x_7x_6\cdots x_1x_{10}x_9x_8$, which still has type $\hat a \hat c \hat b \hat c \hat b \hat c \hat a \hat c \hat b \hat c$.

Note that $C_4$ intersects $C_1$. 
Otherwise, up to a symmetry, $C_1$ and $C_4$ are as in Figure~\ref{fig:1atypeiicase01}(I), and we have $\Pi_{C_4}(C_i)\subset e_1$ for $i\neq 4,9$. If $\Pi_{C_4}(C_9)$ is not contained in~$e_1$, then, up to a symmetry, $C_9$ is as in Figure~\ref{fig:1atypeiicase01}(I). 
Since $\Pi_{\whC_4}(P)$ is homotopically trivial in~$\whC_4$,  by Lemma~\ref{lem:retraction property} we obtain that $P_4$ is homotopic in $\whC_4$ to $e^*_1$ or $e^*_1e^*_2e^*_1$. The former is impossible, since it implies $x_3=x_5$. The latter implies that $x_3$ and~$x_5$ have a common neighbour of type $\hat a$, and so $\omega$ is admissible by Lemma~\ref{lem:admissible}(\ref{3x5}).

 We now show that $C_3,C_5,C_9$ intersect $C_1$. Otherwise, up to a symmetry, we can assume that one of them equals $B$ in Figure~\ref{fig:1atypeiicase01}(II). If exactly one of them equals~$B$, then, by Lemma~\ref{lem:retraction property}, we have $\Pi_{\widehat B}(P)=e^*$, implying that $\omega$ is not locally embedded at one of $x_3,x_5,x_9$. Thus we can assume that at least two of them equal $B$. Then, up to a symmetry, we have $C_3=B$ and at least one of $C_5,C_9$ equals $B$. 
 Up to a symmetry of $\Sigma$ interchanging $B'$ with $D_0$, we can also assume $\{C_5,C_9\}\subset\{B,D_0,D_2,D_4\}$.

First assume that $C_9$ equals $B$ or $D_0$. Let $\cK$ be the type~II sub-arrangement of $\ca$ with $\ka_\cK(C_2)=\sX_{31},\ka_\cK(D_0)=\sX_{32},\ka_\cK(C_1)=\sX_{42}$. Since $C_5\neq B'$, we have $\ka_\cK(C_6)=\sX_{31}$ or $\sX_{33}$. By Lemma~\ref{lem:admissible}(\ref{3x5}), we can suppose that $x_3$ and $x_5$ do not have a common neighbour of type $\hat a$. If $\omega$ is not admissible, then, by Lemma~\ref{lem:lift2}, $\omega^\cK=\ka^*_\cK(\omega)$ satisfies the assumptions (2) and (3) of Proposition~\ref{prop:10cycle diagram1}. This contradicts Corollary~\ref{cor:remove}. 

Second, assume $C_9=D_2$, and $C_8=C_{10}=D_1$. Let $\cL$ be the type II sub-arrangement of $\ca$ with $\ka_\cL(D_1)=\sX_{31},\ka_\cL(D_0)=\sX_{32},\ka_\cL(C_1)=\sX_{42}$. 
Note that, to achieve that, we need to reflect Figure~\ref{fig:6} with respect to, say, the line $h_3$, and then apply an orientation-preserving isometry carrying it appropriately to Figure~\ref{fig:1atypeiicase01}(II).
Then $\ka_{\cL}(C_3)$ is an edge.
Declaring that $P_i^\cL$ are hosted by~$\wsX_{33}$ for $i=2,\ldots, 6,$ we obtain a 6-cycle $x^\cL_1x^\cL_2x^\cL_7x^\cL_8x^\cL_9x^\cL_{10}$ in~$\bV_{34}$ as in Definition~\ref{def:loop}. Since $\bV_{34}$ is $\mathrm{CAT}(0)$, the angle at $x^\cL_8$ or~$x^\cL_{10}$ is $\frac{\pi}{4}$, and so, by Lemma~\ref{lem:lift2}, the angle at $x_8$ or $x_{10}$ is $\frac{\pi}{4}$. Thus $\omega$ is admissible.

Third, assume $C_8=D_3$. If $C_{10}\neq D_3$, then, by considering $\Pi_{\widehat C_8}(P)$, we obtain that the angle at $x_8$ is $\frac{\pi}{4}$, and so $\omega$ is admissible. If $C_{10}=D_3$, then let $\cj$ be the type~I sub-arrangement of $\ca$ with $\ka_\cj(C_1)=X_{22},\ka_\cj(D_3)=X_{11},\ka_\cj(D_4)=X_{21}$.  
Note that again we reflect Figure~\ref{fig:2}.
Since $\ka_\cj(C_3)=\ka_\cj(C_2)$ is an edge of $X_{11}$, declaring that $P_i^\cj$ are hosted by~$\wX_{11}$ for $i=1,\ldots, 7$, we obtain a $4$-cycle $x^\cj_7x^\cj_8x^\cj_9x^\cj_{10}$ in~$\bU_{12}$ as in Definition~\ref{def:loop}. Since $\bU_{12}$ is $\mathrm{CAT}(0)$, the angle at $x^\ch_8$ is $\frac{\pi}{4}$, and so, by Lemma~\ref{lem:lift}, the angle at $x_8$ is $\frac{\pi}{4}$ and $\omega$ is again admissible. 

Up to a symmetry, this exhausts all the possibilities, since in particular the case $C_9=D_2$, $C_8=D_1,$ and $C_{10}=D_3$ is sent to the case $C_9=D_2,C_8=D_3,$ and $C_{10}=D_1$ under the involution $\mathcal I'$. Thus we can assume that all the~$C_i$ intersect $C_1$. 
 
Let $\mathcal C$ be the family of hexagons appearing among the $C_i$. Recall that two hexagons intersecting $C_1$ are \emph{consecutive} if they intersect a common square.
We can assume that there is no hexagon $C$ that equals $C_i$ for a unique $i$, and such that for $C',C''$ consecutive with $C$, there is at most one $j$ with $C_j\in \{C',C''\}$ or there are two such $j_1,j_2$, and $\{i,j_1,j_2\}=\{2,4,6\}$. Otherwise, we consider $\Pi_{\whC}(P)$, 
and the argument is similar to Case 3(i) of Proposition~\ref{prop:Ia1}, except no denting is necessary here and we possibly need Lemma~\ref{lem:admissible}(\ref{3x5}) when $i=4$.

We claim that then $\mathcal C$ is contained in a sequence $F_1,F_3,F_5$ of consecutive hexagons, see Figure~\ref{fig:1atypeiicase3}, left. Indeed, if $|\mathcal C|=5$, then $C_2,C_4,C_6$ are distinct and consecutive, and we can take $i=4$ above. If $|\mathcal C|=4$, then we can take $C_i$ to be one of the two hexagons that have only one consecutive hexagon among the $C_j$. If $\mathcal C$ consists of three hexagons that are not consecutive, then two of them are consecutive. If this consecutive pair equals $\{C_8,C_{10}\}$, then we can take $i=8$ or $10$. If this pair equals $\{C_2,C_4, C_6\}$, then we can take $i=2,4,$ or $6$. This justifies the claim.

\medskip
\noindent
\underline{Case 1: $C_2,C_4,C_6$ are distinct.} Since $\{C_8,C_{10}\}\neq\{F_1,F_5\}$, 
 using a symmetry of $\Sigma$ we can assume $F_1\neq C_8,C_{10}$ and $F_2\neq C_9$. Using the involution $\mathcal I'$, we can assume $F_1=C_2$. Furthermore, one of 
$C_8,C_{10}$ equals $F_3$, since otherwise we could take $C_i=C_2$ above. In particular, we have $C_9\neq F_6$.

Let~$\cK$ be the type II sub-arrangement of $\ca$ with $$\ka_\cK(F_1)=\sX_{31},\ka_\cK(F_2)=\sX_{32},\ka_\cK(C_1)=\sX_{42}.$$
Here we reflect Figure~\ref{fig:6} before comparing it with Figure~\ref{fig:1atypeiicase3}, left.
We construct the following cycle~$\omega^\cK$ in $\bV_{34}$, which is $\mathrm{CAT}(0)$.
We declare that $P_i^\cK$ is hosted by~$\wX_{43}$ 
for $x_i$ of face type $F_5$, and that $P_5^\cK$ (resp.\ $P_{9}^\cK$) has the same host as $P_6^\cK$ (resp.\ $P_{10}^\cK$). Then $x^\cK_5=x^\cK_6$ and $x^\cK_9=x^\cK_{10}$, see Figure~\ref{fig:1atypeiicase3}, right. By Lemma~\ref{lem:admissible}(\ref{3x5}) and Lemma~\ref{lem:lift2}, we can assume that the path $x_2^\cK x_3^\cK x_4^\cK x^\cK_6$ is a geodesic with angle $\geq\frac{3\pi}{4}$ at $x^\cK_2$, and so $|x^\cK_1,x^\cK_7|\geq 4$. However, the length of the path $x^\cK_7x^\cK_8x^\cK_{10}x^\cK_1$ is $\leq 2+\sqrt 2$, which is a contradiction.

\medskip
\noindent
\underline{Case 2: $C_2,C_4,C_6$ are not distinct.}  Then we can assume that none of them equals~$F_5$. 
Let $\cK$ be the type II sub-arrangement of $\ca$ as in Case 1. 
Declaring that $P_i^\cK$ are hosted by~$\wX_{43}$ for $C_i=F_5$, by~$\wX_{42}$ for $C_i=F_6$, and by~$\wX_{43}$ or $\wX_{33}$ for $C_i=F_4$, we obtain a cycle~$\omega^\cK$ in $\bV_{34}$. 
Suppose first $C_5\neq F_4$. Then by Lemma~\ref{lem:admissible}(\ref{3x5}) and Lemma~\ref{lem:lift2}, we can assume that the path $x_2^\cK x_3^\cK x_4^\cK x_5^\cK x_6^\cK$ is a geodesic with angles $\geq\frac{3}{4}$ at $x^\cK_2,x^\cK_6$, which implies $|x^\cK_1,x^\cK_7|\geq 6$. However, the length of the path $x^\cK_7x^\cK_8\cdots x^\cK_1$ is $\leq 2+2\sqrt 2$, which is a contradiction. If $C_5=F_4$, then we obtain a contradiction exactly as in Case 1.

\begin{figure}[h]
	\centering
	\includegraphics[scale=0.9]{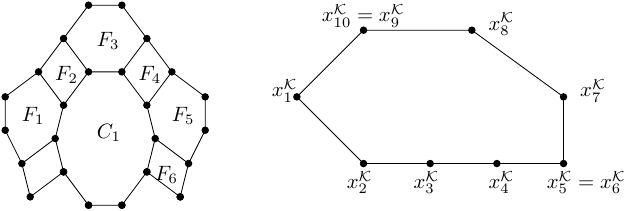}
	\caption{}
	\label{fig:1atypeiicase3}
\end{figure}

\subsection{Case of two decagons}
\label{subsec:type2a2}
\begin{prop}
	\label{prop:II2a}
	Let $\omega$ be an embedded critical $10$-cycle with $C_1\neq C_7$. Then $\omega$ is admissible.
\end{prop}

\begin{figure}[h]
	\centering
	\includegraphics[scale=0.8]{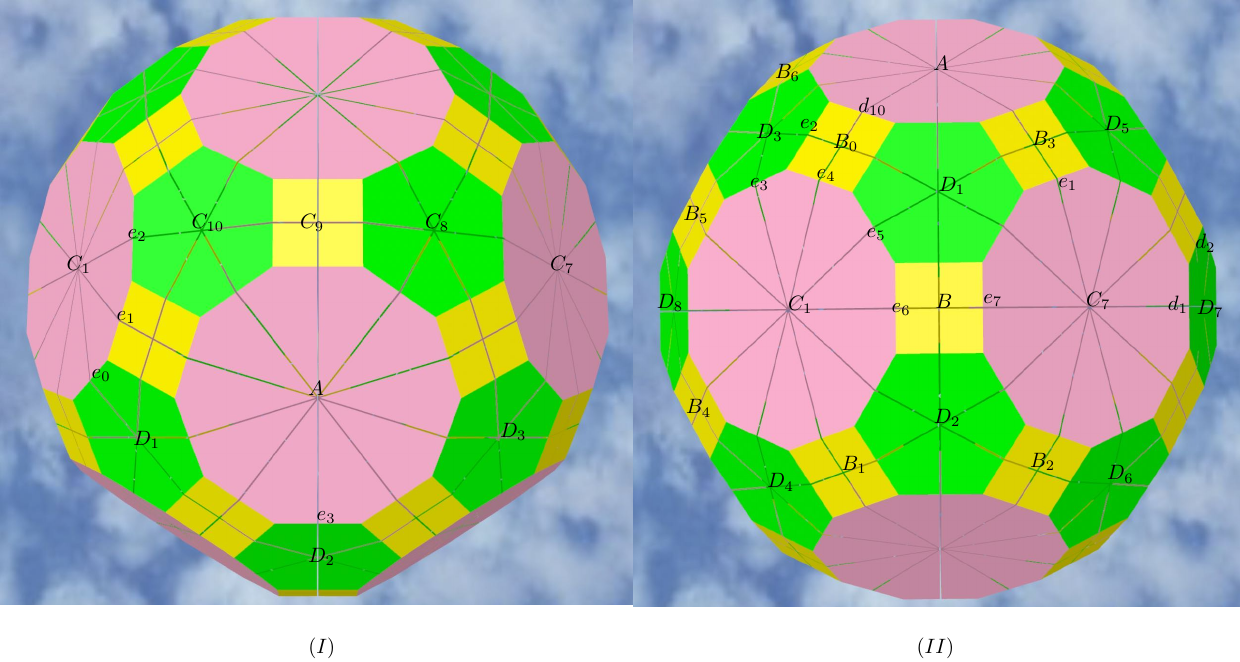}
	\caption{}
	\label{fig:2atypeiicase0}
\end{figure}

 There are two possible configurations for the pair $C_1,C_7$, illustrated in Figure~\ref{fig:2atypeiicase0}(I,II).  
Consider first Figure~\ref{fig:2atypeiicase0}(I). 
If $C_4\neq D_2$,  and all $C_2,C_4,C_6$ belong to $\{C_8,C_{10}\}$, then, by considering $\Pi_{\whC_1}(P)$, we obtain that $P_1=e_2^*$, and so $\omega$ is not locally embedded at $x_1$, which is a contradiction. Thus, if $C_4\neq D_2$, then one of $C_2,C_4,C_6$ is distinct from $C_8,C_{10}$. 
Using a symmetry, we can assume without loss of generality $C_2=D_1, C_4=C_{10},$ and $C_6=C_8$. By considering $\Pi_{\whC_1}(P)$, we obtain that $P_1$ is homotopic in $\whC_1$, relative to the endpoints,  to $e^*_2e^*_1e^*_0$. After possibly replacing the $w_i$ by equivalent words, we can assume $P_1=e^*_1$, and so we can dent $x_1$ to $A$. 
By considering $\Pi_{\whC_7}(P)$, and arguing similarly, we can also dent $x_7$ to $A$.
By Lemma~\ref{lem:admissible}(\ref{replacing1}) and Proposition~\ref{prop:IIa1} it follows that $\omega$ is admissible.
If $C_4=D_2$, then $C_2=D_1$ and $C_6=D_3$.
From considering~$\Pi_{\whC_4}(P)$ it follows that we can choose $P_4=e_3^*$. Thus there is a neighbour of type~$\hat a$ of $x_3$ and $x_5$ and so $\omega$ is admissible by Lemma~\ref{lem:admissible}(\ref{3x5}).

In the remaining part of the subsection, we will consider Figure~\ref{fig:2atypeiicase0}(II). Then~$C_9$ and all even $C_i$ intersect $C_1$ or $C_7$. 
If $C_3$ is disjoint from both $C_1,C_7$, then, by considering $\Pi_{\whC_3}(P)$, we obtain that $\omega$ is not locally embedded at $x_3$, except in one case, where, up to a symmetry, we have $C_1 \cdots C_{10}=C_1D_3B_6D_3B_0D_1C_7D_5B_3D_1$. Note that $\Pi_{\widehat D_7}(P)=d^*_1d^*_2d^*_1d^*_2$ is homotopically trivial in $\widehat D_7$. Since $\omega$ is locally embedded at $x_3$, the $*$ over $d_1$ are nonzero. By Lemma~\ref{lem:injective}, we obtain that the $*$ over $d_2$ are zero. It follows that $\Pi_{\widehat C_7}(P_1)$ is a homotopically trivial loop in $\widehat C_7$. Thus, by considering $\Pi_{\widehat C_7}(P)$, we can choose $P_7=e^*_1$. This allows us to dent $x_7$ to $A$, which, by Lemma~\ref{lem:admissible}(\ref{replacing1}), brings us to the case where all the~$C_i$ intersect $C_1$ or $C_7$.  An analogous discussion applies to $C_5$.  Thus we can assume that all the $C_i$ intersect $C_1$ or $C_7$.

\medskip
\noindent
\underline{Case 1: One of the $C_i$ equals $D_3,D_4,D_5,$ or $D_6$.} This includes the case where $C_2$ or~$C_{6}$ equals $D_7$ or $D_8$. 
We say that $D_j\in \{D_3,D_4,D_5,D_6\}$ is \emph{good}, if there is a unique $i$ with $C_i=D_j$.

\medskip
\noindent
\underline{Case 1.1: One of the $D_j$, say $D_3=C_i$, is good.} 
If $i=4$, then $C_6=D_1$ and $C_2=D_1$ or~$D_8$. Furthermore, $\Pi_{\widehat D_3}(P_8\cup P_9\cup P_{10})$ equals $e_3^*$ or $e_2^*$. Considering~$\Pi_{\widehat D_3}(P)$, it follows that we can choose $P_4=e_3^*$. Thus there is a neighbour of type $\hat a$ of $x_3$ and~$x_5$ and so $\omega$ is admissible by Lemma~\ref{lem:admissible}(\ref{3x5}).
If $i=2$, then $C_4=D_1$ and $C_6=D_1,D_2$ or $D_5$, and again $\Pi_{\widehat D_3}(P_8\cup P_9\cup P_{10})$ equals $e_3^*$ or $e_2^*$. Considering~$\Pi_{\widehat D_3}(P)$, we can choose $P_2$ to be trivial. Thus $\omega$ has angle $\frac{\pi}{4}$ at $x_2$, and so $\omega$ is admissible.

If $i=10$, then $C_9=B_0$, and $C_8=D_1$. 
Considering $\Pi_{\widehat D_{3}}(P)$, we can choose~$P_{10}$ to be trivial, which leads to angle $\frac{\pi}{4}$ at $x_{10}$, except the special cases where $C_2C_3C_4C_5C_6=D_1BD_2B_1D_2$ or $D_8B_4D_4B_1D_2$. In the second special case, by considering $\Pi_{\whC_7}(P)$, we can choose $P_7=e^*_7$, and denting $x_7$ to $C_1$, by Lemma~\ref{lem:admissible}(\ref{replacing1}), reduce to Proposition~\ref{prop:IIa1}.

In the first special case, let $\cK$ be the type II sub-arrangement of $\ca$ with $\ka_\cK(D_2)=\sX_{31},\ka_\cK(B)=\sX_{32},\ka_\cK(C_1)=\sX_{42}$. Declaring that $P_9^\cK$ is hosted by~$\wX_{43}$, we obtain a cycle $\omega^\cK$ in $\bV_{234}$, see Figure~\ref{fig:2atypeiicase2}(I). Since $x^\cK_9=x^\cK_{10}$, we have that $x^\cK_8$ is a neighbour of $x_1^\cK$. By Lemmas~\ref{lem:admissible}(\ref{3x5}) and~\ref{lem:lift2}, we can assume that $x^\cK_3,x^\cK_4,x^\cK_5$ do not have a common neighbour and the angles at $x^\cK_2,x^\cK_6$ are $\geq \frac{3\pi}{4}$. By Lemma~\ref{lem:restriction order}, we have $x^\cK_6\neq x^\cK_8$.
Let~${x'_3}^\cK$ be a type $\hat a$ neighbour of $x^\cK_3$, and let $\omega^\cK_8=x^\cK_7x^\cK_8x^\cK_1x^\cK_2{x'}^\cK_3x^\cK_4x^\cK_5x^\cK_6$. Then $\omega^{\cK}_8$ satisfies the hypotheses of Proposition~\ref{prop:8cycle diagram}. 
Thus $\omega_8$ bounds a minimal disc diagram~$\mathcal D$ that is a subdiagram of Figure~\ref{fig:3cycle1}(III). By Lemma~\ref{lem:n=8}, the link $\lk({x'_3}^{\cK},\bV_{234})$ has girth $\ge 8$. Thus, since ${x'_3}^{\cK}$ in Figure~\ref{fig:2atypeiicase2}(I) corresponds to $x_5$ in Figure~\ref{fig:3cycle1}(III),
the vertex~$x^\cK_3$ lies in the image of $\mathcal D$. Considering the simplicial structure of $\mathcal D$, we obtain a neighbour of type $\hat a$ of $x^\cK_3,x^\cK_4,$ and $x^\cK_5$, which is a contradiction.

\begin{figure}[h]
	\centering
	\includegraphics[scale=1]{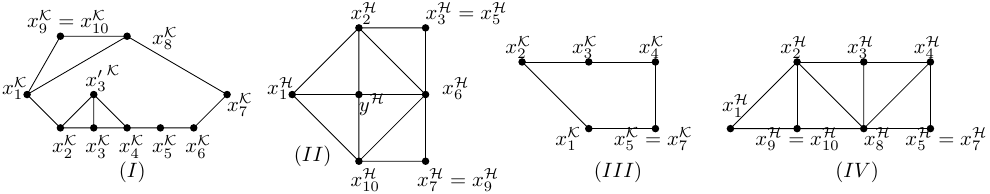}
	\caption{}
	\label{fig:2atypeiicase2}
\end{figure}

Now we assume that none of $D_3,D_4,D_5,D_6$ are good, and in particular $C_i=D_3$ is not good. Consider sequences $\Theta_1=(C_2,C_3,C_4,C_5,C_6)$, and $\Theta_2=(C_8,C_9,C_{10})$. 
Note that $D_3$ cannot occur twice in $\Theta_2$. 
Hence 
$D_3$ occurs at least once in both $\Theta_1$ and $\Theta_2$, or $D_3$ occurs twice in $\Theta_1$ but does not occur in $\Theta_2$.

\medskip
\noindent
\underline{Case 1.2: $D_3$ occurs at least once in both $\Theta_1$ and $\Theta_2$.} Then $C_{8}=D_1,$ and $C_{10}=D_3$. If $C_5=B_0$, 
then, by considering $\Pi_{\whC_7}(P)$, we can choose $P_7=e_7^*$. Denting $x_7$ to~$C_1$, we 
reduce to Proposition~\ref{prop:IIa1} by Lemma~\ref{lem:admissible}(\ref{replacing1}). 

Otherwise, since $D_5$ is not good, we have 
$C_2=D_3,C_4=D_1,C_5=B$ or $B_3,$ and $C_6=D_1$ or $D_2$.
Let $\ch$ be the type I sub-arrangement of $\ca$ with $\ka_\ch(C_1)=X_{22},\ka_\ch(D_3)=X_{11},\ka_\ch(B_5)=X_{21}$ (we reflect Figure~\ref{fig:2}). We declare that $P_6$ is hosted by $X_{22}$ if $C_6=D_2$, and $P_i$ is hosted by $X_{12}$ if $C_i\in \{B,B_3,C_7\}$. 

If $C_6\neq D_2$, then we have $x^\ch_3=x^\ch_4=x^\ch_5=x^\ch_6=x^\ch_7=x^\ch_8=x^\ch_9$.
By Lemma~\ref{lem:bowtie}, the cycle $x_{10}^\ch x^\ch_1 x^\ch_2 x^\ch_3$ has angle $\frac{\pi}{4}$ at $x^\ch_2$, and so $\omega$ is admissible by Lemma~\ref{lem:lift}.

If $C_6=D_2$, then we have $x^\ch_3=x^\ch_4=x^\ch_5$, and $x^\ch_7=x^\ch_8=x^\ch_9$. We can assume that $\omega^\ch$ has angle $\frac{3\pi}{4}$ at $x^\ch_2$ and $x^\ch_{10}$, as before. 
Since $\bU_{12}$ is $\mathrm{CAT}(0)$, we obtain that $\omega^\ch$ bounds a minimal disc diagram in Figure~\ref{fig:2atypeiicase2}(II). If $y^\ch$ has type $\whX_{12}$, then  
we can choose $P^\ch_1$ homotopic in $\whX_{22}$, relative to the endpoints, to a path in $\whX_{22}\cap \whX_{12}$. 
Thus we can assume $P_1\subset \widehat e_6\cup\widehat e_5\cup\widehat e_4$. Consequently, $\Pi_{\widehat D_2}(P_1)$ is trivial. By considering $\Pi_{\widehat D_2}(P)$, it follows that we can choose $P_6$ to be trivial, which implies that the angle at~$x_6$ is $\frac{\pi}{4}$ and so $\omega$ is admissible as before. If $y^\ch$ has face type $\whX_{21}$, then
we can dent~$x_1$ to the decagon distinct from $C_1$ intersecting $B_5$ and reduce to the configuration from Figure~\ref{fig:2atypeiicase0}(I).

\medskip
\noindent
\underline{Case 1.3: $D_3$ occurs at least twice in $\Theta_1$ but does not occur in $\Theta_2$.} Then $C_2=C_4=D_3$ and $C_6=D_1$. We have $\{C_8,C_{10}\}\subset\{D_1,D_2\}$, since $D_4,D_5,D_6$ are not good. Let $\ch$ be the sub-arrangement as in Case 1.2. 
Then $P_5,P_6,P_7$ are hosted by $X_{12}$. We declare that $P_i$ with $C_i=D_1$ (resp.\ $D_2$) are hosted by $X_{12}$ (resp.\ $X_{22}$), and $P_9$ has the same host as $P_{10}$. If $C_8\neq D_2$ or $C_{10}\neq D_1$, then $\omega^\ch$ is a $5$-cycle $x_1^\ch x_2^\ch x_3^\ch x_4^\ch
x_5^\ch$ as in Figure~\ref{fig:2atypeiicase2}(III) with angle $\pi$ at  $x_3^\ch$ and angle  $\frac{\pi}{2}$ at $x_5^\ch$. Since $\bU_{12}$ is $\mathrm{CAT}(0)$, $\omega^\ch$ has angle $\frac{\pi}{4}$ at $x^\ch_2$, and we finish as before. If $C_8=D_2$ and $C_{10}=D_1$, then $x^\ch_5=x_6^\ch=x_7^\ch$ and so $\omega^\ch$ is a $7$-cycle $x_1^\ch x_2^\ch x_3^\ch x_4^\ch
x_5^\ch x_8^\ch x_{10}^\ch$, where $x^\ch_4$ and $x^\ch_8$ are neighbours. If $x^\ch_1, x^\ch_3$ are not neighbours, then, 
since $\bU_{12}$ is $\mathrm{CAT}(0)$, the path 
$\omega^\ch$ bounds the reduced disc diagram in Figure~\ref{fig:2atypeiicase2}(IV). 
Thus $x^\ch_3,x^\ch_4,x^\ch_5$ have a common neighbour, and by Lemma~\ref{lem:lift} so do $x_3,x_5$. Hence, by
Lemma~\ref{lem:admissible}(\ref{3x5}), $\omega$~is admissible.

\medskip
\noindent
\underline{Case 2: All even $C_i$ belong to $\{D_1,D_2\}$.}
Without loss of generality, we can assume that neither $C_3$ nor $C_5$ equals $B_1$ or $B_2$. Let $\cL$ be the type II sub-arrangement of~$\ca$ with $\ka_\cK(D_1)=\sX_{31},\ka_\cK(B)=\sX_{32},\ka_\cK(C_7)=\sX_{42}$. If $C_9\notin \{B_1,B_2\}$, and $\omega$ is not admissible, then $\hk_\cL(\omega)$ contradicts Corollary~\ref{cor:remove}. 

If $C_9\in \{B_1,B_2\}$, then $C_8=C_{10}$, and we consider the $8$-cycle $\omega^\cL_8$ obtained from $x_1^\cL\cdots x_8^\cL$ by replacing $x_3^\cL$ with its neighbour of type $\hat a$. Using Proposition~\ref{prop:8cycle diagram}, we obtain that $\omega$ is admissible by the same argument as in the special case of Case~1.1.

\section{353 square complexes}
\label{sec:353square}
We refer to Definition~\ref{def:353square} for the notion of a $353$-square complex. 

\begin{defin}\label{def:stable}
	A $353$-square complex is \emph{stable} if for any set $S\subset \mathcal A$ or $\mathcal D$ of pairwise close vertices, there is a finite subset $S'\subset S$, such that if a vertex $v$ is close to or a neighbour of the entire $S'$, then $v$ is close (or equal) to or a neighbour of the entire~$S$.
\end{defin}

\begin{defin}\label{def:wide}
	A $353$-square complex is \emph{wide} if any simplex of $X^\boxtimes$ is contained in a simplex $\sigma$ with $|\sigma^0\cap \mathcal A|\geq 2$, and $|\sigma^0\cap \mathcal D|\geq 2$.
\end{defin}

The goal of this section is to prove Theorem~~\ref{thm:contractible intro}.

\subsection{Disc diagrams}
\label{subsec:diagram}

\begin{rem}
	\label{rem:squares}
	From Definition~\ref{def:axioms}(\ref{def:axioms1}) it follows that if $ad_1a_1d_2$, $ad_1a_2d_2$, and $da_1d_1a_2$  are squares, then 
	$d_1a_1d_2a_2$ and $da_1d_2a_2$ are squares (see Figure~\ref{fig:border}(1)). Consequently, if a minimal disc diagram $D$ in $X$ contains a cube corner, then it cannot contain the additional two squares in Definition~\ref{def:axioms}(\ref{def:axioms3}), since otherwise we could replace the five squares by three squares.
\end{rem}

\begin{figure}[h]
	\centering
	\includegraphics[scale=0.9]{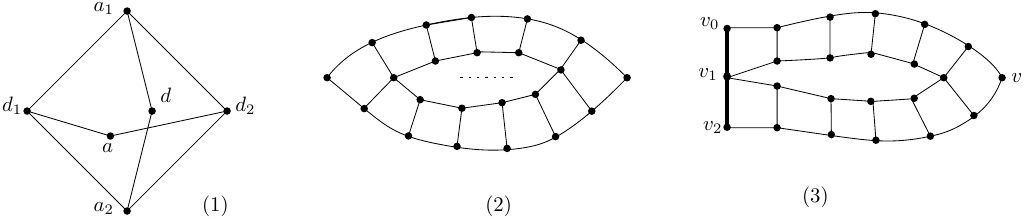}
	\caption{}
	\label{fig:border}
\end{figure}

\begin{defin}
	\label{def:inter}
	Let $D$ be a square disc diagram homeomorphic to a disc. We say that $D$ is \emph{$n$-bordered}, for $n\geq 0$, if 
	\begin{itemize}
		\item
		there are exactly $n$ vertices of $\partial D$ not contained in any interior edge, and 
		\item
		each of the remaining vertices of $\partial D$ is contained in exactly one interior edge.
	\end{itemize}
	See Figure~\ref{fig:border}(2) for an example of a $2$-bordered disc diagram.

	We say that $D$ is an \emph{inter-osculation} (see Figure~\ref{fig:border}(3)) if there are consecutive vertices $v_0,v_1,v_2\in \partial D$, and another vertex $v\in \partial D$, such that 
	\begin{itemize}
		\item
		$v_0,v_2,v$ are not contained in any interior edge, and 
		\item
		$v_1$ is contained in at least one interior edge, and
		\item
		each of the vertices of $\partial D\setminus \{v_0,v_1,v_2,v\}$ is contained in exactly one interior edge.
	\end{itemize}
\end{defin}

A  \emph{hyperplane} in a square disc diagram is a maximal immersed $1$-manifold obtained by connecting the midpoints of opposite edges (called \emph{dual edges}) in consecutive squares, see for example \cite[\S2.4]{sageev}. The \emph{carrier} of an embedded hyperplane $h$ is the union of squares intersecting~$h$. 

\begin{lem} 
	\label{lem:diagram}
	Let $f\colon D\to X$ be a minimal disc diagram in a $353$-square complex. We equip $D^1$ with the path metric such that each edge has length $1$.
	\begin{enumerate}[(i)]
		\item 
		\label{diagram:ii}
		$D$ is not $1$- or $2$-bordered. If $D$ is an inter-osculation, then it is a cube corner.
		\item 
		\label{diagram:i}
		Hyperplanes in $D$ are embedded, not homeomorphic to circles, and pairwise intersect at most once.
		\item 
		\label{diagram:iii}
		A geodesic in $D^1$ intersects each hyperplane at most once.
		\item 
		\label{diagram:iv}
		If $\partial D=\alpha\beta^{-1}$, where $\alpha,\beta$ are geodesics in $D^1$ with first edges $da_1,da_2$, then either $d,a_1,a_2$ lie in a square of $D$, or not in a square but in a  cube corner of~$D$. In the latter case $f(a_1),f(a_2)$ are not close.
	\item
		\label{diagram:v}
		In (iv), the vertex $u$ opposite to $d$ in the top square (resp.\ cube corner) lies on a geodesic in $D^1$ with the same endpoints as $\alpha$.
	\end{enumerate}
\end{lem}
\begin{figure}[h]
	\centering
	\includegraphics[scale=1]{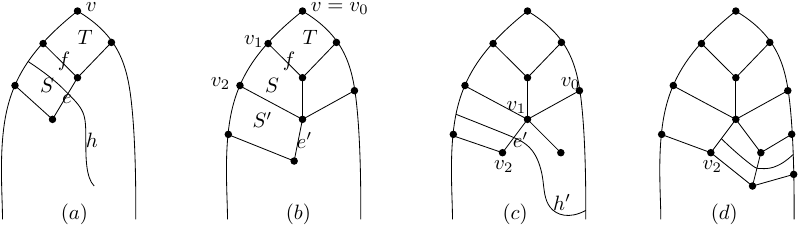}
	\caption{}
	\label{fig:353diagram}
\end{figure}
\begin{proof}
	For part (\ref{diagram:ii}), to reach a contradiction, let $D\to X$ be a minimal disc diagram with the smallest area that is 
	\begin{itemize}
		\item $1$-bordered, or
		\item $2$-bordered, or 
		\item is an inter-osculation but not a cube corner. 
	\end{itemize}
	Let $v$ be the vertex of $\partial D$ from Definition~\ref{def:inter} in the third case, or a vertex not contained in an interior edge, in the first two cases. Let $T\subset D$ be the square containing $v$, and let $f\subset T$ be an edge not containing $v$. Let $S\subset D$ be the second square containing $f$, and let $e$ be the second edge in $S$ containing the vertex of~$T$ opposite to $v$ (see Figure~\ref{fig:353diagram}(a)). By Definition~\ref{def:axioms}(\ref{def:axioms1}), we have that $e$ is not contained in~$T$. Consider the hyperplane $h$ of $D$ dual to $e$. Note that $h$ cannot self-intersect, since we would obtain a smaller area $1$-bordered diagram. Let $p$ be the intersection point of $h$ and $\partial D$ outside $S$. Note that when $D$ is an inter-osculation, $p$ is not the midpoint of a thickened edge in Figure~\ref{fig:border}(3). If there would be a path in $\partial D$ from $p$ to $f\cap \partial D$ whose all vertices are contained in the interior edges of $D$, then we would obtain a smaller area $2$-bordered diagram. Thus there is a path in $\partial D$ from $p$ to $f\cap \partial D$ whose all vertices except for $v$ are contained in the interior edges of $D$. This gives us an inter-osculation with $v_0=v, v_1=f\cap \partial D,$ and $v_2$ the other vertex of $S\cap \partial D$. By the minimality assumption on $D$, this inter-osculation is a cube corner and so $e$ lies in a square sharing an edge with $T$. See Figure ~\ref{fig:353diagram}(b).  Note that this cube corner is minimal, since it is a subdiagram of a minimal diagram.
	
	Let $S'$ be the remaining square of $D$ sharing an edge with $S$, and let $e'$ be its edge intersecting~$e$ but not contained in $S$. By Definition~\ref{def:axioms}(\ref{def:axioms5},\ref{def:axioms3}), and Remark~\ref{rem:squares}, the degree of the vertex $v_1=e\cap e'$ is distinct from $3$ and $4$. 
	Let $h'$ be the hyperplane of~$D$ dual to $e'$. As before, $h'$ does not self-intersect, and we obtain an inter-osculation with $v_0,v_1,v_2$ as in Figure~\ref{fig:353diagram}(c). By the minimality assumption on $D$ this inter-osculation is a cube corner and so $D$ contains the subdiagram in Figure~\ref{fig:353diagram}(d). Analogically, interchanging the left and the right side of the diagram, we obtain that the degree at $v_2$ equals $3$. This contradicts Definition~\ref{def:axioms}(\ref{def:axioms4}).
	
	Hyperplanes not satisfying part (\ref{diagram:i}) give rise to disc diagrams excluded by part~(\ref{diagram:ii}).
	
	For part (\ref{diagram:iii}), note that the $1$-skeleta of hyperplane carriers are isometrically embedded in $D^1$. Indeed, if vertices of the carrier were closer in $D^1$ than in the $1$-skeleton of the carrier, then they would be separated in the carrier by two hyperplanes contained in the same hyperpane of $D$, which would contradict part (\ref{diagram:i}). 
	Thus a path intersecting twice a hyperplane can be shortened by replacing its subpath by a path in the carrier.
	
	In part (\ref{diagram:iv}), if $d,a_1,a_2$ do not lie in a square, then, by part (\ref{diagram:iii}), the hyperplanes dual to $da_1,da_2$ need to intersect elsewhere. Thus they form an inter-osculation, which by part (\ref{diagram:ii}) is a cube corner. Consequently, $a_1$ and $a_2$ are not close by Definition~\ref{def:axioms}(\ref{def:axioms5}) and Corollary~\ref{cor:squares} below. 
	
	Part (v) follows from part (iii), since if the geodesic from $u$ to the last vertex of $\alpha$ and $\beta$ 
    was intersected by any hyperplane $h$ of the top square or cube corner, then $h$ would intersect twice $\alpha$ or $\beta$.
\end{proof}

\begin{cor} 
	\label{cor:squares}
	Let $X$ be a $353$-square complex. Then all $4$-cycles in $X$ are squares.
\end{cor}
\begin{proof} Let $\alpha$ be a $4$-cycle in $X$ and $D\to X$ a minimal disc diagram bounded by~$\alpha$. By Lemma~\ref{lem:diagram}(\ref{diagram:i}), $D$ has only $2$ hyperplanes, which moreover intersect at most once. Thus $D$ consists of a single square.
\end{proof}

\begin{cor}
	\label{cor:exists_a_easy}
	Let $X$ be a $353$-square complex and let $a_1,a_2,a_3$ be pairwise close. Then there exists $d$ that is a neighbour of all $a_i$.
\end{cor}
\begin{proof} 
	Assume by contradiction that $a_1,a_2,a_3$ do not have a common neighbour. Then we have an embedded $6$-cycle in $X$ passing through $a_1,a_2,a_3$. By Lemma~\ref{lem:diagram}(\ref{diagram:i}), its minimal disc diagram $D\to X$ has 3 hyperplanes and 3 squares, and so it is a cube corner. Moreover, $a_i$ are not contained in the interior edges of $D$. By Lemma~\ref{lem:diagram}(\ref{diagram:iv}), $a_1,a_2$ are not close, which is a contradiction.
\end{proof}

\subsection{Structure of downward links}
\label{subsec:353contractibility}
We fix from now on a `base' vertex $w$ of a $353$-square complex $X$. Let $S^k=S^k(w)$ denote the set of vertices of $X$ at distance~$k$ from $w$ in $X^1$, and let $B^k$ be the subgraph of~$X^1$ induced on the union of $S^l$ over $0\leq l\leq k$. We also suppose that $X$ is stable.
Given $d\in S^k$, we define $p(d)$ as the set of all neighbours of $d$ in $S^{k-1}$.

\begin{lem}
	\label{lem:easy_projection}
	Let $A\subset p(d)$ be a set of pairwise close vertices.
    Then there exists $d'\in S^{k-2}$ that is a neighbour of all the elements of $A$. 
\end{lem}
\begin{proof}
	Assume first that $A$ is finite, that is, $A=\{a_i\}_{i=1}^n$.
    For $n=1$, there is nothing prove. For $n=2$,
	choose geodesics $\alpha, \beta$ from $d$ to the base vertex $w$ passing through $a_1,a_2$, and let $D\to X$ be a minimal disc diagram with boundary $\alpha\beta^{-1}$. By Lemma~\ref{lem:diagram}(iv), we have that $d,a_1,a_2$ lie in a square of $D$, whose remaining vertex belongs to $S^{k-2}$ by Lemma~\ref{lem:diagram}(v).

	Suppose now $n=3$. 
	By $n=2$ case, there are $d_1,d_2,d_3 \in S^{k-2}$ that are neighbours of both~$a_i$ distinct from $a_1$ (resp.\ $a_2,a_3$). Suppose without loss of generality that all~$d_i$ are distinct.
	Consider a minimal disc diagram $D$ bounded by geodesics from~$a_2$ to~$w$ via $d_1$ and $d_3$. By Lemma~\ref{lem:diagram}(iv,v), we have one of the following options. Suppose first that there is a square containing $d_1,a_2,d_3$ and a vertex $a\in S^{k-3}$. Then, by Definition~\ref{def:axioms}(\ref{def:axioms5}),
	the cube corner with vertices $d,a_i,d_i,$ is not minimal, and so for some $i$ there is an edge $a_id_i$, as desired. Second, suppose that there are two squares containing $a_2$ (and some $d_0$) in $D$. By Lemma~\ref{lem:diagram}(v), we have $d_0\in S^{k-2}$.
	By Definition~\ref{def:axioms}(\ref{def:axioms3}), we have either again an edge $a_id_i$, or $d_0$ is a neighbour of both $a_1$ and $a_3$ and so $d_0$ is the required~$d'$.
	
	If $n\geq 4$, arguing by induction, there are again distinct $d_1,d_2,d_3 \in S^{k-2}$ that are neighbours of all $a_i$ except for $a_1$ (resp.\ $a_2,a_3$). In particular, we have a square $d_1a_2d_3a_4$ and so $d_1,d_3$ are close. By $n=2$ case, there is a vertex $a\in S^{k-3}$ that is a neighbour of both $d_1,d_3$. As in the first case of the previous paragraph, applying Definition~\ref{def:axioms}(\ref{def:axioms5}), we deduce that for some $i$ there is an edge $a_id_i$.
	
	The case of infinite $A$ 
    follows from the stability of $X$.
\end{proof}

Note that the following result would be trivial if we had assumed that $X$ is wide.

\begin{cor}
	\label{cor:exists_a}
	Let $A\subset \mathcal A$ be a set of pairwise close vertices.
    Then there exists $d$ that is a neighbour of all 
    the elements of $A$.
\end{cor}
\begin{proof} By the stability of $X$, we can assume $A=\{a_i\}_{i=1}^n$, and we proceed by the induction on $n$. For the induction step, assume that we have already a common neighbour $d$ of all $a_i$ distinct from $a_1$. If $d$ is not a neighbour of $a_1$, then it belongs to $S^3(a_1)$. We then apply Lemma~\ref{lem:easy_projection} with $w=a_1$ and $k=3$.
\end{proof}

In view of Corollary~\ref{cor:exists_a}, in Lemma~\ref{lem:easy_projection} instead of assuming that there exists $d$ with $a_i\in p(d)$, we can just assume $a_i\in S^{k-1}$. Indeed, if $d$ is a common neighbour of $a_i$ from Corollary~\ref{cor:exists_a}, then $d$ either belongs to $S_{k-2}$, as required in Lemma~\ref{lem:easy_projection} or $d$ belongs to $S^k$, implying $a_i\in p(d)$.

\begin{lem} 
	\label{lem:2convex}
	Suppose that $a_1,a_2\in p(d)$ are not close, but are both close to a neighbour $a$ of $d$. 
	\begin{enumerate}[(i)]
		\item
		Then $a\in p(d)$ (i.e.\ $a\in S^{k-1}$). 
		\item
		If $a_1,a_2$ are both close to another neighbour $a'$ of $d$, then $a$ and $a'$ are close.
		\item
		If $d_1,d_2\in S^{k-2}$ are neighbours of $a'$ and $a_1,a_2$, respectively, then $d_1$ and $d_2$ are close.
	\end{enumerate}
\end{lem}
\begin{proof} By Lemma~\ref{lem:diagram}(iv,v), there is a minimal cube corner $C$ with boundary vertices $d_1a_1da_2d_2a_3$, where $d_1\in S^{k-2}$. Since $a_1$ and $a$ are close and $a$ and $a_2$ are close, they belong to the remaining squares needed to apply Definition~\ref{def:axioms}(\ref{def:axioms3}). This shows that $a$ is a neighbour of $d_1$ (and $d_2$) and so $a\in S^{k-1}$, proving (i).
	For (ii), we analogously obtain that $a'$ is a neighbour of $d_1$ and $d_2$. Thus $ad_1a'd_2$ is a square and so $a$ and $a'$ are close.
	
	For (iii), If $d_1,d_2$ are not close, then, by Lemma~\ref{lem:diagram}(iv,v), there is a minimal cube corner $C$ with boundary vertices $a^-d_1a'd_2a^+d'$ with $a^-\in S^{k-3}$. By Definition~\ref{def:axioms}(\ref{def:axioms3}), we have that $d\in S^k$ is a neighbour of $a^-$, which is a contradiction.
\end{proof}

	\begin{lem}
		\label{lem:graph}
		Let $G$ be a simplicial graph
		\begin{enumerate}
			\item of diameter $2$,
			\item without induced embedded  $4$-cycles, and
			\item whose all induced embedded $5$-cycles have a common neighbour.
		\end{enumerate}
		Let $A_1,A_2$ be the vertex sets of finite complete subgraphs of $G$. Then there is a vertex~$a$ of $G$ that is a neighbour of or equal to all of the elements of $A_1\cup A_2$.
	\end{lem}

\begin{proof}
	First note that we can assume that each element $a_1\in A_1$ is not a neighbour of some element $a_2'\in A_2$, since otherwise we can take $a=a_1$. Second, note that we can assume that $a_1$ is not a neighbour of {\bf each} element $a_2\in A_2$. Indeed, otherwise we replace $A_2$ by $A_2\setminus \{a_2\}$, and we find by induction a neighbour $a$ of all the elements of $A_1\cup A_2\setminus \{a_2\}$. Applying assumption (2) to $a_1a_2a_2'a$ with $a_2'$ as above, we obtain that $a$ is a neighbour of $a_2$, as desired.
	
	If $|A_1|=|A_2|=1$, then the lemma follows from assumption (1). Assume now $A_1=\{a_1\}$ and $A_2=\{a_3,a_4\}$. 
	By assumption (1), $a_1$ and $a_3$ have a neighbour~$a_2$, and $a_1$ and $a_4$ have a neighbour $a_5$. We can assume that $a_2$ and $a_4$ are not neighbours, since otherwise we can take $a=a_2$. Analogously, we can assume that $a_3$ and~$a_5$ are not neighbours, since otherwise we can take $a=a_5$. Then $a_2$ and $a_5$ are not neighbours by assumption (2). Thus $a_1a_2a_3a_4a_5$ is an induced embedded $5$-cycle and it remains to apply assumption (3). 
	
	We now suppose $A_1=\{a_1\}$ and $3\leq m=|A_2|$, and we argue by induction on $m$.
	By induction, there is a vertex $a$ (resp.\ $a'$) that is a neighbour of or equal to all the elements of $A_1\cup A_2$ except possibly for $a_2$ (resp.~$a_2'$) in~$A_2$.  Applying assumption~(2) to $aa_1a'a_2''$ for some $a_2''\in A_2\setminus \{a_2,a_2'\}$, we obtain that $a$ and $a'$ are neighbours. Applying assumption (2) to $aa'a_2a_2'$, we obtain that $a$ is a neighbour of $a_2$ or $a'$ is a neighbour of $a_2'$.
	
	Finally, assume $2\leq |A_1|,|A_2|$. Choose $a_1\neq a_1'\in A_1,a_2\neq a_2'\in A_2$. By induction, there is a vertex $a^1$ (resp.\ $a'^1,a^2,a'^2$) that is a neighbour of or equal to all the elements of $A_1\cup A_2$ except possibly for $a_1$ (resp.\ $a_1',a_2,a_2'$).
	By assumption (2), both~$a^1$ and~$a'^1$ are neighbours of both $a^2,a'^2$. Applying assumption (2) to $a^1a^2a'^1a'^2$ we obtain that, say, $a^1$ is a neighbour of $a'^1$. Applying again assumption (2) to $a^1a'^1a_1a'_1$, we obtain that  $a=a^1$ or $a=a'^1$ satisfies the lemma.
\end{proof}

\begin{lem}
	\label{lem:maxcliques}
	Let $A_1,A_2\subset p(d)$ be sets of pairwise close vertices. Then there is $a\in p(d)$ that is a close (or equal) to all the elements of $A_1\cup A_2$. In particular, two maximal sets of pairwise close vertices in $p(d)$ have non-empty intersection.
\end{lem}
\begin{proof} 
	Assume to start with that both $A_1$ and $A_2$ are finite. Let $G$ be the simplicial graph with vertex set $p(d)$, and edges between close elements. It suffices to verify that $G$ satisfies the assumptions of Lemma~\ref{lem:graph}. Assumption (1) follows from Lemma~\ref{lem:diagram}(iv,v). Assumption (2) follows from Lemma~\ref{lem:2convex}(ii). To verify assumption~(3), let $a_1\cdots a_5$ be an induced embedded $5$-cycle. 
By Lemma~\ref{lem:easy_projection}, for $i=1,\ldots,5$ there are $d_i\in S^{k-2}$ that are neighbours of $a_i$ and $a_{i-1}$ (mod $5$). By Lemma~\ref{lem:2convex}(iii), each $d_i$ is close to $d_{i+1}$. Again by Lemma~\ref{lem:easy_projection}, this leads to the existence of squares $d_ia_id_{i+1}a_i'$ with $a'_i\in S^{k-3}$. Note that the cube corners with centres $a_i$ and boundaries $a_{i-1}da_{i+1}d_{i+1}a'_id_i$ are minimal. By Definition~\ref{def:axioms}(\ref{def:axioms4}), there is $a\in p(d)$ that is close to all $a_i$, as desired.
	
	The cases of infinite $|A_1|$ or $|A_2|$ follow from the stability of $X$.
\end{proof}

\subsection{Contractibility}

\begin{lem}
	\label{lem:contract}
	Let $L$ be a simplicial complex containing a simplex $M$ satisfying the following properties.
	\begin{enumerate}[(i)]
		\item
		Every maximal simplex $\sigma$ in $L$ intersects $M$.
		\item 
		For every set $V$ of vertices in $L$ pairwise connected by edges and such that $V\setminus M$ spans a simplex, we have that $V$ spans a simplex.
	\end{enumerate}
	Then $L$ is contractible.
\end{lem}

	\begin{proof}
Assume to start with that $L$ is finite. Let $L'$ be the barycentric subdivision of~$L$. We equip the vertex set of $L'$ with the poset structure coming from the inclusion between the simplices of $L$. Let $K$ be the subcomplex of $L'$ spanned on the barycentres of all the simplices of $L$ intersecting~$M$. Let $M'\subset K$ be the barycentric subdivision of $M$.
		
		First, we justify that $K$ is contractible. Indeed, assign to each vertex $x$ of $K$ corresponding to a simplex $\sigma$ of $L$, the barycentre $F(x)$ of $\sigma\cap M$. For $x\leq x'$, we have $F(x)\leq F(x')$. In particular, we have that $F$ extends to a simplicial map from~$K$ to $M'$, which is homotopically trivial. 
		Finally, since $F(x)\leq x$, by  \cite[Prop 2.1]{Segal1083}, we have that $F$ is homotopic to the identity map on $K$. Consequently, the identity map on $K$ is homotopically trivial.
		
		Second, we justify that $L$ is contractible. For that, let $L''$ be the barycentric subdivision of $L'$. Let $v$ be a vertex of $L''$, which is a chain $x_0\leq x_1\leq \cdots \leq x_n$ of vertices of $L'$. We consider two maps $F_1,F_2$ assigning to each such $v$ a vertex of $L'$. Let $F_1(v)=x_0$. For $F_2$, let $\tau_j$ be the simplex of $L$ corresponding to $x_j$. Let $\pi_n$ be the set of all the vertices of $M$ that are neighbours of, or equal to, all the vertices of $\tau_n$. By (i), we have that $\pi_n$ is non-empty and by (ii) we have that $\tau_n\cup \pi_n$ spans a simplex. We consider its subsimplex spanned on $\tau_0\cup \pi_n$, the barycentre of which we denote $F_2(v).$ Note that for $v\subseteq v'$, we have $F_1(v)\geq F_1(v')$ and $F_2(v)\geq F_2(v')$, and so $F_i$ extend to simplicial maps from $L''$ to $L'$. In fact, we have that $F_1$ is homotopic to the identity map on $L''$ (see \cite[Prop~4.2]{przytycki2009g}). Furthermore, since $F_1(v)\subseteq F_2(v)$, by \cite[Prop 2.1]{Segal1083}, we have that $F_1$ and $F_2$ are homotopic. Finally, the image of~$F_2$ is contained in the subcomplex corresponding to $K$. Since $K$ is contractible, we obtain that the identity map on $L''$ is homotopically trivial.
		
	If $L$ is infinite, then each finite subcomplex of $L$ is contained in a finite subcomplex $L'\subset L$ satisfying the assumptions of the lemma. This implies that all the homotopy groups of $L$ vanish, and so $L$ is contractible by Whitehead theorem. 
\end{proof}

\begin{rem}
	\label{rem:LM}
	Suppose that $L$ and $M$ satisfy conditions (i) and (ii) from Lemma~\ref{lem:contract}. Then any induced subcomplex of $L$ containing $M$ also satisfies (i) and (ii) with the same $M$.  
\end{rem}

\begin{proof}[Proof of Theorem~\ref{thm:contractible intro}]
	Inside the thickening $X^\boxtimes$, we consider the full subcomplexes $\mathrm{Span}\,B^k$ spanned on the graphs $B^k$ (see the beginning of Section~\ref{subsec:353contractibility}). Suppose without loss of generality $S^k\subset \mathcal D$. 
To start with, suppose that $S^k$ is finite. We call a simplex $\tau$ of $\mathrm{Span}\,B^k$ \emph{peelable} if $|\mathcal A\cap \tau|=1$ and $\tau$ is not contained in a simplex of $\mathrm{Span}\,B^k$ with another element of~$\mathcal A$. Note that if $\tau$ is peelable, and $\tau$ is contained in a simplex $\rho$ of $\mathrm{Span}\,B^k$, then $\rho$ is peelable.
	
	Since $X$ is wide, there are no peelable simplices in $X^\boxtimes$. This implies that peelable~$\tau$ in $\mathrm{Span}\,B^k$ satisfies $\tau\cap \mathcal D \subset S^k$ (since a common neighbour of $\tau\cap \mathcal D$ must be missing from $B^k$, hence belongs to  $S^{k+1}$).
	
	For $\tau$ peelable, 
	let $\tau'=\tau\cap \mathcal D$. We claim that if $\tau_1'=\tau_2'$, then $\tau_1=\tau_2$. Indeed, suppose $\tau_1=\tau_1'\cup \{a_1\}, \tau_2=\tau_1'\cup \{a_2\}$. If $|\tau_1'|\geq 2$, then $a_1,a_2$ are close, and so $\tau_1$ is contained in $\mathrm{Span}\,B^k$ in the simplex  $\tau_1\cup \{a_2\}$, contradicting peelability. If $\tau'_1=\{d\}$, then we have $d\in S^k, a_1,a_2\in S^{k-1}$. Thus $|p(d)|\geq 2$, and so by Lemma~\ref{lem:diagram}(iv,v) there is a square in $\mathrm{Span}\,B^k$ containing $a_1d$, which is a contradiction.

	We denote by $P^k$ the subcomplex of $\mathrm{Span}\,B^k$ obtained from removing the peelable open simplices $\tau$ and their `free' open faces $\tau'$. Note that by the claim above, we have that $P^k$ is obtained from performing successive collapses on $\mathrm{Span}\,B^k$ (starting with $\tau$ of the maximal dimension), so they are homotopy equivalent.

	Let now $d\in P^k\cap S^k$. We will describe the link $L$ of $d$ in $P^k$. Let $A_\lambda$ be the maximal subsets of $p(d)$ of pairwise close elements, over $\lambda\in \Lambda$. Since $d\in P^k$, we have $|A_\lambda|\geq 2$ for each $\lambda\in \Lambda$.
	For each $\lambda\in \Lambda$, let $D_\lambda\subset S^{k-2}$ be the set of common neighbours of $A_\lambda$, which is non-empty by Lemma~\ref{lem:easy_projection}. Since $|A_\lambda|\geq 2$, any $d_\lambda,d'_\lambda\in D_\lambda$ are close. Furthermore, by Lemma~\ref{lem:maxcliques} and Lemma~\ref{lem:2convex}(iii), any $d_\lambda\in D_\lambda,d_\mu\in D_\mu$ are close. Since $|A_\lambda|\geq 2$, we have that all $d_\lambda\in D_\lambda$ are also close to $d$. Thus the union of $D_\lambda$ over $\lambda\in \Lambda$ spans a simplex $M$ in the link $L$ of $d$ in $P^k$. We will now show that $L$ and $M$ satisfy the conditions of Lemma~\ref{lem:contract}.
	
	By Lemma~\ref{lem:contract}, this will imply that $P^k$ deformation retracts to its subcomplex obtained from removing the open star of $d$.
	Repeating this procedure with $d$ replaced by other elements of $P^k\cap S^k$, which can be done by Remark~\ref{rem:LM}, shows that $P^k$ deformation retracts to $\mathrm{Span}\,B^{k-1}$.	
	
	Condition (ii) of Lemma~\ref{lem:contract} follows from the fact that $\mathrm{Span}\,B^k$ is flag, and no peelable simplices have a vertex in $M$, since $M$ is contained in $S^{k-2}$. 
	For condition~(i), since $d\cup \sigma$ is a maximal simplex of $P^k$, it contains at least one $a\in A$ by Lemma~\ref{lem:easy_projection} and Corollary~\ref{cor:exists_a}. 
	Since $P^k$ does not contain peelable simplices, in fact we have $\sigma^0=A\cup D$, where $|A|\geq 2$. We have $A\subset p(d)$, and so we can pick $\lambda$ satisfying $A\subseteq A_\lambda$. Then any $d_\lambda\in D_\lambda$ is close to all the elements of $D$, and so by the maximality of $\sigma$ we have $d_\lambda\in \sigma$. Thus $M\cap \sigma$ contains~$d_\lambda$, as desired.
	
	 If $S^k$ is infinite, then each finite subcomplex of $\mathrm{Span}\,B^k$ is contained in the span of $\mathrm{Span}\,B^{k-1}$ and a finite subset of $S^k$. Hence, by the above discussion, we can homotope this subcomplex into 
     $\mathrm{Span}\,B^{k-1}$.
     This implies that all the homotopy groups of $\mathrm{Span}\,B^k$ vanish, and so $\mathrm{Span}\,B^k$ is contractible by the Whitehead theorem. 
\end{proof}

\section{353 Simplicial complexes}
\label{sec:353simplicial}
Let $\Delta$ be a simplicial complex of type $S=\{\hat a,\hat b,\hat c,\hat d\}$, see Section~\ref{subsec:chamber complex}. We equip~$S$ with the total order $\hat a<\hat b<\hat c<\hat d$, which induces a relation $<$ on the vertex set of~$\Delta$ as in Definition~\ref{def:order}. We denote by $\mathcal A$ the set of vertices of type $\hat a$ etc.

\begin{defin}
	\label{def:353simplicial1}
	$\Delta$ is a \emph{$353$-simplicial complex} if it is simply connected and satisfies the following properties.
	\begin{enumerate}
		\item \label{353order} The relation $<$ on $\Delta^0$ is a partial order.
		\item \label{35346} Each $\lk(d,\Delta)^0$ (resp.\ $\lk(a,\Delta)^0$) is bowtie free and upward (resp.\ downward) flag.		
		\item 
		\label{3538} Each cycle $a_1c_1a_2c_2a_3c_3bc_4$ (resp.\ $d_1b_1d_2b_2d_3b_3cb_4$) in some $\lk(d,\Delta)$ (resp.\ $\lk(a,\Delta)$) is not embedded or not induced.
		\item \label{35310} If $\gamma=c_1b_1c_2a_2c_3b_3c_4b_4c_5a_5$ is an induced embedded cycle in some $\lk(d,\Delta)$, then there is a neighbour $a\in \lk(d,\Delta)$ of all the vertices of $\gamma$ in $\mathcal B$. An analogous condition holds for $\mathcal A, \mathcal B$ interchanged with $\mathcal D,\mathcal C$.
	\end{enumerate}
$\Delta$ is \emph{wide} if each vertex in $\mathcal B\cup \mathcal C$ has at least two neighbours in $\mathcal A$ and two neighbours in $\mathcal D$.
\end{defin}

We will provide the main example of a $353$-simplicial complex in Theorem~\ref{lem:main example}.
Our goal for the moment is to prove the following. 

\begin{thm}
	\label{thm:353}
	 Let $\Delta$ be a  
     $353$-simplicial complex. 
		Let $X^1\subset \Delta^1$ be the subgraph induced on $\mathcal A\cup \mathcal D$. Let $X$ be the square complex with $1$-skeleton $X^1$ and squares that are the $4$-cycles $\gamma$ having a common neighbour in $\Delta$, called an \emph{apex} of $\gamma$. 
	Then $X$ is a $353$-square complex.
\end{thm}

\subsection{Links}
\label{subsec:353simplicial}
We need a series of preparatory observations on the link $\Gamma=\lk(d,\Delta)$. By reversing the order $<$ we have the obvious analogues of all the results in this subsection for $\lk(a,\Delta)$.
From Definition~\ref{def:353simplicial1}(\ref{353order},\ref{35346}), we obtain:

\begin{rem} 
		\label{rem:53}
		Let $\gamma$ be an induced embedded $n$-cycle in $\Gamma$.
		\begin{enumerate}[(i)]
			\item If $n=4$, then $\gamma=a_1c_1a_2c_2$ and there is a common neighbour $b$ of all $a_1,c_1,a_2,c_2$.
			\item If $n=6$, then $\gamma$ has three vertices in $\mathcal A$ and they have a common neighbour.
		\end{enumerate}
	\end{rem}

\begin{cor}
	\label{cor:6} 
	There is no cycle $\gamma=u_1c_1bc_2u_2c$ in $\Gamma$ with $u_1c_1bc_2u_2$ embedded and induced. 
\end{cor}
\begin{proof} Since $u_1c_1bc_2u_2$ is embedded and induced, we have that $\gamma$ is embedded. By Remark~\ref{rem:53}(ii), we have that $\gamma$ is not induced. 
	Thus $c$ is a neighbour of $b$, which contradicts Remark~\ref{rem:53}(i).
\end{proof}

	\begin{cor}
		\label{cor:8}
		Let $\gamma=a_1c_1bc_2a_2c_3ac_4$ be a cycle in $\Gamma$ with $a_1c_1bc_2a_2$ embedded and induced. Then $b$ is a neighbour of $a$.
	\end{cor}
	\begin{proof} By Definition~\ref{def:353simplicial1}(\ref{3538}), we have that $\gamma$ is not embedded or not induced. 
		
		Suppose first that $\gamma$ is not embedded. Since $a_1c_1bc_2a_2$ was embedded and induced, we have $c_3\neq c_1$, and by Corollary~\ref{cor:6}, we have $a\neq a_1,a_2$ and $c_3\neq c_4$. Thus without loss of generality we can assume $c_3=c_2$. If $c_1=c_4$, then the corollary follows from Remark~\ref{rem:53}(i) applied to the cycle $c_1bc_2a$. If $c_1\neq c_4$, then we argue in the same way as in the proof of Lemma~\ref{lem:restriction order}, using Definition~\ref{def:353simplicial1}(\ref{353order},\ref{35346}).

		Suppose now that $\gamma$ is embedded but not induced. If $a$ is a neighbour of $c_2$ (or~$c_1$), then we argue as above. Note that $c_3$ is not a neighbour of $b$, since this would contradict Remark~\ref{rem:53}(i) applied to the cycle $bc_2a_2c_3$.
		Finally, if $c_3$ is a neighbour of~$a_1$, then this contradicts Corollary~\ref{cor:6}. Up to replacing $c_3$ by~$c_4$, this exhausts all the possibilities.
	\end{proof}

\begin{cor}
	\label{cor:10}
	Let $\gamma=
	c_1b_1c_2a_2c_3b_3c_4b_4c_5a_5$ 
	be a cycle in $\Gamma$ with the paths of length~$4$ centred at each $b_i$ embedded and induced. Then there is $a$ that is a neighbour of all~$b_i$. 
\end{cor}
\begin{proof}  Suppose by contradiction that such $a$ does not exist. Then by Definition~\ref{def:353simplicial1}(\ref{35310}), we have that $\gamma$ is not embedded or not induced.

	Suppose first that $\gamma$ is not embedded. If $c_1=c_3,c_2=c_5,$  or $b_1=b_3$ (or~$b_4$), then this contradicts Corollary~\ref{cor:6}. If $c_4=c_1$ or $c_3=c_5$, then this contradicts the assumption that $b_3c_4b_4c_5a_5$ is embedded and induced. We obtain an analogous contradiction for $c_2=c_4$. If $c_2=c_3$, then let $a_i$ be a neighbour of $b_i$ for $i=1,3$. If $a_3c_4b_4c_5a_5$ is embedded and induced, then, applying Corollary~\ref{cor:8} to the cycle $a_3c_4b_4c_5a_5c_1a_1c_2$, we obtain that $a_1$ is a neighbour of $b_4$. Hence $a_1$ is a neighbour of $c_4$ by Definition~\ref{def:353simplicial1}(1). By Definition~\ref{def:353simplicial1}(2) applied to $a_1c_4b_3c_2$, we obtain that $a_1$ is a neighbour of $b_3$, and so we can take $a=a_1$. If $a_3c_4b_4c_5a_5$ is not embedded, then $a_3=a_5$,
    contradicting the hypothesis that $b_3c_4b_4c_5a_5$ is induced.
   If $a_3c_4b_4c_5a_5$ is embedded but not induced, then $a_3$ is a neighbour of $b_4$ or $c_5$, which implies that $a_3$ is a neighbour of both $b_4$ and $c_5$ by Definition~\ref{def:353simplicial1}(1,2). By Corollary~\ref{cor:6} applied to the 6-cycle $a_3c_2b_1c_1a_5c_5$, we obtain that $a_3c_2b_1c_1a_5$ is not induced. Since $c_2b_1c_1a_5$ is embedded and induced, the only possibility is that $a_3$ is a neighbour of $b_1$ or $c_1$, which implies that $a_3$ is a neighbour of both $b_1$ and $c_1$ as before. Hence we can take $a=a_3$. 
The case of $c_1=c_5$ is analogous. 
	
	Second,  suppose that $\gamma$ is embedded but not induced. By Corollary~\ref{cor:6}, we have that $a_2$ is not a neighbour of $c_5$ and $b_3$ is not a neighbour of $c_1$. If $b_3$ is a neighbour of~$c_2$ or $b_1$ is a neighbour of~$c_3$, then Remark~\ref{rem:53}(i) implies $c_2=c_3$, which is a contradiction.
    If $b_1$ is a neighbour of $c_4$, but not of $c_3$, then the path $a_2c_3b_3c_4b_1$ is embedded and induced, and so this contradicts Corollary~\ref{cor:6} applied to $a_2c_3b_3c_4b_1c_2$. Up to symmetries, this exhausts all the possibilities.
\end{proof}

\subsection{Proof of Theorem~\ref{thm:353}}

\begin{rem}
	\label{rem:centre} Let $ada'd'$ be an embedded cycle in $\Delta$ with common neighbour $v$.
	Suppose that $a,d$ and $a'$ have a common neighbour~$b$. Then $b$ is a neighbour of $d'$. 
	Indeed, we apply Remark~\ref{rem:53}(i) to the link of $d$, and the $4$-cycle $aba'v$.  If $v\in \mathcal B$, then $b=v$, which is a neighbour of $d'$. If $v\in \mathcal C$, then $b$ is a neighbour of $v$. Since $v$ is a neighbour of $d'$, we have $b<v<d'$, and so $b<d'$ by Definition~\ref{def:353simplicial1}(\ref{353order}).
\end{rem}

\begin{proof}  First note that since $\Delta$ is a simplicial complex of type $\{a,b,c,d\}$, it follows from Definition~\ref{def:353simplicial1}(\ref{353order}) that $X$ is connected. We now justify that $X$ is simply connected. Let $\alpha$ be a cycle in $X^1$, and view it as a cycle in $\Delta^1$. Let $D\to \Delta$ be a minimal disc diagram bounded by $\alpha$. We first justify that we can assume that there is no edge $bc$ in $D$. Suppose that there were such an edge, in triangles $v_0bc, v_nbc$. By Definition~\ref{def:353simplicial1}(\ref{353order}), the link of $bc$ is complete bipartite, which is connected, and so it contains a path $v_0v_1\cdots v_n$. We can then replace in $D$ the above two triangles by the union of the triangles $v_iv_{i+1}b,v_iv_{i+1}c$ over $0\leq i<n$. Repeating this procedure removes each edge $bc$ from~$D$. Then the set of the $2$-cells of $D$ can be partitioned into subsets consisting of the $2$-cells belonging to each of the stars around the vertices in $\mathcal B$ and~$\mathcal C$. Each link of such a vertex $x$ in $D$ has vertices in $\mathcal A\cup \mathcal D$. Since in $X^1$ there is an edge between any such $a$ and $d$, we can replace the open star of $x$ in $D$ by a square or a union of squares with apex $x$. Consequently, $D\to \Delta$ can be replaced by a disc diagram in~$X$.

	It remains to verify parts (1)-(4) of Definition~\ref{def:axioms}. For part~(\ref{def:axioms1}), consider squares $ada'd_1, ada'd_2$. First assume that they have apices $b_1,b_2$. Then applying Remark~\ref{rem:53}(i) to the link of $d$ and the $4$-cycle $ab_1a'b_2$, we see that $b_1=b_2$, as desired. Second, suppose that they have apices $c_1,c_2$. Then, by Remark~\ref{rem:53}(i), there is a neighbour $b$ of all $a,c_1,a',c_2$, which then is a neighbour of all $a,d_1,a',d_2$. Third, suppose that they have apices $c_1,b_2$. Then, by Remark~\ref{rem:53}(i), we have that $c_1$ and $b_2$ are neighbours and so $b_2$ is a neighbour of all $a,d_1,a',d_2$.
	
	For the remaining parts, consider a minimal cube corner $C$ with centre $a$, boundary $d_1a_3d_2a_1d_3a_2$ and square apices $v_1,v_2,v_3$. Then $d_1v_3d_2v_1d_3v_2$ is a cycle in the link of $a$. Thus, by Definition~\ref{def:353simplicial1}(\ref{35346}), there is $b$ that is a common neighbour of all $d_i$ and~$a$. 
	
We claim that for each $j\neq i$ the apex $v_i$ is neither a neighbour of $d_i$ nor of $a_j$. Indeed, if, say, $v_2$ is a neighbour of $d_2$, then by Definition~\ref{def:353simplicial1}(\ref{353order}) $a_2$ is a neighbour of $d_2$. Hence $a_2d_2ad_1$ and $a_2d_2ad_3$ are squares of $X$ (with apex $v_2$). By part (1), $a_2d_2a_3d_1$ and $a_2d_2a_1d_3$ are squares of $X$, contradicting the minimality of $C$. If, say, $v_2$ is a neighbour of $a_3$, then $a_2d_3a_3d_1, ad_3a_3d_1$ are squares of $X$ (with apex $v_2$). By considering the sequence $ad_3a_3d_1, ad_2a_3d_1,ad_2a_3d_3,ad_2a_1d_3,a_3d_2a_1d_3$ and repeatedly applying part (1), we obtain that all these $4$-cycles are squares.
This contradicts again the minimality of $C$, and justifies the claim. In particular, all $v_i$ are distinct.

	Then we can assume that $v_i$ belong to $\mathcal C$, since if, say, $v_1=b'$, then it would have to be distinct from $b$, and so, by Remark~\ref{rem:53}(i) applied to the link of $a$, there would be~$c$ that is a neighbour of $b,d_2,b',d_3,$ and $a$. Then $c$ would be also an apex of $ad_2a_1d_3$ by Remark~\ref{rem:centre} with the roles of $\mathcal A,\mathcal D$ interchanged. We will thus write $c_i$ instead of~$v_i$. By Remark~\ref{rem:53}(i), we have that $c_i$ is a neighbour of $b$. Note that since $c_i$ is not a neighbour of $a_j$, for $j\neq i$, we have by Definition~\ref{def:353simplicial1}(\ref{353order}) that $b$ is not a neighbour of~$a_j$.
Consequently, the path $a_1c_1bc_2a_2$ is an induced embedded path in the link of~$d_3$.
	
	\begin{figure}[h]
		\centering
		\includegraphics[scale=0.86]{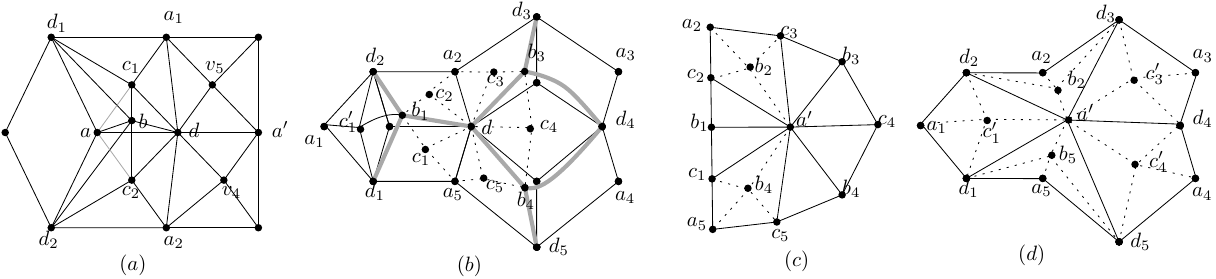}
		\caption{}
		\label{fig:353simplicial}
	\end{figure}
	
	For part~(\ref{def:axioms5}), let $v$ be an apex of $a_1da_2d'$. Let $c_1,c_2$ be apices of $ada_1d_1, ada_2d_2$ guaranteed by the above paragraph (note that the labelling of the $\mathcal D$ vertices of~$C$ changed). Since $a_1c_1bc_2a_2$ is an induced embedded path in the link of $d$, this contradicts Corollary~\ref{cor:6} applied to $a_1c_1bc_2a_2v$.
	
	For part~(\ref{def:axioms3}), and $v_4,v_5$ the apices of the two last squares, we similarly obtain in the link of $d$ a cycle $\gamma=a_1c_1bc_2a_2v_4a'v_5$, see Figure~\ref{fig:353simplicial}(a). After possibly replacing the~$v_i$ by the elements of~$\mathcal C$, by Corollary~\ref{cor:8} we obtain that $b$ is a neighbour of~$a'$. Thus $d_1,d_2$ are neighbours of $a'$  by Definition~\ref{def:353simplicial1}(\ref{353order}), as desired. The required $4$-cycles are squares with apices~$b$.
	
	In part~(\ref{def:axioms4}), we label the boundary $10$-cycle of $E$ by $d_1a_1\cdots d_5a_5$. We similarly obtain the indicated $10$-cycle 
	$c_1b_1c_2a_2c_3b_3c_4b_4c_5a_5$, see Figure~\ref{fig:353simplicial}(b). By Corollary~\ref{cor:10}, 
$b_1,b_3,b_4$ have a common neighbour $a'$ in the link of $d$. By Definition~\ref{def:353simplicial1}(\ref{353order}) $a'$ is a neighbour of all the $c_i$, see Figure~\ref{fig:353simplicial}(c). Again by Definition~\ref{def:353simplicial1}(\ref{353order}), $a'$ is a neighbour of all the $d_i$, see Figure~\ref{fig:353simplicial}(d). For $i=1,3,4$, the $4$-cycles $a'd_ia_id_{i+1}$ are squares, since they have common neighbours 
$c'_i$, where $c'_i$ is a common neighbour of $a_i,b_i,d_i,d_{i+1}$ in Figure~\ref{fig:353simplicial}(b).
As for the $4$-cycle $a'd_2a_2d_3$, it has either a common neighbour~$c_2$, in the case where $c_2=c_3$, or, in the case where $c_2\neq c_3$, a common neighbour~$b_2$, where $b_2$ is a common neighbour of $a',c_2,a_2,c_3$ in the link of $d$, guaranteed by Remark~\ref{rem:53}(i). Analogously, the $4$-cycle $a'd_5a_5d_1$ is a square.
	 Part~(5) is proved analogously.
\end{proof}

\subsection{Stability and contractibility}

\label{subsec:353stability}

\begin{lem}
	\label{lem:upward flag}
	Let $\Delta$ be a wide  
    $353$-simplicial complex. Then the relation $<$ on $\Delta^0$ is a partial order that is bowtie free.
		Furthermore, $\mathcal P=\mathcal A\cup\mathcal B\cup\mathcal C$  
	 is bowtie free and upward flag. Moreover, if in $K\subset \mathcal P$ each pair has an upper bound in~$\mathcal P$, then $K$ has the join in $\mathcal P$.
\end{lem}

\begin{proof}
	By Corollary~\ref{cor:squares} and Theorem~\ref{thm:353}, each embedded 4-cycle $ada'd'$ in $\Delta$ has a common neighbour in $\mathcal B\cup\mathcal C$. Thus the first assertion of the lemma follows from \cite[Lem~8.1]{huang2024}. In particular, $\mathcal P$ is also bowtie free. 
	
Now we show that $\mathcal P$ is upward flag. Let $u_1,u_2,u_3\in \mathcal P$ be pairwise upper bounded. We can assume that none of them are neighbours. For each $i$, let $a_i$ be a neighbour of or equal to $u_i$. Since $\Delta$ is wide, we have that $a_i$ are pairwise close in $X$, and so, by Corollary~\ref{cor:exists_a_easy}, there is a common neighbour $d$ of $a_1,a_2,a_3$. 
	
	We claim that each $u_i$ is also a neighbour of $d$. To justify the claim for, say, $u_1=b_1$, let $b_1,u_2\leq c_2$ and $b_1,u_3\leq c_3$. Applying the bowtie freeness of $\Delta^0$ to 
$da_1c_ia_i$, for $i=2,3$, we obtain that either $c_i$ is a neighbour of $d$ or $da_1c_ia_i$ have a common neighbour $b_i$. In each case, $d,c_2,c_3$ are pairwise lower bounded in $\lk(a_1,\Delta)^0$, hence they have a common lower bound $b'$ by Definition~\ref{def:353simplicial1}(\ref{35346}). Considering the cycle $c_2b_1c_3b'$, we obtain $b_1=b'$, justifying the claim.

Since $\Delta^0$ is bowtie free, $u_1,u_2,u_3$ are pairwise upper bounded in $\lk(d,\Delta)^0$. By Definition~\ref{def:353simplicial1}(\ref{35346}), they have a common upper bound in $\lk(d,\Delta)^0$, as desired.
 
The last assertion of the lemma follows from Lemma~\ref{lem:posets} and \cite[Lem~6.2]{haettel2021lattices}.
\end{proof}

\begin{lem} 
	\label{lem:stable}
	Let $\Delta$ be a wide $353$-simplicial complex and let $X$ be its $353$-square complex constructed in Theorem~\ref{thm:353}. 
	\begin{enumerate}
		\item $X$ is stable.
		\item The vertices of each simplex of $X^\boxtimes$ have a common neighbour $b$ or $c$ in $\Delta$. In particular, $X$ is wide.
	\end{enumerate}
	
\end{lem}
\begin{proof} 
	For (1), assume without loss of generality $S=A\subset \mathcal A$ and $|A|=\infty$. Then $A$ is pairwise upper bounded in $\mathcal P$. By Lemma~\ref{lem:upward flag}, $A$ has a join $u\in \mathcal B\cup\mathcal C$ in $\mathcal P$. 
	Hence $u$ is also the join of $A$ in $\Delta^0$.
	
	Suppose first $u=b\in \mathcal B$. Choose $a_1\neq a_2$ from $A$ and let $S'=\{a_1,a_2\}$. If $a$ is close to each element of $S'$, then $a,a_1,a_2$ are pairwise upper bounded in $\mathcal P$, hence they have a join $u'\in \mathcal P$. If $u'\in \mathcal B$, then applying the bowtie freeness of Lemma~\ref{lem:upward flag} to $a_1ba_2u'$ we obtain $u'=b$. Then $a$ is a neighbour of $b$, and so $a$ is close to each element of $A$ since $b$ has at least two neighbours in $\mathcal D$. If $u'=c'\in\mathcal C$, then applying the bowtie freeness to $a_1ba_2c'$ we obtain that $b$ is a neighbour of $c'$. Then, by Definition~\ref{def:353simplicial1}(\ref{353order}), each element of $A$ is a neighbour of $c'$. Thus $a$ is close to each element of $A$. If $d$ is a neighbour of each element of $S'$, then applying the bowtie freeness to $a_1ba_2d$, we obtain that $b$ is a neighbour of $d$, and by Definition~\ref{def:353simplicial1}(\ref{353order}) $d$ is a neighbour of each element of $A$.
	
	Second, suppose $u=c\in \mathcal C$. Assume first that the join of each three element subset of $A$ belongs to $\mathcal B$. Then for each $A_1,A_2\subset A$ with $|A_1|=|A_2|=3$ and $|A_1\cap A_2|= 2$, the joins of $A_1,A_2$ are equal.
Consequently, $u\in\mathcal B$, which is a contradiction. Finally, suppose that there is $S'\subset A$ with $|S'|=3$ and join $c'\in \mathcal C$. Since $c'\le c$, we have $c'=c$. If $d$ is a neighbour of each element of $S'$, then $c<d$, and so $d$ is a neighbour of each element of $A$. If $a$ is close to each element of $S'$, then Lemma~\ref{lem:upward flag} implies that $S'\cup\{a\}$ has a join $u'\in \mathcal P$. Since $c\le u'$, we have $u'=c$ and so $a$ is a neighbour of $c$. Hence $a$ is close to each element of $A$.
	
	Part (2) is proved similarly.
\end{proof}

A $353$-simplicial complex $\Delta$ is \emph{non-degenerate} if for each edge $bc$ there is $d\in \lk(b,\Delta)^0$ that is not a neighbour of $c$, and there is $a\in \lk(c,\Delta)^0$ that is not a neighbour of $b$.
		\begin{prop}
			\label{prop:he}
			Let $\Delta$ be a wide non-degenerate $353$-simplicial complex. Then
			$X^\boxtimes$ is homotopy equivalent to $\Delta$. 
		\end{prop}
		
		\begin{proof}
			For each $u\in \mathcal B\cup \mathcal C$, let $\phi(u)$ be the simplex of $X^\boxtimes$ spanned by all the neighbours of~$u$ in $\mathcal A\cup\mathcal D$. Note that $\phi(u)$ is a maximal simplex in $X^\boxtimes$. Indeed, otherwise by Lemma~\ref{lem:stable}(2) we would have $\phi(u)\subsetneq \phi(u')$. Then $u$ and $u'$ would be neighbours contradicting the non-degeneracy for the edge $uu'$.
Similarly, the function $u\to \phi(u)$ is a bijection 
from $\mathcal B\cup\mathcal C$ to the family of the maximal simplices of $\Delta$.

			We claim that for any subset $U\subset \mathcal B\cup \mathcal C$, the intersection $\bigcap_{u\in U}\st(u,\Delta)$ is empty or contractible. Indeed, if $v\in \bigcap_{u\in U}\st(u,\Delta)^0$, then $v$ is a lower bound or an upper bound for $U$, say the latter. By Lemma~\ref{lem:posets} and Lemma~\ref{lem:upward flag}, $U$ has a join $M$, which belongs to $\bigcap_{u\in U}\st(u,\Delta)$. Furthermore, any vertex of $\bigcap_{u\in U}\st(u,\Delta)$ is $\le M$, justifying the claim. 
			
		Thus for any $U\subset \mathcal B\cup \mathcal C$, $\bigcap_{u\in U}\st(u,\Delta)\neq \emptyset$ if and only if $\bigcap_{u\in U}\phi(u)\neq\emptyset$. Since~$\Delta$ is covered by the closed stars of the vertices in $\mathcal B\cup \mathcal C$, and $X^\boxtimes$ is covered by its maximal simplices, it remains to invoke the Nerve Theorem \cite{borsuk1948imbedding}, see also the version in \cite[Thm~6]{bjorner2003nerves}.
		\end{proof}

		By Lemma~\ref{lem:stable}, Theorems~\ref{thm:353} and~\ref{thm:contractible intro}, and Proposition~\ref{prop:he}, we have the following.
		
		\begin{cor} 
			\label{cor:contr}
			Let $\Delta$ be a non-degenerate wide $353$-simplicial complex. Then $\Delta$ is contractible.
		\end{cor}

\subsection{353 Artin complex}
\begin{theorem}
		\label{lem:main example}
		Let $\Lambda=abcd$ be the Coxeter diagram that is the linear graph with $m_{ab}=m_{cd}=3$ and $m_{bc}=5$. 
        Then the Artin complex $\Delta=\Delta_\Lambda$ is a wide non-degenerate $353$-simplicial complex.
	\end{theorem} 

  As usual, we denote by $\mathcal A$ the set of vertices of type $\hat a$, etc.
	\begin{proof}
		The simple connectedness of $\Delta$ follows from \cite[Lem 4]{cumplido2020parabolic}.
		By Remark~\ref{rem:adj}, we have Definition~\ref{def:353simplicial1}(1). 
  By \cite[Lem 6]{cumplido2020parabolic}, we can identify each $\lk(a,\Delta), \lk(d,\Delta)$, with the Artin complex for the Coxeter subdiagram $bcd$ or $abc$. Hence Definition~\ref{def:353simplicial1}(2) follows from 
    Theorem~\ref{thm:flag}, Definition~\ref{def:353simplicial1}(3) follows
    from Proposition~\ref{prop:Ia}, and 
    Definition~\ref{def:353simplicial1}(4) follows from 
    Proposition~\ref{prop:IIa}. Thus $\Delta$ is a $353$-simplicial complex. 
    
    Each vertex in $\mathcal B\cup\mathcal C$ has infinitely many neighbours in $\mathcal A$ and in $\mathcal D$, and so $\Delta$ is wide.

Let $x\in \Delta^0$ be a vertex of type $\hat b$. Then, by Remark~\ref{rem:adj}, $\lk(x,\Delta)$ is a join $K_1*K_2$ where $K_1$ is the full subcomplex spanned by the vertices of type $\hat a$, and $K_2$ is the full subcomplex spanned by the vertices of type $\hat c$ and $\hat d$. By \cite[Lem 6]{cumplido2020parabolic}, we have $K_2\cong \Delta_{\Lambda'}$, where $\Lambda'\subset \Lambda$ is the edge $cd$. 
By considering the simplicial map $\pi$ from the Artin complex~$\Delta_{\Lambda'}$ to the Coxeter complex $\bC_{\Lambda'}$, which is a circle formed of 6 edges, we obtain that for each vertex $z$ of type $\hat c$ in $K_2$, there is a vertex of type $\hat d$ in $K_2$ that is not a neighbour of $z$. This confirms the first part of the definition of the non-degeneracy of $\Delta$. The second part is analogous.
	\end{proof}

\section{Relative Artin complexes and related background}
\label{sec:relative}
\subsection{Relative Artin complexes}
\begin{defin}[\cite{huang2023labeled}]
\label{def:rel}
	Let $\Lambda'\subset \Lambda$ be an induced subdiagram.
	The \emph{$(\Lambda,\Lambda')$-relative Artin complex $\Delta_{\Lambda,\, \Lambda'}$} is 
    the induced subcomplex of the Artin complex $\Delta_\Lambda$ spanned by vertices of type $\hat s$ with $s$ a vertex of $\Lambda'$. 
\end{defin}

\begin{lem}[{\cite[Lem 6.2]{huang2023labeled} and \cite[Lem 4]{cumplido2020parabolic}}]
	\label{lem:sc}
	If $|\Lambda'|\ge 3$, then $\Delta_{\Lambda,\, \Lambda'}$ is 
    simply connected (in particular, it is connected).
\end{lem}

Note that $\Delta_{\Lambda,\, \Lambda'}$ is a simplicial complex of type $S$ (see Section~\ref{subsec:chamber complex}) with $S=\{\hat s\}_{s\in \Lambda'}$.

\begin{defin}
	\label{def:admissible}
	An induced  subdiagram $\Lambda'$ of $\Lambda$ is \emph{admissible} if for any vertex $x$ of $\Lambda'$, if the vertices $x_1,x_2$ of $\Lambda'$ are in distinct connected components of $\Lambda'\setminus\{x\}$, then they are in distinct connected components of $\Lambda\setminus\{x\}$.
\end{defin}

\begin{lem}[{\cite[Lem 6.6]{huang2023labeled}}]
	\label{lem:poset structure}

Suppose that $\Lambda'=s_1\cdots s_n$ is an admissible linear subgraph of a Coxeter diagram $\Lambda$. Let $\Delta'$ be the $(\Lambda,\Lambda')$-relative Artin complex, with the relation~$<$ on its vertex set induced from $s_1<\cdots <s_n$ or $s_n<\cdots <s_1$. Then $({\Delta'}^0,<)$ is a weakly graded poset.
\end{lem}

\begin{defin}
	\label{def:bowtie free}
	Let $\Delta'$ be as in Lemma~\ref{lem:poset structure}. 
    We say that 
    $\Delta'$ is \emph{bowtie free} (resp., \emph{flag}, or \emph{weakly flag}) if $({\Delta'}^0,<)$ is bowtie free, (resp.\ flag, or weakly flag). 
    Note that these definitions do not depend on the choice of  one or the other total order on $\Lambda'$.
\end{defin}

\begin{lem}[{\cite[Lem 6.4(1)]{huang2023labeled}}]
	\label{lem:link}
	Let $v\in \Delta'=\Delta_{\Lambda,\, \Lambda'}$ be a vertex of type~$\hat s$.
	Then there is a type-preserving isomorphism between $\lk(v,\Delta')$ and $\Delta_{\Lambda \setminus \{s\},\, \Lambda' 
    \setminus \{s\}}$.
\end{lem}

We have the following consequence. 
\begin{lem}
	\label{lem:dr}
	Let $\Lambda'\subset \Lambda$ be an induced subdiagram.
	\begin{enumerate}
	\item
	Let $s$ be a vertex of $\Lambda'$. 
	If $\Delta_{\Lambda \setminus \{s\},\,\Lambda'\setminus \{s\}}$ is contractible, then $\Delta_{\Lambda,\, \Lambda'}$ deformation retracts onto  $\Delta_{\Lambda,\, \Lambda'\setminus \{s\}}$.
	\item 
	More generally, let $T$ be a subset of the vertex set of $\Lambda'$. If $\Delta_{\Lambda\setminus R,\, \Lambda' \setminus R}$ is contractible for each non-empty subset $R$ of $T$, then $\Delta_{\Lambda,\, \Lambda'}$ deformation retracts onto $\Delta_{\Lambda,\, \Lambda'\setminus T}$.
	\end{enumerate}
\end{lem}
\begin{proof}
Part (1) is \cite[Lem 7.1]{huang2023labeled} in view of Lemma~\ref{lem:link}. We prove part (2) by induction on $|T|$. For $s\in T$, by the inductive assumption, $\Delta_{\Lambda,\, \Lambda'}$ deformation retracts onto  $\Delta_{\Lambda,\, \Lambda'\setminus T\cup \{s\}}$. 
It remains to prove that $\Delta_{\Lambda,\, \Lambda'\setminus T\cup \{s\}}$ deformation retracts onto $\Delta_{\Lambda,\, \Lambda'\setminus T}$. This will follow from part (1) once we verify that $\Delta_{\Lambda\setminus \{s\},\, \Lambda'\setminus T}$ is contractible. 
This follows from the assumption that $\Delta_{\Lambda\setminus \{s\},\,\Lambda'\setminus \{s\}}$ is contractible, since by the inductive assumption it deformation retracts to $\Delta_{\Lambda\setminus \{s\},\,\Lambda'\setminus T}$. 
\end{proof}

\subsection{Properties of some relative Artin complexes}

\begin{lem}[{\cite[Lem~6.16]{huang2023labeled}}]
	\label{lem:4cycle}
	Suppose that $\Lambda$ is an arbitrary Coxeter diagram.
	Let $\omega=x_1\cdots x_4$ be an embedded $4$-cycle in $\Delta_\Lambda$ of type $\hat s_1\cdots \hat s_4$. Suppose $\hat s_1\neq \hat s_3$. Then there exists a vertex $x'_3\in\Delta_\Lambda$ of type $\hat s_1$ that is a common neighbour of $x_2,x_3,$ and $x_4$. 
\end{lem}

\begin{cor}
	\label{cor:4cycle}
	Suppose that $\Lambda$ is an arbitrary Coxeter diagram with an edge $s_1s_2$ such that $\Delta_{\Lambda,\, s_1s_2}$ has girth $\geq 6$.
Let $\omega$ be an embedded $4$-cycle in $\Delta_\Lambda$ with an edge of type $\hat s_1\hat s_2$. Then $\omega$ not induced. 
\end{cor}

\begin{proof}
Let $\omega=x_1\cdots x_4$ with $x_i$ of type $\hat s_i$. Since $\omega$ is embedded, the girth hypothesis implies that we cannot have simultaneously $s_1=s_3$ and $s_2=s_4$. Assume first $s_1\neq s_3$ and $s_2=s_4$. By Lemma~\ref{lem:4cycle}, there is a vertex $x'_3$ of $\Delta_\Lambda$ of type $\hat s_1$ that is a common neighbour of $x_2,x_3$ and $x_4$. Then $x_1x_2x'_3x_4$ is a 4-cycle in $\Delta_{\Lambda,\, s_1s_2}$. Since $x_2\neq x_4$, we must have $x_1=x'_3$, implying that $x_1$ is a neighbour of $x_3$ and so $\omega$ is not induced. The case  $s_1= s_3$ and $s_2\neq s_4$ is analogous. 

Now assume $s_1\neq s_3$ and $s_2\neq s_4$. Let $x'_3$ of type $\hat s_1$ be chosen as before. If $x_1x_2x'_3x_4$ is not embedded, then we deduce that $x_1$ is a neighbour of $x_3$ as before. If $x_1x_2x'_3x_4$ is embedded, then $x_2$ is a neighbour of $x_4$ by the case of $s_1= s_3$ and $s_2\neq s_4$.
\end{proof}

\begin{theorem}[{\cite[Thm 8.1]{huang2023labeled}}]
	\label{thm:4cycle}
	Let $\Lambda$ be irreducible spherical, and let $\Lambda'\subset\Lambda$ be a linear subdiagram. Then $\Delta_{\Lambda,\, \Lambda'}$ is bowtie free.
\end{theorem}

\subsection{Haettel contractibility criteria}
\label{subsec:contractible}
Let $S=\{s_1,\ldots, s_n\}$ be a totally ordered set. Let $X$ be a simplicial complex of type $S$ with the induced relation on its vertex set, as defined in Section~\ref{subsec:chamber complex}:
for vertices $x,x'\in X^0$ we write $x<x'$ if $x,x'$ are neighbours and $\type(x)<\type(x')$. The following is a consequence of \cite[\S4.3, Thm~B]{haettel2022link} and \cite[Thm~1.15]{haettel2021lattices}.

\begin{thm}
	\label{thm:contractibleII}
	Let $X$ be a simplicial complex of type $S$.  Assume that
	\begin{enumerate}
		\item $X$ is simply connected,
		\item the relation $<$ on $X^0$ is a partial order,
		\item for each $x\in X^0$, the collection of vertices $\ge x$ is bowtie free and upward flag, and
		\item for each $x\in X^0$, the collection of vertices $\le x$ is bowtie free and downward flag.
	\end{enumerate}
	Then $X$ is contractible. 	
\end{thm}

\begin{prop}[{\cite[Lem 5.1]{huang2024}}]
\label{prop:434} 
Suppose that $X$ satisfies the assumptions of Theorem~\ref{thm:contractibleII}. Then $(X^0,<)$ 
is bowtie free and flag. 
\end{prop}

Now we discuss a variation of Theorem~\ref{thm:contractibleII}.

\begin{defin}
	\label{def:ha}
Let $S$ be cyclically ordered with $s_1<s_2<\cdots< s_n<s_1$. 
For each vertex $x$ of $X$ of type $s_i$, we consider the relation $<_x$ on $\lk(x,X)^0$ as follows. 
The cyclic order induces an order on $S\setminus\{s_i\}$ by declaring $s_{i+1}<\cdots<s_n<s_1<\cdots < s_{i-1}$. 
For vertices $y,z\in \lk(x,X)^0$, define $y<_x z$ if $y,z$ are neighbours and $\type(y)<\type(z)$ in $S\setminus \{s_i\}$. 

We say that $X$ is an \emph{$\widetilde A_n$-like complex} if 
	\begin{enumerate}
		\item $X$ is simply connected,
		\item for each $x\in X^0$, the relation $<_x$ on $\lk(x,X)^0$ is a partial order, and
		\item for each $x\in X^0$, the relation $<_x$ on $\lk(x,X)^0$ is bowtie free.
	\end{enumerate}

\end{defin} 

For example, the Coxeter complex of type $\widetilde A_n$ is an $\widetilde A_n$-like complex.

The following is a consequence of \cite[\S4.2, Thm~A]{haettel2022link}. It also follows from \cite[Thm 3.3 and \S8]{bessis2006garside} and \cite{Bestvina1999}, since $\widetilde A_n$-like complexes are Bestvina complexes for a certain Garside groupoid.

\begin{thm}
\label{thm:contractibleI}
	Each $\widetilde A_n$-like complex is contractible. 
\end{thm} 

\begin{lem}
	\label{lem:4-cycle}
	Let $X$ be an $\widetilde A_n$-like complex. Then any induced $4$-cycle in $X^1$ has a common neighbour. In particular, $X$ has no embedded cycles of type $\hat s_1 \hat s_2 \hat s_1 \hat s_2$.
\end{lem}
\begin{proof}
The first assertion is \cite[Lem 4.8]{huang2023labeled}. For the second assertion, if a cycle had type $\hat s_1 \hat s_2 \hat s_1 \hat s_2$ and common neighbour $x$, then it would be a bowtie in the link of $x$.  
\end{proof}

\section{Bestvina convexity in $\widetilde A_n$-like complexes}
\label{sec:convex}

\subsection{Garside complexes}
\begin{defin}[{\cite[Def~4.6]{haettel2024lattices}}]
Let $\widehat X$ be a simply connected flag simplicial complex. Suppose that we have a binary relation $<$ on $\widehat X^0$ (not necessarily a partial order) such that vertices $x,y$ are neighbours exactly when $x<y$ or $y<x$. Furthermore, suppose that the transitive closure of $<$ is a partial order that is weakly graded with rank function $r$. We write $x\leq y$ when $x<y$ or $x=y$.

Assume that we have an automorphism $\varphi$ of $(\widehat X^0,<)$ such that 
\begin{itemize}
\item 
$r\circ \varphi=t\circ r$, for a translation $t\colon \mathbb Z\to \mathbb Z$, and
\item
$x\leq y$ if and only if $y\leq \varphi(x)$, for all $x,y\in \widehat X^0,$ and 
\item 
the interval $[x,\varphi(x)]=\{z\in \widehat X^0\mid x\le z\ \mathrm{\and}\ z\le \varphi(x)\}$ is a lattice for all $x\in \widehat X^0$ (in particular, the relation $<$ restricted to $[x,\varphi(x)]$ is transitive).
\end{itemize}
We then call $\widehat X$ a \emph{Garside flag complex.}
\end{defin}

Let $X$ be an $\widetilde A_n$-like complex of type $S$, with cyclic order $s_1<s_2<\cdots<s_n<s_1$. Consider the type function $\tau:X^{0}\to \mathbb Z/n\mathbb Z$ defined by $\tau(x)=i$ for $x$ of type~$s_i$. 
We define a following simplicial complex structure on $\widehat X=X\times \mathbb R$.
The vertex set~$\widehat X^0$ of~$\widehat X$ is 
$$\{(x, i) \in X^0 \times \mathbb Z \mid \tau(x) = i\},$$ 
 The vertices $(x,i)$ and $(x',j)$ are neighbours if $x$ and $x'$ are equal or neighbours in~$X$, and $|i-j|  \leq  n$. Let $\widehat X$ be the flag simplicial complex with that $1$-skeleton. Note that any maximal simplex of $\widehat X$ has vertices $$(x_i,kn+i),(x_{i+1},kn+i+1),\ldots,(x_n,kn+n),(x_1,kn+n+1),
 \ldots,(x_{i},kn+n+i),$$ where $k\in \mathbb Z$, $1\le i\le n,$ and $x_1,x_2,\ldots,x_n$ are vertices of a maximal simplex of $X$ with $\tau(x_i)=i$. 

Note that the map $\widehat X^{0}\to X^{0}$ sending $(x,i)$ to $x$ extends to a simplicial map, denoted by $\pi: \widehat X\to X$. 
Define $r\colon \widehat X^0\to \mathbb Z$ by $r(x,i)=i$, and $\varphi\colon \widehat X^0\to \widehat X^0$ by $\varphi (x,i)=(x,i+n)$.

We define a binary relation $<$ on 
$\widehat X^0$ by requiring $(x,i)< (y,j)$ exactly when 
these two vertices are neighbours in $\widehat X$ and $i< j$. Note that the transitive closure~$\le_t$ 
of~$\le$ on~$\widehat X^0$ is a partial order. 
The following shows that $\widehat X$ is a Garside flag complex with the automorphism $\varphi$.

\begin{lem}
	\label{lem:lattice}
For each $(x,i)\in \widehat X^0$, the interval $[(x,i),(x,i+n)]$ is a lattice.
\end{lem}

Below, the partial order $<_{x}$ on $\lk(x,X)^0$ was introduced in Definition~\ref{def:ha}.

\begin{proof}
    Note that the poset $[(x,i),(x,i+n)]$ is isomorphic with the poset obtained from $(\lk(x,X)^0,<_x)$ by adding the smallest and the greatest element.
    This poset is a lattice by Lemma~\ref{lem:posets}. 
\end{proof}

By \cite[Thm~1.3]{haettel2024lattices}, given $x\in \widehat X^0$, the poset $\{w\in \widehat X^0\mid w\ge_t x\}$ 
is a lattice, and so we can discuss the meet $\wedge$ in that poset.
By \cite[Thm 4.7]{haettel2024lattices}, a Garside flag complex $\widehat X$ is an instance of a
\emph{homogeneous categorical Garside structure}. We decided not to give here the definition, since we will be only using \cite[Prop~4.2]{haettel2024lattices} on the \emph{Deligne normal form} (term introduced in \cite{charney2004bestvina}), which is more convenient for us to state directly in the terms of $\widehat X$:

\begin{thm} 
\label{thm:normal form}
For each $x,y\in \widehat X^0$, there is a unique edge-path $x_1\cdots x_l\cdots x_n$ from $x_1=x$ to $x_n=y$ such that
\begin{itemize}
    \item $x_i<x_{i+1} \neq \varphi(x_i)$ for $1\leq i<l$, and
    \item $x_{i}=x_{i+1}\wedge \varphi(x_{i-1})$ for $1< i<l$ in $[x_i,\varphi(x_i)]$, and
    \item $x_i=\varphi^{\pm (i-l)} (x_l)$ for $l\leq i\leq n$, with all signs positive or all signs negative.
\end{itemize}
\end{thm}

Note that, as all the notions in this section, the Deligne normal form depends on the cyclic order on the set of the types of $X^0$.

\subsection{Bestvina-convexity}

\begin{defin}
	Given an edge-path $P=x_1\cdots x_n$ in $X$, 
    an \emph{admissible lift} of $P$ is an edge-path $\widehat P=\hat x_1\cdots \hat x_n$ in $\widehat X$ 
    such that $\pi(\hat x_i)=x_i$, for $1\le i\le n$, and  $\hat x_i<\hat x_{i+1}$, for $1\le i\le n-1$.
	Note that for each edge-path $P$ in $X$, once a lift $\hat x_1$ of $x_1$ has been chosen, there is a unique admissible lift of $P$ starting at $\hat x_1$. Different admissible lifts of $P$ differ by the translation by $\varphi^k$ for some $k\in \mathbb Z$.
	
	Let $a,b\in X^0$. Following \cite{Bestvina1999,charney2004bestvina}, we say that an edge-path $P$ from $a$ to $b$ is a \emph{geodesic}, (or \emph{B-geodesic}) if some (hence all) admissible lift of $P$ to $\widehat X$ has Deligne normal form with $n=l$. 
\end{defin}

\begin{lem}
	\label{lem:geodesic0}
	For any $a,b\in X^0$, there is a unique $B$-geodesic in $X$ from $a$ to $b$.
\end{lem}

\begin{proof}
	Let $\hat a$ and $\hat b$ be lifts of $a$ and $b$, respectively, i.e.\ $\pi(\hat a)=a$ and $\pi(\hat b)=b$. Let $P=\hat a \cdots x_l\cdots x_n$ be the path in $\widehat X$ from $\hat a$ to $\hat b$ that has Deligne normal form. Then $\pi(x_l)=\pi(x_n)=b$,
    and so $\pi(x_1)\cdots \pi(x_l)$ is a B-geodesic from $a$ to $b$, which proves the existence. 
    
    Suppose that there are two B-geodesics $P_1$ and $P_2$ from $a$ to $b$. Let $\widehat P_1$ and $\widehat P_2$ be admissible lifts of $P_1$ and $P_2$ starting at the same point. Then the endpoints $\hat b_i$ of~$\widehat P_i$ differ by $\varphi^k$ for some $k\in \mathbb Z$. Since $\widehat P_1$ has Deligne normal form, we have that the concatention of $\widehat P_1$ with $\hat b_1 \varphi (\hat b_1)\cdots \varphi^k(\hat b_1)$ also has Deligne normal form. But since the Deligne normal form is unique (Theorem~\ref{thm:normal form}), the latter path equals $\widehat P_2$, and so $k=0$ and $\widehat P_1=\widehat P_2$, hence $P_1=P_2$.
\end{proof}

\begin{lem}
	\label{lem:geodesic}
	Let $P=x_1\cdots x_n$ be an edge-path in $X$. Then $P$ is a B-geodesic if and only if for each $2\le i\le n-1$, the vertices $x_{i-1}$ and $x_{i+1}$ do not have a common lower bound in $(\lk(x_i,X)^0,<_{x_i})$.
\end{lem}

\begin{proof}
	Let $\widehat P=\hat x_1 \cdots \hat x_n$ be an admissible lift of $P$. 
    By the definition of the Deligne normal form, $P$ is a B-geodesic if and only if 
    $\hat x_{i}=\hat x_{i+1}\wedge \varphi(\hat x_{i-1})$ for $1<i<n$. This means exactly that $\hat x_{i+1}$ and $\varphi(\hat x_{i-1})$ have meet $\hat x_i$ in the interval $[\hat x_{i},\varphi(\hat x_{i})]$. Under the isomorphism of $[\hat x_{i},\varphi(\hat x_{i})]$ with the augmented $(\lk(x_i,X)^0,<_{x_i})$ from the proof of Lemma~\ref{lem:lattice}, this means that $x_{i-1}$ and $x_{i+1}$ have only a trivial lower bound, as desired.
\end{proof}

\begin{defin}
	\label{def:Bconvex}
	Let $X$ be an $\widetilde A_n$-like complex as before, of type $S$ with given cyclic order.  
    Let $Y\subset X$ be a full subcomplex that is also a simplicial complex of type $S$ with the induced type function from $X$. 
    We say that $Y$ is \emph{locally B-convex} if for each vertex $y\in Y^0$ and any vertices $y_1,y_2$ of $\lk(y,Y)$, if the meet $y_1\wedge y_2$ in the poset $(\lk(y,X)^0,<_y)$ exists, then $y_1\wedge y_2\in \lk(y,Y)^0$.
\end{defin}

The property of being locally B-convex depends on the choice of the cyclic order on the set of types of $\widehat X^0$. Reversing the cyclic order gives a different $\widetilde A_n$-like complex structure on $X$, with simplicially isomorphic $\widehat X$, but with different collection of locally B-convex subcomplexes.

\begin{prop}
	\label{prop:convex}
	Let $X$ 
    be an $\widetilde A_n$-like complex, 
    and let $Y\subset X$ be a connected locally B-convex subcomplex.  
    Then $Y$ is simply connected, and for any pair of vertices $y_1,y_2\in Y^0$, the B-geodesic in $X$ from $y_1$ to $y_2$ is contained in $Y$.
\end{prop}

\begin{proof}
	Let $\widetilde Y$ be the universal cover of $Y$. We first show that $\widetilde Y$ is an $\widetilde A_n$-like complex. We induce the type function and the cyclic order on the types from $X$.
    It suffices to show for each $y\in Y^0$, the restriction of the relation $<_y$ to the vertex set of $\lk(y,Y)$ satisfies conditions (2) and (3) of Definition~\ref{def:ha}. Condition (2) holds since $Y$ is a full subcomplex of $X$. To check Condition (3), let $x_1y_1x_2y_2$ be a bowtie in $\lk(y,Y)$.
    Since $y_1,y_2$ have a lower bound in $\lk(y,X)^0$, they have a meet $z\in\lk(y,X)^0$  
    by Lemma~\ref{lem:posets}. Then $x_i\le z\le y_j$ for $i,j\in\{1,2\}$.  
    By the local B-convexity, we have $z\in \lk(y,Y)^0$, 
    as desired.
	
	Let $\theta\colon \widetilde Y\to Y$ be the covering map. We claim that if $\widetilde P$ is a B-geodesic in $\widetilde Y$ from~$\tilde y_1$ to~$\tilde y_2$, then $P=\theta(\widetilde P)$ is the B-geodesic in $X$ from $y_1=\theta(\tilde y_1)$ to $y_2=\theta(\tilde y_2)$. 
    
    Let us assume the claim for the moment and finish the proof of the proposition. We first justify that $Y$ is simply connected and $\widetilde Y=Y$. Otherwise, we have distinct lifts $\tilde y,\tilde y'\in \widetilde Y$ of a vertex $y\in Y$. By Lemma~\ref{lem:geodesic0}, there is a B-geodesic $\widetilde P$ from $\tilde y$ to $\tilde y'$ in~$\widetilde Y$. By the claim, $\theta(\widetilde P)$ is a nontrivial B-geodesic in $X$ from $y$ to $y$. This contradicts the uniqueness of the B-geodesic in Lemma~\ref{lem:geodesic0}. Thus $Y$ is simply connected. The remaining assertion of the proposition follows from the claim and the uniqueness of B-geodesics in $X$.

	It remains to prove the claim.
	Let $\tilde y_{i-1},\tilde y_i,\tilde y_{i+1}$ be three consecutive vertices in $\widetilde P$ with $y_{i-1},y_i,y_{i+1}$ their images under $\theta$. Since $\widetilde P$ is a B-geodesic in $\widetilde Y$, Lemma~\ref{lem:geodesic} implies that $\tilde y_{i-1}$ and $\tilde y_{i+1}$ do not have a common lower bound in $\lk(\tilde y_i,\widetilde Y)^0$. Since $\lk(\tilde y_i,\widetilde Y)\cong\lk(y_i,Y)$, we have that $y_{i-1}$ and $y_{i+1}$ do not have a common lower bound in $\lk(y_i,Y)^0$.
    By the local B-convexity of $Y$, $y_{i-1}$ and $y_{i+1}$ do not have a common lower bound in $\lk(y_i,X)^0$. 
    Hence $P$ is a B-geodesic in $X$ by Lemma~\ref{lem:geodesic}, and the claim follows.
\end{proof}

\begin{cor}
	\label{cor:convex}
	Let 
    $X$ be an $\widetilde A_n$-like complex. 
    Let $Y_1$ and $Y_2$ be connected locally B-convex subcomplexes of $X$. 
    If $Y_1\cap Y_2\neq\emptyset$, then $Y_1\cap Y_2$ is connected.
\end{cor}

\begin{proof}
	Let $y$ and $y'$ be vertices in $Y_1\cap Y_2$, and let $P$ be a B-geodesic in $X$ from $y$ to~$y'$. By Proposition~\ref{prop:convex}, we have that $P$ is contained in $Y_1\cap Y_2$.
\end{proof}

\section{$3$-dimensional Artin groups}
\label{sec:3dim}

\begin{defin} We say that a Coxeter diagram $\Lambda$ satisfies the \emph{girth condition} if for each edge $st$ of $\Lambda$, the graph $\Delta_{\Lambda, \, st}$ has girth $\geq 6$.
\end{defin}

The goal of this section is to prove the following.

\begin{thm}
	\label{thm:3d}
	Let $A_\Lambda$ be an Artin group of dimension $\leq 3$. 
    Then $\Lambda$ satisfies the girth condition. If $\Lambda$ is non-spherical, then its Artin complex~$\Delta$ is contractible. \end{thm}

A vertex of a graph is \emph{isolated} if it has no neighbours. An \emph{$n$-cycle} is the graph with vertices $s_1,\ldots, s_n$ and edges $s_1s_2,\ldots, s_{n-1}s_n,s_ns_1$. Thus what we called an `embedded $n$-cycle' in a graph $\Lambda$ is a subgraph isomorphic to an $n$-cycle.
Given a simplicial graph $\Lambda$, let $\Lambda^c$ denote the \emph{complement graph}, i.e.\ the graph with the same vertex set as $\Lambda$ and $st$ an edge exactly when there is no edge $st$ in $\Lambda$. Note that if $A_\Lambda$ is $3$-dimensional with Coxeter diagram $\Lambda$, then $\Lambda^c$ has no embedded $4$-cycles (though the converse might not be true).	So Theorem~\ref{thm:3d} follows from the following.

\begin{thm}
	\label{thm:4}
	Let $A_\Lambda$ be an Artin group such 
    that $\Lambda^c$ has no embedded $4$-cycles. Then $\Lambda$ satisfies the girth condition. If $\Lambda$ is non-spherical, then its Artin complex~$\Delta$ is contractible. \end{thm}

\begin{cor}
Let $A_\Lambda$ be an Artin group such that $\Lambda^c$ has no embedded $4$-cycles. Then $A_\Lambda$ satisfies the $K(\pi,1)$ conjecture. In particular, each Artin group of dimension $\le 3$ satisfies the $K(\pi,1)$ conjecture.
\end{cor}

\begin{proof}
This follows from 
Theorems~\ref{thm:4} and~\ref{thm:Kpi1}, by induction on the number of the vertices of $\Lambda$ (recall that all spherical Artin groups satisfy the $K(\pi,1)$ conjecture \cite{deligne}).
\end{proof}

It remains to prove Theorem~\ref{thm:4}. As a preparation, we establish the following graph-theoretic result. 
Below $K_{k,l}$ denotes the complete bipartite graph with the parts of size $k$ and $l$, and $K^-_{k,l}$ denotes $K_{k,l}$ with one edge removed. 
\begin{lem} 
\label{lem:graphs}
Let $\Lambda$ be a simplicial graph with 
at least $5$ vertices, no isolated vertices, no embedded $3$-cycles, and such that $\Lambda^c$ has no embedded $4$-cycles. Then $\Lambda$ equals
the $5$-cycle, 
$K_{2,3}$, 
$K^-_{2,3}$,
 $K_{3,3}$, or 
 $K^-_{3,3}$.
\end{lem}
\begin{proof} Assume first that $\Lambda$ is not bipartite. Let $\gamma$ be the shortest odd embedded cycle in $\Lambda$. If $\gamma$ has length $\geq 7$, then $\gamma^c$ contains an embedded $4$-cycle, which is a contradiction. Consequently, $\gamma$ is an induced $5$-cycle. We will prove $\Lambda=\gamma$. Assume for contradiction that $\Lambda$ has a vertex $s$ outside $\gamma$. Since $\Lambda$ has no embedded $3$-cycles, $s$ is a neighbour of at most two (non-adjacent) vertices of $\gamma$. Then the remaining vertices of $\gamma$ together with $s$ form an embedded $4$-cycle in $\Lambda^c$, which is a contradiction.

Second, assume that $\Lambda$ is bipartite with parts $V,W$. Since $\Lambda^c$ has no embedded $4$-cycles, we have $|V|,|W|\leq 3$. 
It remains to prove that there is at most one edge in $\Lambda^c$ from $V$ to $W$. Suppose that there are two such edges $v_1w_1,v_2w_2$. Then they must intersect, since otherwise $v_1w_1w_2v_2$ would be an embedded $4$-cycle in $\Lambda^c$. Suppose without loss of generality $w_1=w_2$. Then $V=\{v_1,v_2,v_3\}$, since otherwise $w_1$ would be isolated in $\Lambda$. But then $v_1w_1v_2v_3$ is an embedded $4$-cycle in $\Lambda^c$, a contradiction. 
\end{proof}

We will verify Theorem~\ref{thm:4} gradually, starting from the simplest $\Lambda$. We set $\Delta=\Delta_\Lambda$.

\begin{rem} 
\label{rem:easygirth}
By \cite[Lem~6]{AS}, if $\Lambda$ is an edge labelled by $m$, then $\Delta_{\Lambda}$ has girth $\ge 2m$. In particular, $\Lambda$ satisfies the girth condition. 

Thus if $\Lambda$ is a $3$-cycle, then by Lemma~\ref{lem:link} all vertex links of $\Delta$ have girth $\geq 6$. Since $\Delta$ is simply connected (Lemma~\ref{lem:sc}), it satisfies the definition of a systolic complex \cite[page~9]{JS}. In particular, $\Delta$ is contractible \cite[Thm 4.1(1)]{JS} and all $\Delta_{\Lambda, \, st}$ have girth $\geq 6$ \cite[Prop~1.4]{JS}. 

If $\Lambda$ is a length $2$ linear graph, then $\Delta$ is bowtie free by \cite[Lem~4.1]{charney2004deligne} (or Theorem~\ref{thm:4cycle}), when both labels are equal to $3$,  and  \cite[Lem 11.5]{huang2024} otherwise. 
In particular, $\Lambda$ satisfies the girth condition. The contractibility of~$\Delta$ for non-spherical~$\Lambda$ follows from \cite[Thm~B]{CharneyDavis}.

If $\Lambda$ is a $4$-cycle, then by Remark~\ref{rem:adj} the relation $<_x$ on each vertex link of $\Delta$ described in Definition~\ref{def:ha} is a partial order. By the previous paragraph, $<_x$ is bowtie free. Thus, by Lemma~\ref{lem:sc}, $\Delta$ is an $\widetilde A_3$-like complex. By Theorem~\ref{thm:contractibleI}, $\Delta$ is contractible. By Lemma~\ref{lem:4-cycle}, we have that $\Lambda$ satisfies the girth condition.
\end{rem}

\begin{rem} 
\label{rem:rank3CAT0}
Let $\Lambda$ be a length $2$ linear graph with labels $m,n$ such that $m,n\geq 4$ or $m\geq 6$. Then equipping each triangle of $\Delta$ with the Euclidean metric of angles $\frac{\pi}{4},\frac{\pi}{2},\frac{\pi}{4}$, or $\frac{\pi}{6},\frac{\pi}{2},\frac{\pi}{3}$, respectively, $\Delta$ is a $\mathrm{CAT}(0)$ metric space.  Indeed, by Remark~\ref{rem:easygirth} the vertex links of $\Delta$ have girth $\geq 2\pi$. Thus by the Cartan--Hadamard theorem \cite[Thm~4.1(2)]{BridsonHaefliger1999}, we obtain that $\Delta$ is $\mathrm{CAT}(0)$.
\end{rem}

\begin{cor} 
\label{cor:girth8}
Let $\Lambda=stp$ be a length $2$ linear diagram with $m_{st}\geq 4$. Then $\Delta_{stp,\, st}$ has girth $\geq 8$.
\end{cor}
\begin{proof} If $m_{tp}=3$ and $m_{st}=4$ or $5$, then the lemma follows from Theorems~\ref{thm:tripleBn} or~\ref{thm:flag}. If $m_{tp}\geq 4$ or $m_{st}\geq 6$, then, by Remark~\ref{rem:rank3CAT0}, we have that $\Delta_{stp}$ is $\mathrm{CAT}(0)$, with triangles of angles $\frac{\pi}{4},\frac{\pi}{2},\frac{\pi}{4}$ or $\frac{\pi}{6},\frac{\pi}{2},\frac{\pi}{3}$. Then the lemma follows from Lemma~\ref{lem:CAT(0)}(iii) or (ii).
\end{proof}

\begin{lem}
\label{lem:stpr}
Let $\Lambda$ be a length $3$ linear graph. Then $\Delta$ is bowtie free. In particular, $\Lambda$ satisfies the girth condition. If $\Lambda$ is not spherical, then $\Delta$ is contractible.
\end{lem}

\begin{proof}
If $\Lambda$ is spherical, then the lemma follows from Theorem~\ref{thm:4cycle}, so we can assume that $\Lambda$ is not spherical. If all proper induced subdiagrams of $\Lambda$ are spherical, then its consecutive edges have labels $353,434,435,$ or $535$. In the $353$ case, the lemma follows from Theorem~\ref{lem:main example}, Lemma~\ref{lem:upward flag}, and Corollary~\ref{cor:contr}.
In the remaining cases, the lemma follows from Proposition~\ref{prop:434} and Theorem~\ref{thm:contractibleII},  whose hypotheses are satisfied by Theorems~\ref{thm:tripleBn} and~\ref{thm:flag}.

Otherwise, $\Lambda=stpr$ contains a non-spherical subdiagram, say $\Lambda'=stp$.  
We have either $m_{st},m_{tp}\geq 4$, or $m_{st}\geq 6, m_{tp}=3$, or $m_{st}=3, m_{tp}\geq 6$. 
Let $\Delta'=\Delta_{\Lambda,\, \Lambda'}$. By Lemma~\ref{lem:dr} and Remark~\ref{rem:easygirth}, we have that $\Delta$ deformation retracts onto $\Delta'$, and so $\Delta'$ is simply connected.

Assume first $m_{st},m_{tp}\geq 4$. By Corollary~\ref{cor:girth8}, we have that $\Delta_{tpr,\, tp}$ has girth $\geq 8$.
By Lemma~\ref{lem:link}, the complex $\Delta_{tpr,\, tp}$ is the link of a vertex of type~$\hat s$ in $\Delta'=\Delta_{stpr,\, stp}$. The vertex of type $\hat p$ in $\Delta'$ has link $\Delta_{str, \, st}=\Delta_{st}$, which has girth $\geq 8$ by Remark~\ref{rem:easygirth}. Consequently, equipping each triangle of $\Delta'$ with the Euclidean metric of angles $\frac{\pi}{4},\frac{\pi}{2},\frac{\pi}{4}$, the complex $\Delta'$ is a locally $\mathrm{CAT}(0)$ metric space. By the Cartan--Hadamard theorem, 
we obtain that $\Delta'$ is $\mathrm{CAT}(0)$, in particular $\Delta$ is contractible.

By Lemma~\ref{lem:CAT(0)}(i), each induced $4$-cycle in $\Delta'$ has type $\hat s\hat p\hat s\hat p$ and has a common neighbour of type $\hat t$. This shows that there are no bowties without vertices of type~$\hat r$. 

Let $v$ be a vertex of type $\hat r$ and let $C_v=\lk (v,\Delta)\subset \Delta'$. 
We claim that $C_v$ is convex in $\Delta'$. By Lemma~\ref{lem:link}, we have that $C_v$ is isomorphic to $\Delta_{\Lambda'}$, which is connected. Thus to justify the claim we only need to prove that $C_v$ is locally convex \cite[Prop~4.14]{BridsonHaefliger1999}. Suppose first that we have $w\in C_v$ of type $\hat s$ with neighbours $u_1,u_3\in C_v$ connected in $\lk(w,\Delta')\cong \Delta_{tpr,\, tp}$ by a path $\gamma$ of length $d<\pi$, that is, $d\leq \frac{3\pi}{4}$. Note that if one of $u_1,u_3$ has type $\hat p$, then its neighbour on $\gamma$ is also a neighbour of $v$, so we can assume that both $u_1,u_3$ are of type $\hat t$ and $\gamma$ has only one interior vertex $u_2$, which has type~$\hat p$. By Remark~\ref{rem:easygirth} and Corollary~\ref{cor:4cycle} applied to $\Delta_{tpr}$ we have that $u_2\in C_v$, as desired. If $w$ has type~$\hat t$, then the local convexity condition is empty. If $w$ has type $\hat p$, then $C_v$ contains all the neighbours of $w$ by Remark~\ref{rem:adj}, and again there is nothing to prove. This justifies the claim.

Consider now a possible bowtie $vu_1u_2u_3$ with only $v$ of type $\hat r$ and $u_2>u_1,u_3$, where $\hat r>\hat p>\hat t>\hat s$. If $u_2$ lies in the simplicial span of the $\mathrm{CAT}(0)$ convex hull of $u_1,u_3$ in $\Delta'$, then by the convexity of $C_v$ we obtain that $v$ and $u_2$ are neighbours, as desired. If $u_2$ lies outside the span of the convex hull of $u_1,u_3$, then since $\st(u_2,\Delta')$ is convex in $\Delta'$, we obtain that the $\mathrm{CAT}(0)$ geodesic $\alpha=u_1u_3$ lies in the boundary of that star. Thus $\alpha$ consists of edges $u_1w$ and $wu_3$ with $w$ a neighbour of $u_2$ of type~$\hat t$.  Then $w\in C_v$ and so $v,u_2>w>u_1,u_3$, as desired.

Finally, consider a possible bowtie $v_1u_1v_2u_k$ with both $v_i$ of type $\hat r$. Since $C_{v_i}$ are convex, we have that $C=C_{v_1}\cap C_{v_2}$ is convex as well. In particular, $C$ is connected, and we denote by $u_1u_2\cdots u_k$ an edge-path from $u_1$ to $u_k$ in $C$ with the least number of edges. 
If some $u_i$ with $2\leq i\leq k$ has type $\hat p$, then $i=2$, since otherwise the $4$-cycle $v_1u_{i-2}v_2u_i$ in $\lk(u_{i-1},\Delta)$ would contradict Corollary~\ref{cor:4cycle} (because $\lk(u_{i-1},\Delta)$ satisfies the girth condition by Remark~\ref{rem:easygirth}). Analogously we have $i=k-1$ and so $k=3$. Thus $v_1,v_2>u_2>u_1,u_k$ are as desired. If there is no $u_i$ of type $\hat p$, there must be $u_i$ of type~$\hat t$. Then we have $i\leq 2$ since otherwise the $4$-cycle $v_1u_{i-2}v_2u_i$ in $\lk(u_{i-1},\Delta)\cong \Delta_{tpr}$, which is not a bowtie by Remark~\ref{rem:easygirth}, would allow us to replace $u_{i-1}$ by a vertex of type $\hat p$ and to proceed as before. Analogously we have $i\geq k-1$ and so $k=3$ and again $v_1,v_2>u_2>u_1,u_k$.

Consider now the second case, where $m_{st}\geq 6$ and $m_{tp}=3$.  
By Remark~\ref{rem:easygirth}, the vertex links in $\Delta'$ of the vertices of types $\hat s$ and $\hat p$ have girth $6$ and~$2m_{st}\geq 12$, respectively. Consequently, equipping each triangle with the Euclidean metric of angles $\frac{\pi}{3},\frac{\pi}{2},\frac{\pi}{6}$, the complex $\Delta'$ is a $\mathrm{CAT}(0)$ metric space as before. Again, by Lemma~\ref{lem:CAT(0)}(i), there are no bowties without vertices of type~$\hat r$. The proof of the convexity of $C_v$ and that there are  no bowties with vertices of type~$\hat r$ is the same as before.

Finally, the case $m_{st}=3$ and $m_{tp}\geq 6$ 
follows from \cite[Cor~9.14 and Lem~6.14]{huang2023labeled}.
\end{proof}

\begin{cor}
\label{cor:5}
If $\Lambda$ is a $5$-cycle, then $\Delta$ is contractible, and $\Lambda$ satisfies the girth condition.
\end{cor}
\begin{proof}
By Theorem~\ref{thm:contractibleI} and Lemma~\ref{lem:4-cycle}, it is enough to show that $\Delta$ is an $\widetilde A_4$-like complex. By Lemma~\ref{lem:sc}, it suffices to prove
that the vertex links of $\Delta$ satisfy the partial order condition and are bowtie free. By Lemma~\ref{lem:link}, each such link is isomorphic to $\Delta_{\Lambda'}$, where $\Lambda'$ is a linear diagram of length $3$. Thus the partial order condition follows from Remark~\ref{rem:adj}. Bowtie freeness follows from Lemma~\ref{lem:stpr}.
\end{proof}

\begin{defin}
\label{def:subdivision}
Consider a decomposition of the vertex set of $\Lambda$ into a disjoint union $\bigsqcup_i S_i$. Let $\Delta^*$ be the subdivision of $\Delta$ obtained by subdividing each simplex $\sigma$ of type $\widehat S_i$ into a cone over $\partial \sigma$ with apex at the barycentre of $\sigma$, and by subdividing each join $\ast_i\sigma_i$ of simplices of type $\widehat S_i$ into the join of the subdivisions of $\widehat S_i$. We call~$\Delta^*$ the \emph{$\mathcal S$-subdivision}, where the collection $\mathcal S$ is obtained from $\{S_i\}_i$ be removing all the elements of size $1$. For example, if $\mathcal S=\{\{s,t\}\}$, then $\Delta^*$ is obtained from $\Delta$ by subdividing each edge of type $\hat s \hat t$, and each simplex containing it, into two. We denote this, shortly, \emph{$\{s,t\}$-subdivision}.
\end{defin}

\begin{lem}
\label{lem:star}
Let $\Lambda=K_{1,3}$ and let $L\subset \Lambda$ be a length~$2$ linear subdiagram. Then $\Delta_{\Lambda,\, L}$ is bowtie free. In particular, $\Lambda$ satisfies the girth condition. If $\Lambda$ is not spherical, then $\Delta$ is contractible.
\end{lem}
\begin{proof}
The contractibility of $\Delta$ follows from \cite[Thm~10.3]{huang2024} and \cite[Thm~3.1]{godelle2012k}. 
Suppose that $\alpha$ is a bowtie in $\Delta_{\Lambda,\, L}$. Let $L=stp$ and let $s'$ be the remaining vertex of $\Lambda$. We define $\Delta^*$ to be the $\{s,s'\}$-subdivision of $\Delta$. Let $m$ be the type of the new vertices, and identify the types $\hat s$ and $\hat s'$. We order the types so that $\hat s= \hat s'<m<\hat t<\hat p$, which gives rise to a relation on the vertex set of $\Delta^*$ by Definition~\ref{def:order}. By \cite[Prop~11.34]{huang2024}, we have that $\Delta^*$ is bowtie free. Thus there is a vertex $x$ of $\Delta^*$ that is a common neighbour of $\alpha$. But $x$ cannot have type $m$ since then $\alpha$ would have two equal vertices of type $\hat s$. Thus $x$ belongs to $\Delta_{\Lambda,\, L}$, as desired.
\end{proof}

By Lemma~\ref{lem:star} and \cite[Lem 6.14 and Prop 6.17]{huang2023labeled}, we have the following.
\begin{cor}
\label{cor:star}
Let $\Lambda=K_{1,3}$ with parts $\{s_1,s_2,s_3\}$ and $\{t\}$. Then each induced cycle of type $\hat s_1\hat s_2\hat s_1\hat s_2 $ or $\hat s_1\hat s_2 \hat s_1\hat s_3$ in $\Delta$ has a common neighbour of type $\hat t$.
\end{cor}

\begin{lem}
\label{lem:Kmn}
If $\Lambda=K_{k,l}$ with $k,l\geq 2$, then $\Delta$ is contractible and $\Lambda$ satisfies the girth condition. 
\end{lem}
\begin{proof}
Let $\{s_1,\ldots, s_k\},\{t_1,\ldots, t_l\}$ be the parts of $\Lambda$. To prove the girth condition for the edge $s_1t_1$, consider the $\left\{\left\{s_2,\ldots s_k\right\},\left\{t_2,\ldots,t_l\right\}\right\}$-subdivision $\Delta^*$ of $\Delta$. Let $m,m'$ be the types of the barycentres of the simplices of types $s_2\cdots s_k,t_2\cdots t_l$. By \cite[Lem~11.10]{huang2024} (which relies on \cite[Theorem 1.4]{huang2024Dn}), the subcomplex $\Delta^*_{s_1t_1mm'}$ of $\Delta^*$ spanned on the vertices of types $\hat s_1, \hat t_1, m,$ and $m'$ is an $\widetilde A_3$-like complex. Thus, by Lemma~\ref{lem:4-cycle}, the complex~$\Delta^*$ has no cycles of type $\hat s_1\hat t_1\hat s_1\hat t_1$.
The contractibility of~$\Delta$ can be deduced from \cite[Lem 11.11]{huang2024}.
\end{proof}

\begin{lem}
\label{lem:Kmn2} Let $\Lambda=K_{2,3}$ with parts $\{s_1,s_2,s_3\}$ and $\{t_1,t_2\}$. Then each induced cycle of type $\hat s_1\hat s_2\hat s_1\hat s_2$ in $\Delta$ has common neighbours of type $\hat t_1$ and $\hat t_2$.
\end{lem}

\begin{proof} 
Let $\Delta^*$ be the $\{s_2,s_3\}$-subdivision of $\Delta$. Let $m$ be the type of the new vertices. By \cite[Lem 11.10]{huang2024}, the subcomplex $\Delta^*_{s_1t_1mt_2}$ of $\Delta^*$ spanned on the vertices of types $\hat s_1, \hat t_1, m,\hat t_2,$ is an $\widetilde A_3$-like complex. We fix any of the two cyclic orders on $s_1t_1mt_2$ to be able to discuss meets and joins in the links.

Let $v$ be a vertex of type $\hat s_2$ and let $C_v=\lk(v,\Delta^*)\subset \Delta^*_{s_1t_1mt_2}$. We claim that $C_v$ is B-convex in $\Delta^*_{s_1t_1mt_2}$. 
Note that $C_v$ is connected, since it is isomorphic to $\lk(v,\Delta)$, which, by Lemma~\ref{lem:link}, is in turn isomorphic to $\Delta_{s_1t_1s_3t_2}$. 
For the local B-convexity, let $w\in C_v$, and let $u_1,u_2\in C_v$ be neighbours of $w$. Assume that the meet $u$ of $u_1$ and $u_2$ in $\lk(w,\Delta^*_{s_1t_1mt_2})^0$ exists and is distinct from $u_1,u_2$.
We need to show $u\in C_v$, so we can assume that none of $w,u_1,u_2$ has type $m$. If $u_1$ or $u_2$ has type $\hat t_i$, and $u$ is not of type $m$, then this follows immediately from applying Corollary~\ref{cor:4cycle} to the 4-cycle $vu_1uu_2$ in $\lk(w,\Delta)$, which satisfies the girth condition by Lemmas~\ref{lem:star} and~\ref{lem:Kmn}. If $u_1$ or $u_2$ has type $\hat t_i$, and $u$ is the midpoint of an edge $u^+u^-$ of $\Delta$ of type $\hat s_2\hat s_3$, then applying as before Corollary~\ref{cor:4cycle} to the 4-cycle $vu_1u^+u_2$, we obtain $u^+=v$, and so $u$ is a neighbour of $v$.
 It remains to assume that $u_1$ and $u_2$ are of type $s_1$. Again by Corollary~\ref{cor:4cycle}, we can assume that $u$ is not of type~$t_i$, so it is the midpoint of an edge $u^+u^-$ in $\Delta$ of type $\hat s_2\hat s_3$. Applying Corollary~\ref{cor:star} to $vu_1u^+u_2$ in $\lk(w,\Delta)$, we obtain that $u^+=v$ as before  
 or there is a common neighbour of type $\hat t_i$ of $u_1,u_2$ in $\lk(w,\Delta)$, contradicting the assumption that $u=u_1\wedge u_2$ in $\lk(w,\Delta^*_{s_1t_1mt_2})^0$.

Let $u_1v_1u_kv_2$ be a cycle of type $\hat s_1\hat s_2\hat s_1\hat s_2$ in $\Delta$. By Corollary~\ref{cor:convex}, we have that $C=C_{v_1}\cap C_{v_2}$ is connected. Let $u_1\cdots u_k$ be an edge-path in $C$ from $u_1$ to~$u_k$ with the least number of edges. None of the $u_i$ can have type $m$, since $v_1\neq v_2$. Thus one of the $u_i$ has type $t_1$ or $t_2$, and then by the girth condition in the links of $u_{i-1},u_{i+1},$ (see Lemmas~\ref{lem:star} and~\ref{lem:Kmn}) 
and Corollary~\ref{cor:4cycle}, we have that $u_i$ is a neighbour of $u_1$ and $u_k$. Then the lemma follows from Corollary~\ref{cor:star} applied to the link of~$u_i$.
\end{proof}

\begin{prop}
\label{prop:23}
If $\Lambda=K^-_{2,3}$, then $\Delta$ is contractible and $\Lambda$ satisfies the girth condition.
\end{prop}

\begin{proof}
Denote the parts to be $\{s,p\},\{t_1,t_2,r\}$ with missing edge $sr$. Let $\Delta'=\Delta_{\Lambda,\, \Lambda'}$ with $\Lambda'=st_1pt_2$. By Lemma~\ref{lem:dr} and Remark~\ref{rem:easygirth}, we have that $\Delta$ deformation retracts onto $\Delta'$, and so the latter is simply connected.  
Note that $\Delta'$ is an $\widetilde A_3$-like complex.
In other words, the links $\Delta_{t_1st_2r,\, t_1st_2}=\Delta_{t_1st_2}, \Delta_{t_1pt_2r,\, t_1pt_2},$ and $\Delta_{st_1pr,\, st_1p}$ satisfy the partial order condition and are bowtie free. The partial order condition follows from Remark~\ref{rem:adj}. The bowtie freeness of the first link follows from Remark~\ref{rem:easygirth}. For the middle one, this is Lemma~\ref{lem:star}. For the last one, this follows from Lemma~\ref{lem:stpr}. By Theorem~\ref{thm:contractibleI}, we have that $\Delta'$ is contractible, and so is $\Delta$. We fix any of the two cyclic orders on $st_1pt_2$ to be able to discuss meets and joins in the links.

By Lemma~\ref{lem:4-cycle}, each induced $4$-cycle in $\Delta'$ is contained in the link of a vertex, and so the girth condition for the edges in $\Lambda'$ follows from the girth condition in Lemmas~\ref{lem:stpr} and~\ref{lem:star}. 

It remains to consider a $4$-cycle with vertices of types $\hat p$ and $\hat r$. First we justify the following.

\smallskip
\noindent \textbf{Claim.} \emph{For $v$ of type $\hat r$, the subcomplex $C_v=\lk(v,\Delta)\subset \Delta'$ is B-convex.}
\smallskip

Note that $C_v$ is isomorphic to $\Delta_{\Lambda'}$ and so it is connected. For the local B-convexity, let $w\in C_v$, and let $u_1,u_2\in C_v$ be neighbours of $w$. Assume that $u=u_1\wedge u_2$ exists in $\lk(w,\Delta')^0$ and is distinct from $u_1,u_2$. We need to show $u\in C_v$.
If $w$ has type $\hat p$, then this is immediate, so without loss of generality we only need to consider the cases where $w$ has type $\hat t_1$ and $\hat s$. 

In the first case, the link of $w$ is isomorphic to $\Delta_{st_2pr}$. If the $4$-cycle $u_1uu_2v$ contains an edge whose type lies on the path $\hat s\hat t_2 \hat p \hat r$, then $uv$ is an edge by the girth condition in Lemma~\ref{lem:stpr} and Corollary~\ref{cor:4cycle}. Otherwise, the type of the cycle is $\hat s \hat p\hat s \hat r$. Since the link of $w$ is bowtie free by Lemma~\ref{lem:stpr}, we have that $u$ is a neighbour of~$v$, or there is a common neighbour of type~$\hat t_2$ of all $w,u_1,u,u_2,v$. But then it is this vertex and not $u$ that is the meet of $u_1,u_2$, contradiction. 

Now consider the case, where $w$ has type $\hat s$. If the $4$-cycle $u_1uu_2v$ contains an edge whose type lies in the Coxeter diagram of  $t_1t_2pr$, then $uv$ is an edge by the girth condition in Lemma~\ref{lem:star} and Corollary~\ref{cor:4cycle}. Otherwise, without loss of generality, the type of the cycle is $\hat t_1 \hat t_2\hat t_1 \hat r$. By Corollary~\ref{cor:star},
we have that $u$ is a neighbour of~$v$, or
there is a common neighbour of type $\hat p$ of all $w,u_1,v,u_2,u$, which contradicts the assumption that $u$ is the meet of $u_1,u_2$. This ends the proof of the claim. 

\smallskip

Let $v_1u_1v_2u_k$ be a cycle with both $v_i$ of type $\hat r$ and $u_1,u_k$ of type $\hat p$. By the B-convexity of $C_{v_i}$ and Corollary~\ref{cor:convex}, we have that $C=C_{v_1}\cap C_{v_2}$ is connected. Let $u_1u_2\cdots u_k$ be an edge-path in $C$ from $u_1$ to $u_k$ with the least number of edges. By the girth condition in the link of $u_2$ (see Lemmas~\ref{lem:stpr} and~\ref{lem:star}), and Corollary~\ref{cor:4cycle} applied to the $4$-cycle $u_1v_1u_3v_2$, we obtain that $u_1$ is a neighbour of $u_3$, contradiction.
\end{proof}

\begin{prop}
\label{prop:33}
If $\Lambda=K^-_{3,3}$, then $\Delta$ is contractible and $\Lambda$ satisfies the girth condition.
\end{prop}
\begin{proof}
Denote the parts of $\Lambda$ by $\{s_1,s_2,s_3\},\{t_1,t_2,t_3\}$ with the missing edge $s_3t_3$. Consider the $\{t_1,t_2\}$-subdivision $\Delta^*$ of $\Delta$. Let $m$ be the type of the new vertices. We claim that the subcomplex $\Delta^*_{s_1t_3s_2m}$ of $\Delta^*$ spanned on the vertices of types $\hat s_1, \hat t_3, \hat s_2,$ and $m$ is an $\widetilde A_3$-like complex. The partial order condition follows from Lemma~\ref{lem:transitive}. 
The bowtie freeness of the link of a vertex of type $m$ in $\Delta^*_{s_1t_3s_2m}$ follows from Remark~\ref{rem:easygirth}, since such link is isomorphic to $\Delta_{s_1t_3s_2}$. To obtain the bowtie freeness of the link $\lk(z,\Delta^*_{s_1t_3s_2m})$ of a vertex $z$ of type $\hat s_1$ (or $\hat s_2$), for all the possible bowties except for the type $\hat t_3 m \hat t_3 m$, it suffices to use Proposition~\ref{prop:23} and Corollary~\ref{cor:4cycle}. 

Consider now a possible bowtie $vev'e'$ of type $\hat t_3 m \hat t_3 m$, where $e,e'$ are midpoints of edges $uw,u'w'$ of type $\hat t_1\hat t_2$. Our goal is to find a vertex of type $\hat s_2$ in  $L=\lk(z,\Delta)$ that is a neighbour of all $z,u,w,u',w',v,v'$.  Let $$\Delta'=\lk(z, \Delta_{\Lambda, \, s_1s_2t_1s_3t_2})\cong \Delta_{t_3s_2t_1s_3t_2,\, s_2t_1s_3t_2},$$
which was studied in Proposition~\ref{prop:23}. Let $C_v=\lk(v, \Delta)\cap \Delta'$.
By the Claim in the proof of Proposition~\ref{prop:23} and Corollary~\ref{cor:convex}, we have that $C_v\cap C_{v'}\subset \Delta'$ is connected. Let $\alpha$ be an edge-path in $C_v\cap C_{v'}$ from $uw$ to $u'w'$ with the least number of edges. 

If $\alpha$ is a single vertex, say $u=u'$, then by Lemma~\ref{lem:stpr} applied to $\lk(u,L)$, the cycle $vwv'w'$ has a common neighbour of type $\hat s_2$ that is also a neighbour of $z$ and $u=u'$, as desired. 
If $\alpha$ is a single edge, say $uw'$, then analogously the cycle $vwv'w'$ has a common neighbour $y$ of type $\hat s_2$ that is also a neighbour of $z$ and $u$. By Corollary~\ref{cor:4cycle}, applied to the cycle $vyv'u'$ in $\lk(w',L)$,
we obtain that $y$ is also a neighbour of $u'$, as desired.
We can now assume that $\alpha$ has at least two edges.

If $\alpha$ contains a vertex of type $\hat s_2$, then by Corollary~\ref{cor:4cycle} it is a neighbour of $u,u',w,w'$, as desired. Otherwise, if 
$\alpha$ contains a subpath $xyx'$ with $y$ of type $\hat s_3$, then applying Corollary~\ref{cor:star} to $xvx'v'$ in $\lk(y,L)$, we can replace~$y$ by a vertex of type $\hat s_2$ and we conclude as before. Otherwise, let $x_0x_1x_2$ be any subpath of $\alpha$.
By Lemma~\ref{lem:stpr} applied to  $\lk(x_1,L),$ which is isomorphic to $\Delta_{t_3s_2t_is_3}$, for~$i=1$ or~$2$, we have that there is a common neighbour of type $\hat s_2$ of $x_0vx_2v'$ and we conclude as before. 

To obtain the bowtie freeness of  the link of a vertex of type $\hat t_3$, we apply Lemmas~\ref{lem:Kmn}, \ref{lem:Kmn2}, and Corollaries~\ref{cor:star} and~\ref{cor:4cycle}. This finishes the proof of the claim.

 Note that $\Delta^*_{s_1t_3s_2m}$ is obtained from $\Delta_{\Lambda, \, s_1t_3s_2t_1t_2}$ by removing the disjoint stars in~$\Delta^*$ of the vertices of type $\hat t_1$ and $\hat t_2$. These stars are isomorphic to the stars of the vertices of type $\hat t_1$ and $\hat t_2$ in $\Delta_{\Lambda,\, s_1t_3s_2t_1t_2}$, and hence they are cones over links isomorphic to $\Delta'$ whose contracibility we established in the proof of Proposition~\ref{prop:23}. 
Consequently, $\Delta^*_{s_1t_3s_2m}$ is a deformation retract of $\Delta_{\Lambda, \, s_1t_3s_2t_1t_2}$, and thus of $\Delta$ by Lemma~\ref{lem:dr}. 
In particular, since $\Delta$ is simply connected, this completes the proof of the claim that $\Delta^*_{s_1t_3s_2m}$ is an $\widetilde A_3$-like complex. 

By Theorem~\ref{thm:contractibleI}, we have that $\Delta^*_{s_1t_3s_2m}$ is contractible, and so is $\Delta$.

The girth condition for the edges $s_1t_3$ and $s_2t_3$ follows from Lemma~\ref{lem:4-cycle} applied to $\Delta^*_{s_1t_3s_2m}$. 
Using a symmetry of $\Lambda$, we also obtain the girth condition for the edges $t_1s_3$ and $t_2s_3$. 

Using a symmetry again, it remains to verify the girth condition for  $s_1t_2$. Let $v$ be a vertex of type $\hat t_2$ and let $D_v=\lk(v,\Delta^*)\cap \Delta^*_{s_1t_3s_2m}$. We claim that $D_v$ is B-convex.

Note that $D_v$ is isomorphic to $\Delta_{t_1s_1t_3s_2s_3,\,t_1s_1t_3s_2}$ 
and so it is connected. 
For the local B-convexity, all the cases follow from Corollary~\ref{cor:4cycle} and the girth conditions for smaller Coxeter diagrams, except for the case where $w$ has type $\hat s_1$ or~$\hat s_2$, say~$\hat s_1$, with $u_1uu_2$ of type $\hat t_3m\hat t_3$, where $u=u_1\wedge u_2$ in $\lk(w,\Delta^*_{s_1t_3s_2m})^0$. Then, as in the verification of the bowtie freeness in the  second, third and fourth paragraph of the proof, there is a vertex of type~$\hat s_2$ in $\lk(w,\Delta)$ that is a neighbour of $u_1$ and $u_2,$ which contradicts the assumption that $u$ is the meet of $u_1,u_2$. This justifies the claim.

Consider an induced $4$-cycle $v_1u_1v_2u_k$ with $v_i$ of type $\hat t_2$ and $u_1,u_k$ of type $\hat s_1$. Since $D_{v_i}$ are B-convex, by Corollary~\ref{cor:convex} we have that $D=D_{v_1}\cap D_{v_2}$ is connected. Let $\alpha=u_1u_2\cdots u_k$ be an edge-path from $u_1$ to $u_k$ in $D$ with the least number of edges. None of the $u_i$ has type $m$, since otherwise $v_1=v_2$. Since $\alpha$ is not a single vertex, we obtain a contradiction by applying Proposition~\ref{prop:23} or 
Lemma~\ref{lem:Kmn} 
to $\lk(u_2,\Delta)$, and Corollary~\ref{cor:4cycle}.

\end{proof}

\begin{prop}
\label{prop:triangle}
Let $\mathcal C$ be a class of Coxeter diagrams closed under taking induced subdiagrams. 
Suppose that we have $\mathcal C_1\subset \mathcal C$ such that each diagram in $\mathcal C-\mathcal C_1$ contains a triangle.
Then
\begin{enumerate}
	\item if each diagram in $\mathcal C_1$ satisfies the girth condition, then each diagram in $\mathcal C$ satisfies the girth condition, and
	\item if in addition for each nonspherical $\Lambda_1\in \mathcal C_1$ the Artin complex $\Delta_{\Lambda_1}$ is contractible, 
    then for each nonspherical $\Lambda\in \mathcal C$ the Artin complex $\Delta_\Lambda$ is contractible. 
    In particular, each diagram in $\mathcal C$ satisfies the $K(\pi,1)$ conjecture.
\end{enumerate}
\end{prop}

\begin{proof}
We prove assertion (1) by induction on the number of the vertices of $\Lambda\in \mathcal C$. We can assume that $\Lambda$ contains a triangle $stp$. 
By 
the inductive hypothesis, 
the vertex links of $\Delta'=\Delta_{\Lambda, \, stp}$ have girth $\geq 6$. We have that $\Delta'$ is simply connected by Lemma~\ref{lem:sc}, and so it is systolic. In particular, $\Delta_{\Lambda, \, st}, \Delta_{\Lambda,\,  tp}, \Delta_{\Lambda, \, sp}$ have girth $\geq 6$ \cite[Prop~1.4]{JS}, which verifies part of assertion (1) for $\Lambda$. Furthermore, equipping each triangle with the Euclidean metric of an equilateral triangle, $\Delta'$ is $\mathrm{CAT}(0)$.

Consider now an edge $pq$ of $\Lambda$ with $q\neq s,t$. We will justify that $\Delta_{\Lambda, \, pq}$ has girth $\geq 6$. Let $v$ be a vertex of $\Delta$ of type $\hat q$. 

We claim that $C_v=\lk(v,\Delta)\cap \Delta'$ is convex in $\Delta'$ with respect to the $\mathrm{CAT}(0)$ metric. 
We have that $C_v$ is isomorphic to  $\Delta_{\Lambda \setminus \{q\},\, stp}$, which is connected. Thus to justify the claim we only need to prove that $C_v$ is locally convex \cite[Prop~4.14]{BridsonHaefliger1999}. Suppose that we have $w\in C_v$ with neighbours $u_1,u_2\in C_v$ and $u\in \Delta'$ such that $u_1u,uu_2$ are edges but $u_1$ and $u_2$ are not neighbours. Then $vu_1uu_2$ is a $4$-cycle in the link of $w$, to which we can apply Corollary~\ref{cor:4cycle} 
by the inductive hypothesis. Thus $u\in C_v$, which justifies the claim.

Suppose for contradiction that $v_1u_1v_2u_k$ is a $4$-cycle in $\Delta_{\Lambda,\, pq}$ with both $v_i$ of type~$\hat q$. Then $u_1,u_k\in C_{v_1}\cap C_{v_2}$, which is convex in $\Delta'$. In particular, $C_{v_1}\cap C_{v_2}$ is connected. Consider an edge-path $u_1u_2\cdots u_k$ from $u_1$ to $u_k$ in $C_{v_1}\cap C_{v_2}$ with the least number of edges. Note that we have $k\geq 3$. The $4$-cycle $u_1v_1u_3v_2$ in the link of $u_2$ violates Corollary~\ref{cor:4cycle} by 
the inductive hypothesis. 

Consider now an edge $qr$ of $\Lambda$ with $q,r\notin \{s,t,p\}$. We will justify that $\Delta_{\Lambda,\, qr}$ has girth $\geq 6$. Suppose for contradiction that $v_1u_1v_2u_2$ is a $4$-cycle in $\Delta_{\Lambda,\, qr}$ with both~$v_i$ of type $\hat q$ and both $u_i$ of type $\hat r$. Consider disc diagrams $D\to \Delta'$ with boundary cycle $\alpha_0\alpha_1\alpha_2\alpha_3$ such that $\alpha_0\subset C_{v_1},\alpha_1\subset C_{u_1},\alpha_2\subset C_{v_2},\alpha_3\subset C_{u_2}$. Choose $D$ of minimal area, and among such $D$, choose $D$ with minimal perimeter.  Then the boundary cycle of $D$ is a concatenation of paths $I_i$ embedded in $D$ that are the domains of~$\alpha_i$, and the intersections $x_i=I_i\cap I_{i+1}$ (mod $4$) are single vertices by the minimality assumption. We can assume that $D$ is not a single vertex, since then we would obtain a contradiction with 
the inductive hypothesis. In particular, two consecutive~$I_i$ cannot be trivial. Thus, up to a symmetry, we have one of the following:
\begin{itemize}
\item all $x_i$ are distinct, and so all $I_i$ are nontrivial,
\item $x_0=x_1,x_2,x_3$ are distinct, and so only $I_1$ is trivial,
\item $x_0=x_1\neq x_2=x_3$, and so only $I_1,I_3$ are trivial.
\end{itemize}
The $x_i$ equal to $x_j$ for $i\neq j$ are called \emph{singular}.
We apply Theorem~\ref{thm:GB} to $D$. Since $\Delta'$ is $\mathrm{CAT}(0)$, the curvature at each interior vertex of $D$ is non-positive. Consider now an interior vertex $y$ of one of the $I_i$, with $\alpha_i\subset C_v$. If the curvature at $y$ was positive, then $y$ would be contained in exactly one or two triangles of $D$. By the convexity of $C_v$, the images in $\Delta'$ of these triangles would be contained in $C_v$. Consequently, we could alter $\alpha_i$ be removing these triangles from $D$, which would contradict the minimality of the area of $D$. 
Thus the curvature is non-positive also 
at each interior vertex of $I_i$. Consequently, the curvature can be positive only at the~$x_i$, 
and it then equals $\frac{\pi}{3},\frac{2\pi}{3},$ or $\pi$. Since their sum equals~$2\pi$, there must be
\begin{enumerate}[(i)]
\item a non-sigular $x_i$ with curvature $\geq \frac{2\pi}{3}$, or 
\item a singular $x_i$ with curvature $\pi$. 
\end{enumerate}
In case (i), $x_i$ is contained in only one triangle od $D$. Thus if, say, $x_i=x_2$, then we have a $4$-cycle in the link of $x_2$ containing $v_2u_2$. By the inductive hypothesis and Corollary~\ref{cor:4cycle}, this $4$-cycle has a diagonal, which contradicts the minimality of $D$. In case (ii), $x_i$ is not contained in any trangle od $D$. Thus if, say, $x_i=x_0=x_1$, then we have a $4$-cycle in the link of $x_0$ containing $v_1u_1v_2$ and contradicting the minimal perimeter assumption on $D$ in view of Corollary~\ref{cor:4cycle} and 
the inductive hypothesis. This finishes the proof of assertion~(1).

For assertion (2), we first show by induction on the number of the vertices of $\Lambda\in \mathcal C$ that $\Delta_\Lambda$ is contractible whenever $\Lambda$ is not spherical. Indeed, we can assume that $\Lambda$ contains a triangle $stp$.  Let $\Delta'$ be as in the proof of assertion (1). We proved that $\Delta'$ is $\mathrm{CAT}(0)$, and so it is contractible. 
By Lemma~\ref{lem:dr}, and the  
inductive hypothesis, we have that $\Delta$ deformation retracts onto $\Delta'$, and so $\Delta$ is contractible. The last part of  assertion (2) follows from 
Theorem~\ref{thm:Kpi1}.
\end{proof}

\begin{proof}[Proof of Theorem~\ref{thm:4}]
If $\Lambda$ has multiple connected components, then the associated Artin complex is a join of several smaller Artin complexes, one for each connected component of $\Lambda$. Thus it suffices to consider the case where $\Lambda$ is connected.
If $|S|\leq 4$, then the theorem follows from Remark~\ref{rem:easygirth}, and Lemmas~\ref{lem:stpr} and \ref{lem:star}. Otherwise, if $\Lambda$ is complete bipartite, then the theorem follows from Lemma~\ref{lem:Kmn}.
Consequently, the theorem follows from Lemma~\ref{lem:graphs}, Corollary~\ref{cor:5}, Propositions~\ref{prop:23}, \ref{prop:33}, and~\ref{prop:triangle}.
\end{proof}

By Theorem~\ref{thm:Kpi1}, we have the following consequences.

\begin{cor} 
\label{cor:reduction}
Suppose that all non-spherical $\Lambda$ without triangles satisfy the girth condition and have contractible $\Delta_\Lambda$. Then all Artin groups satisfy the $K(\pi,1)$ conjecture.

\end{cor}

\begin{theorem}
	\label{thm:general}
Let $\mathcal C$  be a class of Coxeter diagrams closed under taking induced subdiagrams. Suppose that each $\Lambda\in \mathcal C$ not containing a triangle satisfies at least one of the following conditions:
\begin{enumerate}
	\item $A_\Lambda$ is spherical, or more generally $\Lambda$ satisfies the assumption of \cite[Thm~1.1]{huang2023labeled},
	\item $\Lambda^c$ does not contain embedded $4$-cycles,
	\item $\Lambda$ is locally reducible.
\end{enumerate} 
Then each $A_\Lambda$ with $\Lambda\in \mathcal C$ satisfies the $K(\pi,1)$ conjecture.
\end{theorem}

\begin{proof}
By Theorem~\ref{thm:Kpi1}, and Proposition~\ref{prop:triangle}, it suffices to show that each $\Lambda$ in one of the above classes satisfies the girth condition and, if it is not spherical, then $\Delta_\Lambda$ is contractible. For class (2), this is Theorem~\ref{thm:4}. For class (1), this follows from \cite[Prop~9.11 and~9.12]{huang2023labeled}. For class (3), this follows from \cite[Cor~9.14]{huang2023labeled} (as stated, this result only treats the case where $\Lambda$ is a locally reducible tree, but the same argument works for any locally reducible diagram, and it also gives the contractibility of $\Delta_\Lambda$).
\end{proof}

\bibliographystyle{alpha}
\bibliography{mybib}

\end{document}